\documentclass[11pt,reqno]{amsart}

\usepackage[T1]{fontenc}
\usepackage{lmodern}
\usepackage{amsmath,amssymb,amsthm,mathtools}
\usepackage{mathrsfs}
\usepackage{enumitem}
\usepackage{microtype}
\usepackage{tikz-cd}
\usepackage[colorlinks=true,linkcolor=blue,citecolor=blue,urlcolor=blue,pagebackref=true]{hyperref}

\newtheorem{theorem}{Theorem}[section]
\newtheorem{proposition}[theorem]{Proposition}
\newtheorem{lemma}[theorem]{Lemma}
\newtheorem{corollary}[theorem]{Corollary}
\theoremstyle{definition}
\newtheorem{definition}[theorem]{Definition}
\newtheorem{hypothesis}[theorem]{Hypothesis}
\newtheorem{example}[theorem]{Example}

\theoremstyle{remark}
\newtheorem{remark}[theorem]{Remark}

\newcommand{\kk}{\Bbbk}

\newcommand{\cC}{\mathcal C}

\newcommand{\Hom}{\operatorname{Hom}}
\newcommand{\End}{\operatorname{End}}
\newcommand{\Ext}{\operatorname{Ext}}

\newcommand{\cR}{\mathcal R}

\DeclareMathOperator{\Frac}{Frac}
\DeclareMathOperator{\Sym}{Sym}
\DeclareMathOperator{\Mat}{Mat}

\DeclareMathOperator{\Stab}{Stab}

\DeclareMathOperator{\rad}{rad}

\numberwithin{equation}{section}
\allowdisplaybreaks[2]
\hypersetup{pdfauthor={Zhi-Wei Li},pdftitle={Quiver Demazure algebras I: Residue descent, finite quotients, and PBW bases}}
\title[Quiver Demazure algebras I]{Quiver Demazure algebras I:\protect\\ Residue descent, finite quotients, and PBW bases}

\author{Zhi-Wei Li}
\address{School of Mathematics and Statistics, Jiangsu Normal University,
Xuzhou 221116, Jiangsu, P. R. China}
\email{zhiweili@jsnu.edu.cn}
\date{}

\begin{document}
\begin{abstract}
We study quiver Demazure algebras through residue covers and
root-power quotients.  The cover determined by the zero-root
subsystem identifies the additive algebra with a finite fixed
algebra of Sauter's generalized quiver Hecke construction.
Her basis theorem then gives additive PBW in every finite Weyl
type, for arbitrary admissible residue data and uniform quivers.
A sequence-cover argument gives global multiplicative PBW in
simply-laced type.  Under a directed-path condition, root-power
quotients are finite dimensional and the formal exponential
identifies their additive and multiplicative forms.
For arrow-free quivers, we reduce mixed-residue quotients to
matrix algebras over finite skew group rings, obtain sharp
nonvanishing criteria from coinvariant geometry, and determine
their simple modules.  In a mixed type-$D_4$ family we compute
an explicit coefficient presentation, dimensions, Loewy lengths,
and Cartan matrices.  We also specify the direct and finite-cover
comparisons with Liu's algebras.

\end{abstract}
\maketitle
\clearpage
\setcounter{tocdepth}{2}
{\small\tableofcontents}
\clearpage

\section{Introduction}\label{sec:introduction}

Quiver Demazure algebras combine divided differences at zero root
residue with weighted Weyl reflections at nonzero residue.  A
finite Weyl group, a uniform residue quiver, and a coordinate law
are specified independently.  This flexibility raises two
questions: which residue data have an integral reduced-product
basis, and when do their root-power quotients admit a usable
description of their representations?

We answer the basis question in all finite Weyl types for additive
coordinates, and in simply-laced finite types for uncompleted
multiplicative coordinates.  The additive result is obtained by
identifying a canonical finite cover with a generalized quiver
Hecke algebra of Sauter and descending her PBW basis.  This places the additive basis theorem within an existing
geometric construction.  We then use the residue data to develop
explicit applications to mixed-residue root-power quotients.  For
arrow-free quivers we give a matrix description, a sharp
nonvanishing criterion, and the simple modules.  A type-$D_4$
family has explicit dimensions, radical lengths, and Cartan
matrices.

\vskip5pt
\noindent\textbf{The operator construction and PBW.}
Let $\Phi$ be a finite reduced crystallographic root system,
$Q$ its root lattice, and $W=\langle s_i:i\in I\rangle$ its
Weyl group.  Let $\mathbb I$ be an abelian group.  A uniform
residue quiver has $d(a-b)$ arrows from $a$ to $b$, for
$d:\mathbb I\setminus\{0\}\to\mathbb Z_{\ge0}$.
An admissible residue datum is a finite transitive $W$-set
$\cC$ with an equivariant map
$\rho:\cC\to\Hom_{\mathbb Z}(Q,\mathbb I)$, such that
$m_i(c):=\rho_c(\alpha_i)=0$ implies $s_ic=c$.
Injectivity of $\rho$ is not required.

For a coefficient domain $R$ with fraction field $K$ and root
coordinates $y_\alpha$, put $D_R=\bigoplus_c R e_c$ and
$D_K=\bigoplus_c K e_c$.  The QDA $\cR$ is generated inside
$D_K\rtimes W$ by $D_R$ and
\[
 \psi_i e_c=
 \begin{cases}
 y_{\alpha_i}^{-1}([s_i]-1)e_c,&m_i(c)=0,\\
 y_{-\alpha_i}^{d(m_i(c))}[s_i]e_c,&m_i(c)\ne0.
 \end{cases}
\]
The coordinate assumptions in~\eqref{rc:eq:root-data} make
$D_R$ a faithful representation.
Put $V=\mathbb C\otimes_{\mathbb Z}Q$ and
$P=\Sym_{\mathbb C}(V)$.  The standard choices are
\[
 (R,y_\alpha)=(P,x_\alpha)
 \qquad\text{and}\qquad
 (R,y_\alpha)=(\mathbb C[Q],1-X_\alpha).
\]

\begin{theorem}[PBW in the standard coordinates]\label{intro:thm:pbw}
For arbitrary admissible residue data and uniform residue quivers,
\[
 \cR=\bigoplus_{c\in\cC}\bigoplus_{w\in W}R\psi_w e_c
\]
holds for any choices of reduced words in the following cases:
additive coordinates in every finite Weyl type; uncompleted
multiplicative coordinates in finite simply-laced type; and
completed additive or multiplicative coordinates in every finite
Weyl type.  The coefficient, idempotent, mixed, quadratic, and
integral rank-two braid relations give a presentation in each case.
\end{theorem}

\begin{proof}
See Theorems~\ref{des:thm:additive-all} and~\ref{des:thm:ADE}
and Corollary~\ref{des:cor:completed-all}.
\end{proof}

For general root coordinates, Theorem~\ref{rc:thm:pbw} gives
an explicit rank-two integrality criterion.  Preservation of the
coefficient ring by the individual operators is not alone a
proof of that criterion.  In particular, our completed
multiplicative theorem does not assert global multiplicative PBW
for all non-simply-laced types.

\vskip5pt
\noindent\textbf{The cover and coefficient-generated quotients.}
Fix $c_0\in\cC$, put $H=\Stab_W(c_0)$, and let
\[
 \Psi=\{\alpha\in\Phi:\rho_{c_0}(\alpha)=0\},\qquad
 K_0=W_\Psi.
\]
We prove $K_0\triangleleft H$ and construct the canonical cover
$W/K_0\to W/H=\cC$.  Its deck group is $F=H/K_0$.
The cover has a useful exact property: a simple reflection fixes
an index if and only if its root residue is zero.  This removes
the fixed-index obstruction to matching Sauter's operators.

\begin{theorem}[Descent of the algebra and its coefficient quotients]
\label{intro:thm:descent}
Let $\widetilde{\cR}$ be the QDA on any finite free residue
cover with deck group $F$.  Pullback identifies
$\cR=\widetilde{\cR}^{F}$, and PBW holds upstairs if and
only if it holds downstairs.  If $I$ is generated by a family
in $D_R$ and $\widetilde I$ by its pullback, then
\[
 \widetilde I\cap\cR=I,\qquad
 \cR/I\cong(\widetilde{\cR}/\widetilde I)^F,
 \qquad
 (\widetilde{\cR}/\widetilde I)\rtimes F
                         \cong\Mat_{|F|}(\cR/I).
\]
The matrix isomorphism depends on choices of representatives.
\end{theorem}

\begin{proof}
This is Theorem~\ref{des:thm:cover} and
Corollary~\ref{des:cor:matrix}.
\end{proof}

In type $A$, the cover can be classified by sequences: every
admissible datum is $(W\mathbf a)/T$, where $T$ is a finite
group of simultaneous translations preserving the multiset of
$\mathbf a$.  It is abelian and its order divides the sequence
length.  This gives an independent sequence proof of additive PBW
and, through a faithfully flat Laurent extension, the global
multiplicative theorem.  Rank-two reduction then gives the latter
in types $D$ and $E$.

For a distinguished simple root $\alpha_1$, consider
\[
 A_\star^\Lambda=\cR_\star/
 \langle y_1^{\Lambda_{m_1(c)}}e_c:c\in\cC\rangle,
 \qquad \star\in\{\mathrm{add},\mathrm{mult}\}.
\]
Under the directed-path condition of
Proposition~\ref{des:prop:root-nilpotence}, these quotients are
finite dimensional for every uniform quiver, with explicit
nilpotence and dimension bounds.  The condition holds at any
vertex of a connected simply-laced diagram, and at a short-root
vertex in the other irreducible finite types.  The formal
exponential identifies the corresponding uncompleted quotients.
The proof uses central completion of a finite invariant-ring
module, justified before the quotient comparison; it does not
assume global multiplicative PBW in the remaining types.

\vskip5pt
\noindent\textbf{A representation-theoretic application.}
Suppose now that the residue quiver has no arrows.  Choose
$\Psi^+$ and put $\Gamma=\{h\in H:h\Psi^+=\Psi^+\}$.
Then $H=K_0\rtimes\Gamma$.  Let $S_\Psi=P^{K_0}$ and let
$J_\Lambda\subseteq S_\Psi$ be generated by the matrix entries
of multiplication by
$x_\beta^{\Lambda_{\rho_{c_0}(\beta)}}$ on the free
$S_\Psi$-module $P$, for $\beta\in W\alpha_1$.

\begin{theorem}[Mixed-residue quotient theorem]
\label{intro:thm:mixed}
For arrow-free data, with $N=|W|/|\Gamma|$, one has
\[
 A_{\mathrm{add}}^\Lambda\cong
 \Mat_N\bigl((S_\Psi/J_\Lambda)\rtimes\Gamma\bigr).
\]
This quotient is nonzero if and only if, for every
$\beta\in W\alpha_1$,
\[
 \Lambda_{\rho_{c_0}(\beta)}
 \ge 1+\#\{\gamma\in\Psi^+:
                       \langle\beta,\gamma^\vee\rangle\ne0\}.
\]
Under the path condition, its simple modules are indexed by
$\operatorname{Irr}(\Gamma)$ and have dimensions $N\dim\chi$.
Under that path condition the same conclusions hold for the
multiplicative root-power quotient.
\end{theorem}

\begin{proof}
See Theorem~\ref{mix:thm:quotient} and
Corollary~\ref{des:cor:cyclotomic-exp}.
\end{proof}

For the type-$D_n$ parity residues
$\rho_J(v)=\sum_{j\in J}v_j\pmod2$, with $|J|=k$ and
$2\le k\le n-2$, the criterion becomes
\[
 \Lambda_{\overline0}\ge\max\{4k-6,4(n-k)-6\},
 \qquad \Lambda_{\overline1}\ge2n-3.
\]
The group $\Gamma$ has order two, or four when $2k=n$.
For $D_4$ and $k=2$, put $p=\lfloor\Lambda_{\overline0}/2\rfloor$.
When $p\ge1$ and $\Lambda_{\overline1}\ge8p-3$, the quotient is
\[
 \Mat_{48}\left(
 \bigl(\mathbb C[z_1,z_2,z_3,z_4]/(z_1^p,z_2^p,z_3^p,z_4^p)\bigr)
                         \rtimes(\mathbb Z/2\mathbb Z)^2\right),
\]
with the group acting by double transpositions.  Its dimension is
$9216p^4$, its Loewy length is $4p-3$, and its Cartan matrix has
diagonal entries $(p^4+3p^2)/4$ and off-diagonal entries
$(p^4-p^2)/4$.  Theorem~\ref{mix:thm:D4} also gives an explicit
coefficient-ideal presentation for all exponents.

\vskip5pt
\noindent\textbf{Relation to earlier work.}
The divided-difference and nil-Hecke results are classical
\cite{BGG,Demazure,KostantKumar}; the sequence PBW input is the
quiver Hecke theorem of Khovanov--Lauda and Rouquier
\cite{KL,Rouquier}.  Rostam's fixed-point PBW and cyclotomic
descent results~\cite{RostamFixed} precede the method used here.
Our construction specifies the cover attached to an arbitrary
root-residue datum and its coefficient ideals.  Theorem~\ref{des:thm:additive-all}
matches that cover explicitly with Sauter's centralizer, index
set, and root weights~\cite{Sauter}, and uses her PBW theorem.
The mixed quotient calculation combines classical nil-Hecke
endomorphisms and coinvariant geometry with this residue data.

When there is a compatible weight-torus realization, the additive
algebra is Liu's finite-root algebra $B^\Omega$
\cite{LiuKZ}.  We retain the attribution of its PBW, Frobenius,
and affine-corner results.  Our mixed $D_4$ orbit has no direct
weight-torus realization, but it is a fourfold quotient of a
twelve-point weight-torus orbit.  Thus even this example is
covered by finite descent from Liu's theory; the direct
realization obstruction is not a separation from that theory.

The Kong--Li root-coordinate and Brundan--Kleshchev--Rouquier
strand comparisons are recovered results
\cite{BK,KongLiD,KongLiND,KongLiTypeA}.  We place the explicit
maps and their hypotheses in an appendix.  They require the
specified residue quiver and spectral parameters.  Neither our
general PBW theorem nor the same-data exponential comparison
implies an unrestricted isomorphism with a fixed affine Hecke
target.  The zero-residue restriction theorem uses Huang's
regular coinvariant representation~\cite{HuangCoinvariants}.
Questions concerning centers, matrix orders, fibers, and root-power
invariant theory lie beyond the scope of the present paper.

\vskip5pt
\noindent\textbf{Organization and notation.}
Sections~\ref{rc:sec:operators}--\ref{str:sec:pbw-families}
give the operators, the general PBW criterion, and its elementary
inputs.  Section~\ref{des:sec:descent} proves descent, the Sauter
comparison, global multiplicative PBW in simply-laced type, and
finite root-power comparison.  Section~\ref{app:sec:mixed-quotients}
contains the mixed-residue quotient calculations.
Sections~\ref{liu:sec:comparison} and~\ref{sec:zeroorbit} treat
the Liu comparison and zero-residue restriction.
Appendix~\ref{rc:sec:specializations} records the recovered
Hecke specializations; Appendix~\ref{app:proof-root-comparison}
supplies the inverse-map details.

The root index set is $I$ and the residue group is $\mathbb I$.
For $\lambda\in Q$, write $x_\lambda=1\otimes\lambda\in V$
for its degree-one polynomial generator and $X_\lambda$ for
its group-algebra character, so
$X_\lambda X_\mu=X_{\lambda+\mu}$.  Put $Y_\lambda=1-X_\lambda$.
All statements about a zero algebra allow some exponent to be
zero; an exponent zero kills its residue idempotent.

\section{Quiver Demazure algebras}
\label{rc:sec:operators}

We define quiver Demazure algebras as operator algebras and prove
that their coefficient representation is faithful. The definition
keeps the finite Weyl group, the residue quiver, and the root
coordinates separate. A PBW presentation will require the additional
condition stated in Section~\ref{rc:sec:pbw}.

Fix the root data of the introduction. In particular,
$Q=\mathbb Z\Phi=\bigoplus_{i\in I}\mathbb Z\alpha_i$.
For $\alpha\in\Phi$, let $\alpha^\vee$ be its coroot and use
the reflection convention
\begin{equation*}\label{rc:eq:root-reflections}
 s_\alpha(\lambda)
 =\lambda-\langle\lambda,\alpha^\vee\rangle\alpha
 \quad(\lambda\in Q),\qquad s_i=s_{\alpha_i}.
\end{equation*}
The pairing is the root--coroot pairing; on the real root space
it is $\langle\lambda,\alpha^\vee\rangle
=2(\lambda,\alpha)/(\alpha,\alpha)$ for a $W$-invariant
inner product. These integers make $Q$ stable under $W$.
Write $m_{ij}$ for the order of $s_i s_j$.
The coefficient field in this section is an arbitrary field
$\kk$ subject to the hypotheses below. Let $\mathbb I$ be an
abelian group.
A uniform loop-free residue quiver is specified by a function
\[
 d\colon \mathbb I\setminus\{0\}\longrightarrow\mathbb Z_{\geq0};
\]
there are $d(a-b)$ arrows from $a$ to $b$, for $a\ne b$.
We denote this quiver by $\Gamma_d$. The function $d$ is not
assumed symmetric. Translation by any element of $\mathbb I$
is a quiver automorphism. Conversely, translation invariance of the
arrow multiplicities gives this form by taking $d(r)$ to be the
number of arrows from $r$ to $0$ for $r\ne0$. Thus the present construction concerns
uniform quivers; it does not assert that every finite Dynkin
quiver can be obtained in this form. The Dynkin diagram of
$\Phi$ and the residue quiver $\Gamma_d$ are independent.

\begin{example}[Uniform residue quivers]\label{rc:ex:uniform-quivers}
The convention is that an arrow $a\to a+h$ is counted by $d(-h)$.
The following examples make both uniformity and this sign convention
explicit.
\begin{enumerate}[label=\textup{(\roman*)},leftmargin=*]
\item \emph{The oriented integer line.}
Take $\mathbb I=\mathbb Z$, set $d(-1)=1$, and set $d(r)=0$
for every other nonzero integer $r$. There is exactly one arrow
$a\to a+1$ at every vertex. The underlying graph is the infinite
line of type $A_\infty$:
\[
\begin{tikzcd}[column sep=1.6em]
\cdots \arrow[r] & -2 \arrow[r] & -1 \arrow[r] & 0 \arrow[r] & 1 \arrow[r] & 2 \arrow[r] & \cdots
\end{tikzcd}
\]
Reversing all arrows instead gives $d(1)=1$.

\item \emph{An oriented cycle.}
Take $\mathbb I=\mathbb Z/N\mathbb Z$ with $N\geq3$, and put
$d(-\overline1)=1$ and $d(r)=0$ otherwise. The arrows are
$\overline a\to\overline{a+1}$, including
$\overline{N-1}\to\overline0$. The underlying graph is the
cycle of affine type $\widetilde A_{N-1}$:
\[
\begin{tikzcd}[column sep=1.8em]
\overline0 \arrow[r] & \overline1 \arrow[r] & \cdots \arrow[r] & \overline{N-2} \arrow[r] & \overline{N-1} \arrow[llll,bend left=20]
\end{tikzcd}
\]
For $N=2$ the same rule gives one arrow in each direction between
the two vertices, since $-\overline1=\overline1$; it does not
give a single oriented edge.

\item \emph{An enlarged integer-line quiver.}
Take $\mathbb I=\mathbb Z$ with $d(-1)=d(-2)=1$ and all
other values zero. Each vertex has arrows $a\to a+1$ and
$a\to a+2$. Thus $a,a+1,a+2$ span a triangle in the underlying
undirected graph. A local picture is
\[
\begin{tikzcd}[column sep=2.2em]
a \arrow[r] \arrow[rr,bend left=18] & a+1 \arrow[r] & a+2
\end{tikzcd}
\]
This quiver is uniform but its underlying graph is not a finite or
affine Dynkin diagram. It is used in
Example~\ref{str:ex:mixed-sequence-pbw}.

\item \emph{Parallel arrows and arrows in both directions.}
Take $\mathbb I=\mathbb Z$ with $d(-1)=2$, $d(1)=1$, and
all other values zero. Between $a$ and $a+1$ there are two
arrows $a\to a+1$ and one arrow $a+1\to a$:
\[
\begin{tikzcd}[column sep=3.2em]
a \arrow[r,shift left=.7ex] \arrow[r,shift right=.7ex] & a+1 \arrow[l]
\end{tikzcd}
\]
Uniformity therefore allows multiple arrows and does not require
either $d(r)=d(-r)$ or an acyclic orientation.

\item \emph{Disconnected and arrow-free quivers.}
For $\mathbb I=\mathbb Z$, the choice $d(-2)=1$ and all other
values zero gives two disjoint oriented lines, on the even and odd
integers:
\[
\begin{tikzcd}[column sep=1.4em]
\cdots \arrow[r] & -4 \arrow[r] & -2 \arrow[r] & 0 \arrow[r] & 2 \arrow[r] & 4 \arrow[r] & \cdots
\end{tikzcd}\]
\[\begin{tikzcd}[column sep=1.4em]
\cdots \arrow[r] & -3 \arrow[r] & -1 \arrow[r] & 1 \arrow[r] & 3 \arrow[r] & 5 \arrow[r] & \cdots
\end{tikzcd}
\]
The choice $d=0$ gives no arrows, for any abelian residue group
$\mathbb I$; its vertices are isolated points, for example
\[
\cdots \qquad -2 \qquad -1 \qquad 0 \qquad 1 \qquad 2 \qquad \cdots
\]
\end{enumerate}
In every case translation preserves all arrow multiplicities. More
generally, choosing a finite set
$H\subseteq\mathbb I\setminus\{0\}$ and positive integers
$r_h$ gives a uniform quiver with $r_h$ arrows $a\to a+h$ by
setting $d(-h)=r_h$ and $d=0$ elsewhere. Finite support is
useful for examples but is not required in the definition.
\end{example}

\begin{remark}[Uniformity and Dynkin diagrams]\label{rc:rem:uniform-dynkin}
Uniformity is a condition on the arrow multiplicities on the whole
residue group. In particular, all vertices have the same incoming
and outgoing arrow multiplicities. An oriented finite path with at
least three vertices, or a finite Dynkin diagram of type $D$ or
$E$, has vertices of different degrees and cannot be the whole
underlying graph of a uniform residue quiver. Finite paths do occur
as full subquivers of the integer-line example. The term
``uniform'' does not describe the Dynkin diagram of the finite Weyl
group: any permitted Weyl type can be chosen independently of the
examples above, subject to the coefficient and residue hypotheses.
\end{remark}

\begin{definition}\label{rc:def:residues}
An admissible residue datum consists of a nonempty finite transitive
$W$-set $\cC$ and a $W$-equivariant map
\[
 \rho\colon \cC\longrightarrow\Hom_{\mathbb Z}(Q,\mathbb I)
\]
such that, with $m_\alpha(c)=\rho_c(\alpha)$ and
$m_i(c)=m_{\alpha_i}(c)$,
\begin{equation}\label{rc:eq:admissible}
 m_i(c)=0\quad\Longrightarrow\quad s_i c=c.
\end{equation}
Equivariance means
$m_\alpha(wc)=m_{w^{-1}\alpha}(c)$.
\end{definition}

An orbit in $\Hom_{\mathbb Z}(Q,\mathbb I)$ with its inclusion is
admissible: a homomorphism vanishing on $\alpha_i$ is fixed by
$s_i$. The formulation above also keeps track of residue sequences
whose restriction to the root lattice records only differences.

\begin{example}[An admissible noninjective residue datum]
\label{rc:ex:admissible-noninjective}
Let $\Phi$ be of type $A_3$, so that $
 W=\mathfrak S_4$, $Q=\bigoplus_{i=1}^3\mathbb Z\alpha_i$ and $\alpha_i=\varepsilon_i-\varepsilon_{i+1}$.
Take $\mathbb I=\mathbb Z/2\mathbb Z$
and let $\cC$ be the $\mathfrak S_4$-orbit of $c_0=(\overline0,\overline0,\overline1,\overline1)$.
Thus
\[
\cC=
\left\{
(\overline0,\overline0,\overline1,\overline1),(\overline0,\overline1,\overline0,\overline1),(\overline0,\overline1,\overline1,\overline0),
(\overline1,\overline0,\overline0,\overline1),(\overline1,\overline0,\overline1,\overline0),
(\overline1,\overline1,\overline0,\overline0)
\right\},
\]
where $W$ acts by permuting the four positions.

For $c=(a_1,a_2,a_3,a_4)\in\cC$,
define
\[
 \rho_c(\varepsilon_i-\varepsilon_j)=a_i-a_j.
\]
Since the roots of $A_3$ are the differences
$\varepsilon_i-\varepsilon_j$, this determines a homomorphism
\[
 \rho_c\colon Q\longrightarrow\mathbb Z/2\mathbb Z.
\]
In particular,
\[
 m_i(c)=m_{\alpha_i}(c)
       =\rho_c(\alpha_i)
       =a_i-a_{i+1},
 \qquad 1\le i\le3.
\]
More generally, $
 m_{\varepsilon_i-\varepsilon_j}(c)=a_i-a_j$.

For the six elements of $\cC$, the simple-root residues are
\[
\begin{array}{c|c}
c &(m_1(c),m_2(c),m_3(c))\\
\hline
(\overline0,\overline0,\overline1,\overline1)
 &(\overline0,\overline1,\overline0)\\
(\overline0,\overline1,\overline0,\overline1)
 &(\overline1,\overline1,\overline1)\\
(\overline0,\overline1,\overline1,\overline0)
 &(\overline1,\overline0,\overline1)\\
(\overline1,\overline0,\overline0,\overline1)
 &(\overline1,\overline0,\overline1)\\
(\overline1,\overline0,\overline1,\overline0)
 &(\overline1,\overline1,\overline1)\\
(\overline1,\overline1,\overline0,\overline0)
 &(\overline0,\overline1,\overline0)
\end{array}
\]

The map
\[
 \rho\colon \cC\longrightarrow
 \Hom_{\mathbb Z}(Q,\mathbb Z/2\mathbb Z)
\]
is $W$-equivariant.  Indeed,
\[
 m_\alpha(wc)=m_{w^{-1}\alpha}(c)
\]
for every $w\in W$, $\alpha\in\Phi$, and $c\in\cC$.

The admissibility condition can also be seen directly.  If
\[
 m_i(c)=a_i-a_{i+1}=0,
\]
then $a_i=a_{i+1}$.  Hence the simple reflection $s_i$, which
interchanges the $i$th and $(i+1)$st entries, fixes $c$.
Thus
\[
 m_i(c)=0\quad\Longrightarrow\quad s_i c=c.
\]

For example, at $
 c_0=(\overline0,\overline0,\overline1,\overline1)$
one has $
 (m_1(c_0),m_2(c_0),m_3(c_0))
   =(\overline0,\overline1,\overline0)$.
Hence
\[
 s_1c_0=c_0=s_3c_0,
 \qquad
 s_2c_0\ne c_0.
\]
Thus a simple-root residue may vanish even though
$\cC$ has several elements.  Vanishing of $m_i(c)$ only forces
the corresponding simple reflection $s_i$ to fix $c$; it does
not force the whole Weyl group to fix $c$.

Finally, $\rho$ is not injective.  Translation of all four
entries by $\overline1$ does not change any difference:
\[
 \rho_{c+\overline1}=\rho_c.
\]
Explicitly,
\[
\begin{aligned}
\rho_{(\overline0,\overline0,\overline1,\overline1)}
 &=\rho_{(\overline1,\overline1,\overline0,\overline0)},\\
\rho_{(\overline0,\overline1,\overline0,\overline1)}
 &=\rho_{(\overline1,\overline0,\overline1,\overline0)},\\
\rho_{(\overline0,\overline1,\overline1,\overline0)}
 &=\rho_{(\overline1,\overline0,\overline0,\overline1)}.
\end{aligned}
\]
Hence $\rho$ is two-to-one in this example.

It is useful to distinguish this from the zero-residue case.
If $\rho_c=0$ for some $c\in\cC$, then $m_i(c)=0$ for every $i\in I$.
By admissibility, every simple reflection fixes $c$: $s_i c=c$ for every $i\in I$.
Since the simple reflections generate $W$, one has $wc=c$ for
all $w\in W$.  Because $\cC$ is transitive, it follows that $\cC=Wc=\{c\}$.
Thus the zero-residue datum is necessarily a singleton.  This is
different from the situation above, where only some simple-root
residues vanish and $\cC$ still has several elements.
\end{example}

Let $R$ be a commutative $\kk$-domain on which $W$ acts faithfully,
and set $K=\Frac(R)$. Choose nonzero elements $y_\alpha\in R$ for
$\alpha\in\Phi$ satisfying
\begin{equation}\label{rc:eq:root-data}
 \begin{aligned}
 &w(y_\alpha)=y_{w\alpha},\qquad
 \nabla_i(f)=\frac{s_i(f)-f}{y_{\alpha_i}},\\
 &\nabla_i(R)\subseteq R,\qquad
 \kappa_i:=\frac1{y_{\alpha_i}}+\frac1{y_{-\alpha_i}}\in R.
 \end{aligned}
\end{equation}
The notation $y_i=y_{\alpha_i}$ will be used. No addition law for
the coordinates is assumed. The coefficient condition in
\eqref{rc:eq:root-data} is part of the definition of the permitted data.
In the additive specialization $y_\alpha=x_\alpha$; in the
multiplicative specialization $y_\alpha=Y_\alpha=1-X_\alpha$.
The symbol $x_\alpha$ always denotes an additive linear
coordinate, whereas $X_\lambda$ denotes a group-algebra character.

Put $D_R=\bigoplus_{c\in\cC}R e_c$ and
$D_K=\bigoplus_{c\in\cC}K e_c$, with componentwise multiplication, where the $e_c$ are pairwise orthogonal central idempotents summing to $1$. We identify $R$ and $K$ with their diagonal images. Give these rings the diagonal action
$w(fe_c)=w(f)e_{wc}$. In the skew group algebra
$D_K\rtimes W$, the convention is
\begin{equation}\label{rc:eq:skew}
 [w]a=w(a)[w],\qquad [w][v]=[wv].
\end{equation}

\begin{definition}[Quiver Demazure algebra]\label{rc:def:root-algebra}
For $i\in I$ define $\psi_i$ by
\begin{equation}\label{rc:eq:psi}
 \psi_i e_c=
 \begin{cases}
 y_i^{-1}([s_i]-1)e_c,&m_i(c)=0,\\
 y_{-\alpha_i}^{\,d(m_i(c))}[s_i]e_c,&m_i(c)\ne0.
 \end{cases}
\end{equation}
The \emph{quiver Demazure algebra} associated with the chosen
coefficient ring, root coordinates, residue quiver, and admissible
residue datum is the subalgebra $\cR\subseteq D_K\rtimes W$
generated by $D_R$ and all $\psi_i$.
All rational coefficients in \eqref{rc:eq:psi} are on the left of
the group symbol.
\end{definition}

The name records the two ingredients of the crossing operators:
the residue quiver determines their weights, and the zero-residue
operators are of Demazure type. The root coordinates remain part
of the defining data. The additive strand specialization recovers
the usual polynomial realization of quiver Hecke algebras, while
the multiplicative singleton zero residue gives the Demazure
algebra used below; see Section~\ref{rc:sec:specializations}.
The term applies to the operator algebra without a PBW assumption.
The presentation theorem will use the explicit rank-two condition
in Hypothesis~\ref{rc:hyp:integral}.

\begin{proposition}\label{rc:prop:basic}
The natural action of $\cR$ on $D_R$ is well defined and faithful.
For $f\in R$ one has
\begin{align}
 \psi_i e_c&=e_{s_i c}\psi_i,\label{rc:eq:transport}\\
 \psi_i f e_c&=
 \bigl(s_i(f)\psi_i+\delta_{m_i(c),0}\nabla_i(f)\bigr)e_c,
 \label{rc:eq:mixed}\\
 \psi_i^2 e_c&=
 \begin{cases}
 -\kappa_i\psi_i e_c,&m_i(c)=0,\\
 y_i^{d(m_i(c))}y_{-\alpha_i}^{d(-m_i(c))}e_c,&m_i(c)\ne0.
 \end{cases}\label{rc:eq:quadratic}
\end{align}
\end{proposition}

\begin{proof}
On a coefficient vector the action is
\[
 \psi_i(fe_c)=
 \begin{cases}
 \nabla_i(f)e_c,&m_i(c)=0,\\
 y_{-\alpha_i}^{d(m_i(c))}s_i(f)e_{s_i c},&m_i(c)\ne0.
 \end{cases}
\]
It preserves $D_R$ by \eqref{rc:eq:root-data}. Equations
\eqref{rc:eq:transport} and~\eqref{rc:eq:mixed} follow from
\eqref{rc:eq:skew}. If $u=y_i^{-1}$, then
\[
 \bigl(u([s_i]-1)\bigr)^2
 =-(u+s_i(u))u([s_i]-1).
\]
This gives the first line of \eqref{rc:eq:quadratic}. The other line
uses $m_i(s_i c)=-m_i(c)$ and $s_i(y_{-\alpha_i})=y_i$.

For faithfulness, let $A\in\mathcal R$ act as zero on $D_R$.
In the skew group algebra $D_K\rtimes W$, write uniquely
\[
 A=\sum_{c\in\mathcal C}\sum_{w\in W}
 a_{w,c}[w]e_c,\qquad a_{w,c}\in K.
\]
For every $c\in\mathcal C$ and $f\in R$, the action gives
\[
 0=A(fe_c)=\sum_{w\in W}a_{w,c}\,w(f)e_{wc}.
\]
Taking the component indexed by $d\in\mathcal C$, we obtain
\[
 \sum_{\substack{w\in W\\wc=d}}a_{w,c}\,w(f)=0
 \qquad\text{for every }f\in R.
\]

We now show that distinct automorphisms of $R$, viewed as maps
$R\to K$, are linearly independent over $K$. Suppose otherwise,
and choose a nontrivial relation
\[
 b_1\sigma_1(f)+\cdots+b_t\sigma_t(f)=0
 \qquad(f\in R)
\]
with the smallest possible number of nonzero coefficients,
where all $b_j\in K^\times$ and the $\sigma_j$ are distinct.
Evaluating at $f=1$ shows that $t\ge2$. Choose $g\in R$ such that \(\sigma_1(g)\ne\sigma_t(g)\).
Evaluate the relation at $gf$, and subtract $\sigma_1(g)$ times
the relation at $f$. Since each $\sigma_j$ is multiplicative,
this gives
\[
 \sum_{j=2}^{t}
 b_j\bigl(\sigma_j(g)-\sigma_1(g)\bigr)\sigma_j(f)=0
 \qquad(f\in R).
\]
The coefficient of $\sigma_t(f)$ is nonzero. Thus this is a
nontrivial relation with fewer than $t$ nonzero terms,
contradicting minimality.

Since $W$ acts faithfully on $R$, distinct elements of $W$
induce distinct automorphisms of $R$. Applying the preceding
independence argument to the componentwise relation gives \(
 a_{w,c}=0\) whenever \(wc=d\).
As $c$ and $d$ vary, every coefficient vanishes. Hence $A=0$,
which proves faithfulness. This argument does not require the
action of $W$ on $\mathcal C$ to be free.
\end{proof}

\begin{example}[Crossing weights for the uniform quivers]
\label{rc:ex:uniform-crossings}
Fix the coefficient and admissible residue data of
Definition~\ref{rc:def:root-algebra}, and put $r=m_i(c)\ne0$.
For the oriented integer line in
Example~\ref{rc:ex:uniform-quivers}\textup{(i)}, the crossing is
\[
 \psi_i e_c=
 \begin{cases}
 y_{-\alpha_i}[s_i]e_c,&r=-1,\\
 [s_i]e_c,&r\ne-1.
 \end{cases}
\]
For the enlarged quiver in \textup{(iii)}, the weighted case occurs
at both $r=-1$ and $r=-2$. For \textup{(iv)}, the weights at
$r=-1$ and $r=1$ are respectively $y_{-\alpha_i}^2$ and
$y_{-\alpha_i}$, so the corresponding quadratic coefficients are
\[
 \psi_i^2e_c=
 \begin{cases}
 y_i^2y_{-\alpha_i}e_c,&r=-1,\\
 y_i y_{-\alpha_i}^2e_c,&r=1.
 \end{cases}
\]
These formulas follow directly from
\eqref{rc:eq:psi} and~\eqref{rc:eq:quadratic}; they do not require
PBW. For the arrow-free quiver every nonzero-residue crossing is
an unweighted reflection. In all these examples a zero-residue
crossing is still $y_i^{-1}([s_i]-1)e_c$, independently of $d$.
\end{example}

\section{Rank-two corrections and a PBW presentation}
\label{rc:sec:pbw}

We pass from the operator definition to a presentation. The leading
coefficient of a reduced product is independent of the reduced word.
The two rank-two braid words therefore differ by terms of smaller
Bruhat order. Requiring the coefficients of these terms to lie in
the coefficient ring gives the integrality hypothesis for PBW.

For each $w\in W$ choose a reduced expression and denote its
operator product by $\psi_w$. Define
\[
 N(w)=\{\beta\in\Phi^+:w^{-1}\beta\in-\Phi^+\}.
\]

\begin{lemma}[Leading coefficient]\label{rc:lem:leading}
For $c\in\cC$ and $j=wc$,
\begin{equation}\label{rc:eq:leading}
 \psi_w e_c=a_{w,c}[w]e_c+
 \sum_{v<w}a_{v,w,c}[v]e_c,
\end{equation}
where $v<w$ denotes strict Bruhat order and
\begin{equation}\label{rc:eq:leading-product}
 a_{w,c}=
 \prod_{\beta\in N(w)}
 \begin{cases}
 y_\beta^{-1},&m_\beta(j)=0,\\
 y_{-\beta}^{d(-m_\beta(j))},&m_\beta(j)\ne0.
 \end{cases}
\end{equation}
In particular, $a_{w,c}\ne0$ and it is independent of the reduced
expression. The family $(\psi_w e_c)_{w,c}$ is linearly independent
over $R$ and is a $K$-basis of $D_K\rtimes W$.
\end{lemma}

\begin{proof}
Fix a reduced word $w=s_{i_1}\cdots s_{i_\ell}$ and put
\[
 u_k=s_{i_1}\cdots s_{i_{k-1}},\qquad
 v_k=s_{i_{k+1}}\cdots s_{i_\ell},\qquad
 \beta_k=u_k(\alpha_{i_k}).
\]
The roots $\beta_k$ enumerate $N(w)$ without repetition:
apply $N(us_i)=N(u)\sqcup\{u(\alpha_i)\}$ along the prefixes
of the word.  These are the standard inversion-set identities
\cite[Section 4.4]{BB}.

Choose the reflection term at every crossing.  At the $k$th
factor the source residue is $c_k=v_kc$, so its coefficient is
$y_{\alpha_{i_k}}^{-1}$ if $m_{\alpha_{i_k}}(c_k)=0$, and
$y_{-\alpha_{i_k}}^{d(m_{\alpha_{i_k}}(c_k))}$ otherwise.
Since $j=u_ks_{i_k}c_k$, residue equivariance gives
\[
 m_{\beta_k}(j)=m_{\alpha_{i_k}}(s_{i_k}c_k)
                         =-m_{\alpha_{i_k}}(c_k).
\]
Moving coefficients to the left applies $u_k$ to this factor.
Root-coordinate equivariance therefore gives the $\beta_k$ factor
in~\eqref{rc:eq:leading-product}.  Multiplication proves the
stated coefficient of $[w]e_c$, including the sign in the
argument of $d$.

Every other term selects a scalar part at some zero-residue
crossing.  Admissibility says that its simple reflection fixes
that residue, so omitting the reflection preserves idempotent
transport.  The resulting group element is represented by a
proper subword of the reduced word for $w$.  The deletion and
subword properties give $v<w$
\cite[Proposition 1.4.7 and Theorem 2.2.2]{BB}.  In particular
no such term contributes to $[w]e_c$.  Every term has target
$j$, so a nonzero lower coefficient also requires $vc=j$.

Each factor in~\eqref{rc:eq:leading-product} is nonzero in
$K$, and the product depends only on $N(w)$ and $j$, not the
chosen word.  For each fixed $c$, the change from $[w]e_c$
to $\psi_w e_c$ is thus triangular in a total order extending
Bruhat order, with nonzero diagonal.  It is invertible over
$K$.  This proves the basis and independence assertions.
\end{proof}

Fix $i\ne j$ and write $W_{ij}=\langle s_i,s_j\rangle$.
Its longest element $w_{ij}$ has length $m_{ij}$. Choose reduced
words for its elements of smaller length. The difference between
the two alternating products of length $m_{ij}$ has a unique
expansion
\begin{equation}\label{rc:eq:ranktwo}
 \left(\underbrace{\psi_i\psi_j\psi_i\cdots}_{m_{ij}}
 -\underbrace{\psi_j\psi_i\psi_j\cdots}_{m_{ij}}\right)e_c
 =\sum_{\substack{v\in W_{ij}\\\ell(v)<m_{ij}}}
 b^{ij}_{v,c}\psi_v e_c,
 \qquad b^{ij}_{v,c}\in K.
\end{equation}
The leading terms cancel by Lemma~\ref{rc:lem:leading}. All remaining
terms belong to the rank-two skew group algebra, so the same lemma
gives existence and uniqueness of the displayed expansion.
The coefficients are computable: expand in $[v]e_c$, remove terms
in decreasing length by division by $a_{v,c}$, and continue to the
identity. There are only finitely many operations.
Idempotent transport also forces $b^{ij}_{v,c}=0$ unless
$vc=w_{ij}c$, which can reduce the number of coefficients to check.

\begin{hypothesis}[Rank-two integrality]\label{rc:hyp:integral}
For every pair $i\ne j$, every $c\in\cC$, and every
$v\in W_{ij}$ of length less than $m_{ij}$, the coefficient
$b^{ij}_{v,c}$ in \eqref{rc:eq:ranktwo} belongs to $R$.
\end{hypothesis}

\begin{theorem}[PBW and presentation]\label{rc:thm:pbw}
Under Hypothesis~\ref{rc:hyp:integral},
\begin{equation*}\label{rc:eq:pbw}
 \cR=\bigoplus_{c\in\cC}\bigoplus_{w\in W}R\psi_w e_c.
\end{equation*}
The algebra has a presentation by the coefficient algebra $D_R$,
the generators $\psi_i$, and the relations
\eqref{rc:eq:transport}--\eqref{rc:eq:quadratic} and~\eqref{rc:eq:ranktwo}.
In particular, its rank over $R$ is $|W|\,|\cC|$.
The basis assertion holds for any choice of reduced words.
\end{theorem}

\begin{proof}
Let $A$ be the abstract algebra generated by $D_R$ and symbols
$\widetilde\psi_i$, subject to the displayed relations.  Their
validity in the operator algebra gives a surjection
$\theta:A\to\cR$ which is the identity on $D_R$.
In particular $D_R$ embeds in $A$.

Filter $A$ by the number of crossings: $F_nA$ is spanned over
$\kk$ by words with at most $n$ crossings and arbitrary
interspersed coefficient factors.  Set $F_{-1}A=0$.
Idempotent transport moves idempotents to the right, and the
mixed relation
\[
 \widetilde\psi_i f=s_i(f)\widetilde\psi_i+
       \sum_c\delta_{m_i(c),0}\nabla_i(f)e_c
\]
moves coefficients to the left.  Its correction has one fewer
crossing and lies in $D_R$.  Thus $F_nA$ is the $R$-span of
the words $\widetilde\psi_{i_1}\cdots\widetilde\psi_{i_r}e_c$
with $r\le n$, and $F_pA F_qA\subseteq F_{p+q}A$.

The quadratic relation gives $\widetilde\psi_i^2e_c\in F_1A$.
For alternating words $B_{ij},B_{ji}$ of length $m_{ij}$,
rank-two integrality gives
\[
 (B_{ij}-B_{ji})e_c
   =\sum_{\ell(v)<m_{ij}}b^{ij}_{v,c}\widetilde\psi_v e_c
                  \in F_{m_{ij}-1}A.
\]
These reductions remain valid inside longer words by
multiplicativity of the filtration.  In particular a braid
move changes an $n$-crossing word by an element of $F_{n-1}A$.

Fix a reduced word for each $w$.  We prove inductively that
\[
 F_nA=\sum_{c\in\cC}\sum_{\ell(w)\le n}
                         R\widetilde\psi_w e_c.
\]
The case $n=0$ is the coefficient algebra.  If an $n$-letter
Coxeter word is not reduced, the word property allows braid
moves until an adjacent repetition occurs
\cite[Theorem 3.3.1(i)]{BB}.  The corresponding crossing word
then lies in $F_{n-1}A$ by the quadratic relation.  If it is
reduced for $w$, braid moves connect it to the chosen word
\cite[Theorem 3.3.1(ii)]{BB}, so their crossing products differ
by $F_{n-1}A$.  In either case induction gives the assertion.
All coefficients stay in $R$ throughout the reductions.

The resulting spanning family maps under $\theta$ to the
$R$-linearly independent operators of
Lemma~\ref{rc:lem:leading}.  A kernel element expressed in
this family has every coefficient zero.  Hence $\theta$ is
injective and gives both the presentation and PBW.

Finally, the same braid reductions compare any other chosen
reduced products with these products.  For each right idempotent,
the change has leading coefficient $1$ and only terms of smaller
length.  This finite unitriangular matrix is invertible over
$R$, proving independence of the choices of reduced words.
\end{proof}
\begin{remark}\label{rc:rem:pbw-scope}
The construction and its faithful representation require only
\eqref{rc:eq:admissible} and~\eqref{rc:eq:root-data}. The PBW assertion
additionally requires Hypothesis~\ref{rc:hyp:integral}; preservation of
$R$ by the simple divided differences alone is not being asserted
to imply that hypothesis. The latter is automatic in rank one, and
also when every $y_\alpha$ is a unit in $R$: the triangular
elimination then divides only by units. At the multiplicative
singleton zero residue the ordinary braid relations hold, so every
coefficient in \eqref{rc:eq:ranktwo} is zero.
\end{remark}

\section{Unconditional PBW families and changes of coordinates}
\label{str:sec:pbw-families}

We record two classes in which PBW holds without the general rank-two
integrality hypothesis.  Root-regular data have only weighted-reflection
crossings, while type-$A$ sequence data are controlled by the ordinary
quiver Hecke PBW theorem.  Between these two families we show that an
equivariant change of root coordinates by units does not alter the operator
algebra or the PBW property.

\subsection{PBW for root-regular data}

\begin{definition}\label{str:def:regular-residues}
An admissible residue datum is \emph{root-regular} if
$m_\alpha(c)\ne0$ for every $\alpha\in\Phi$ and $c\in\cC$.
\end{definition}

\begin{theorem}[Unconditional PBW away from zero root residues]
\label{str:thm:regular-pbw}
For root-regular residue data and every coefficient ring satisfying
\eqref{rc:eq:root-data}, all braid corrections
$b^{ij}_{v,c}$ vanish. Consequently
\[
 \cR=\bigoplus_{c\in\cC}\bigoplus_{w\in W}R\psi_w e_c
\]
without any additional integrality assumption. The generators
satisfy the ordinary Coxeter braid relations, and the mixed and
quadratic relations are those of
Proposition~\ref{rc:prop:basic} at nonzero residue.
\end{theorem}

\begin{proof}
Every factor in a product of crossings is a weighted reflection;
there is no scalar divided-difference term. Thus the expansion
in Lemma~\ref{rc:lem:leading} has just one term:
\[
 \psi_w e_c=a_{w,c}[w]e_c.
\]
The coefficient $a_{w,c}$ is independent of the reduced
expression. The two alternating products in each braid relation
are therefore equal. Apply Theorem~\ref{rc:thm:pbw} with all
corrections zero.
\end{proof}

\subsection{Equivariant changes by units}

\begin{proposition}[Invariance under a change of root units]
\label{str:prop:unit-change}
Fix the residue datum and coefficient ring. Suppose $y_\alpha$
and $\widetilde y_\alpha=u_\alpha y_\alpha$ are two permitted
root-coordinate systems with $u_\alpha\in R^\times$ and
$w(u_\alpha)=u_{w\alpha}$. Their operator algebras are equal
inside $D_K\rtimes W$. If either system satisfies the PBW
assertion for chosen reduced words, so does the other.
\end{proposition}

\begin{proof}
The two crossing systems satisfy
\begin{equation}\label{str:eq:generator-change}
 \widetilde\psi_i e_c=\gamma_{i,c}\psi_i e_c,\qquad
 \gamma_{i,c}=
 \begin{cases}
 u_{\alpha_i}^{-1},&m_i(c)=0,\\
 u_{-\alpha_i}^{d(m_i(c))},&m_i(c)\ne0.
 \end{cases}
\end{equation}
All factors are units in the common coefficient algebra.  Thus each
generating set is contained in the algebra generated by the other,
and $\widetilde{\cR}=\cR$ inside $D_K\rtimes W$.

Assume PBW for the old system, and use the same reduced words for
both systems.  Equality of the algebras gives an expansion
\[
 \widetilde\psi_w e_c=\sum_v r_{v,w,c}\psi_v e_c,
 \qquad r_{v,w,c}\in R.
\]
Apply Lemma~\ref{rc:lem:leading} to both sides.  A nonzero term of
maximal length greater than $\ell(w)$ could not cancel its leading
Weyl symbol, so no such term occurs.  Among the terms of length
$\ell(w)$, only $v=w$ can occur.  Consequently
\begin{equation}\label{str:eq:triangular-change}
 \widetilde\psi_w e_c=\lambda_{w,c}\psi_w e_c+
       \sum_{\ell(v)<\ell(w)}r_{v,w,c}\psi_v e_c,
 \qquad\lambda_{w,c}=\widetilde a_{w,c}/a_{w,c}.
\end{equation}
Put $j=wc$.  The leading-coefficient product gives
\begin{equation}\label{p56:eq:unit-product}
 \lambda_{w,c}=
 \prod_{\beta\in N(w)}
 \begin{cases}
 u_\beta^{-1},&m_\beta(j)=0,\\
 u_{-\beta}^{d(-m_\beta(j))},&m_\beta(j)\ne0.
 \end{cases}
\end{equation}
This is a unit of $R$.  Thus~\eqref{str:eq:triangular-change} is
a finite triangular change-of-basis matrix with invertible diagonal.
It proves PBW for the new system.  Applying the same argument to
the inverse units proves the converse.  Notice that integrality of
the lower coefficients follows from the old PBW basis and equality
of the algebras; fraction-field triangularity alone would not give it.
\end{proof}

\begin{remark}\label{str:rem:formal-units}
Over $\widehat P=\mathbb C[[x_i:i\in I]]$, the additive coordinates $x_\alpha$ and
the multiplicative coordinates $1-\exp(x_\alpha)$ differ by
equivariant units. Proposition~\ref{str:prop:unit-change} applies
for any fixed residue-quiver data, and not only on the singleton
zero orbit. This is a formal comparison: it does not assert
equality of the uncompleted polynomial and Laurent coefficient
rings. It also explains why a formal coordinate change alone
does not supply a new deformation invariant. Formal-group
Demazure algebras provide the broader context~\cite{HMSZ}.
\end{remark}

\subsection{Unconditional PBW for type-\texorpdfstring{$A$}{A} sequence residues}
\label{str:subsec:sequence-pbw}

Root-regular data need no braid corrections. A different
unconditional family allows both zero and nonzero residues.
Here the residue orbit retains the full sequence information.
The proof uses the ordinary quiver Hecke PBW theorem and a
separate invariant mean coordinate.

\begin{theorem}[PBW for arbitrary uniform quivers on sequence residues]
\label{str:thm:sequence-pbw}
Work over $\mathbb C$ and let $W=\mathfrak S_{n+1}$ have
type $A_n$. Let $\mathbb I$ be any abelian group, let
$d:\mathbb I\setminus\{0\}\to\mathbb Z_{\geq0}$ define a
uniform residue quiver, and let $\cC\subseteq\mathbb I^{n+1}$
be a permutation orbit of sequences. Use
\[
 m_{\epsilon_p-\epsilon_q}(c)=c_p-c_q.
\]
Let the coefficient ring be $P=\Sym(V)$ or its completion
$\widehat P$ at the origin, with a permitted coordinate system
$y_\alpha=u_\alpha x_\alpha$, where every $u_\alpha$ is an
equivariant unit. Then the PBW assertion and the presentation
of Theorem~\ref{rc:thm:pbw} hold without an additional
integrality hypothesis. In particular this applies to
additive root coordinates and, in $\widehat P$, to
$y_\alpha=1-\exp(x_\alpha)$.
\end{theorem}

\begin{proof}
First use additive coordinates over $P$.
Put $P_{\rm str}=\mathbb C[z_1,\ldots,z_{n+1}]$ and
$\bar z=(z_1+\cdots+z_{n+1})/(n+1)$.
The subalgebra generated by the differences $z_p-z_q$
is a copy of $P$, and
\[
 P_{\rm str}=P[\bar z].
\]
The group $W$ fixes $\bar z$. All crossings in
Definition~\ref{rc:def:root-algebra} depend only on the
differences. Thus adjoining $\bar z$ to the root-coordinate
operator algebra gives exactly the strand operator algebra:
\begin{equation}\label{str:eq:sequence-mean-extension}
 \cR_{\rm str}=\cR[\bar z].
\end{equation}
For clarity, this equality takes place in the polynomial
extension of $D_K\rtimes W$, embedded in the strand
fraction skew group algebra. Generation by the same crossings
and the freely adjoined central variable proves both inclusions.

The strand operators are the ordinary polynomial quiver Hecke
operators for the quadratic polynomials
\[
 Q_{ab}(u,v)=(u-v)^{d(a-b)}(v-u)^{d(b-a)}\quad(a\ne b),
 \qquad Q_{aa}=0.
\]
The polynomial PBW theorem of~\cite{KL,Rouquier} therefore gives
\[
 \cR_{\rm str}
 =\bigoplus_{c\in\cC}\bigoplus_{w\in W}
                      P[\bar z]\psi_w e_c.
\]
Specializing $\bar z=0$ in
\eqref{str:eq:sequence-mean-extension} gives PBW over $P$.
Every rank-two braid difference has coefficients in $P$
in this basis and has group support in $W_{ij}$.
Parabolic subgroups are closed under Bruhat subwords, so
triangularity forces all indices in its expansion to lie in
$W_{ij}$. Thus its coefficients are precisely those
of~\eqref{rc:eq:ranktwo}, and Hypothesis~\ref{rc:hyp:integral}
holds. The presentation follows from
Theorem~\ref{rc:thm:pbw}.

For additive coordinates over $\widehat P$, the same rank-two
identities hold, with their polynomial coefficients regarded
as power series. Uniqueness of the fraction-field expansion
again verifies the integrality condition. Thus
Theorem~\ref{rc:thm:pbw} applies over $\widehat P$ as well.
Proposition~\ref{str:prop:unit-change} transports PBW to
each of the permitted unit changes in the statement.
The resulting PBW basis makes every braid correction integral,
so the presentation assertion is transported too.
Finally, $(1-\exp(x_\alpha))/x_\alpha$ is an equivariant
unit of constant term $-1$, proving the multiplicative claim.
\end{proof}

\begin{example}[Mixed residues for an enlarged quiver]
\label{str:ex:mixed-sequence-pbw}
Take type $A_2$, $\mathbb I=\mathbb Z$, and the sequence
orbit of $(0,0,2)$. Let the arrows be $a\to a+1$ and
$a\to a+2$, so $d(-1)=d(-2)=1$ and all other
multiplicities are zero. This is a uniform quiver whose
underlying graph has triangles. At $(0,0,2)$ one simple
crossing is a divided difference and the other a weighted
reflection. The orbit has three elements and stabilizer
$\mathfrak S_2$. Theorem~\ref{str:thm:sequence-pbw} gives
PBW rank $18$ over the root coefficient ring, both in the
additive case and in the completed multiplicative case.
A direct invariant-centralizer calculation identifies its center with
$P^{\mathfrak S_2}$, or $\widehat P^{\mathfrak S_2}$.
In centered coordinates $z_1+z_2+z_3=0$, the former is
$\mathbb C[z_1+z_2,(z_1-z_2)^2]$.
In additive coordinates the middle sequence $c=(0,2,0)$
has the nonzero integral braid correction
\[
 (\psi_2\psi_1\psi_2-\psi_1\psi_2\psi_1)e_c=e_c.
\]
Indeed, put $x_i=x_{\alpha_i}$ and
$w_0=s_2s_1s_2=s_1s_2s_1$. Direct expansion in the
skew group algebra gives
\[
 \begin{split}
 \psi_2\psi_1\psi_2e_c
   &=-\frac{x_2}{x_1+x_2}([w_0]-1)e_c,\\
 \psi_1\psi_2\psi_1e_c
   &=-\frac{x_2}{x_1+x_2}[w_0]e_c
     -\frac{x_1}{x_1+x_2}e_c.
 \end{split}
\]
Thus this family includes integral braid corrections that
do not vanish, as well as mixed crossing operators.
\end{example}

\begin{remark}[Sequence data and general residue data]
\label{str:rem:sequence-pbw-scope}
Theorem~\ref{str:thm:sequence-pbw} is the sequence input.
Passing to root differences can identify sequences which differ
by a common residue translation.  The next section classifies
this loss of information and descends PBW through it.  It also
proves PBW over the uncompleted multiplicative coefficient ring.
These conclusions do not assert an unrestricted cyclotomic Hecke
comparison.
\end{remark}

\section{Residue descent, PBW, and finite root-power comparison}
\label{des:sec:descent}

We first prove descent of the operator algebra and its
coefficient-generated ideals.  The canonical zero-root cover
then gives the additive comparison with Sauter's construction
in every finite Weyl type.  A separate classification by
sequence covers gives global multiplicative PBW in type $A$;
rank-two reduction extends it to simply-laced type.  Central
completion and root-nilpotence propagation give the finite
quotient comparison.  Throughout this section the field is
$\mathbb C$.

The fixed-point method has precedents in Rostam's PBW,
presentation, and cyclotomic results
\cite[Theorems 2.25, 2.43 and Corollary 2.27]{RostamFixed}.
Here the cover and its deck group are specified by the root
residues, and the ideal statement is proved for the given
coefficient generators.

\subsection{Descent through a finite residue cover}

Let $\widetilde{\cC}$ be an admissible residue datum.  Suppose a finite
group $G$ acts freely on $\widetilde{\cC}$, commutes with $W$, and
preserves every root residue.  Set
$\cC=\widetilde{\cC}/G$ and let $\pi$ be the quotient map.
The quotient is admissible: a zero simple-root residue fixes every
lift under the corresponding simple reflection.  Use the same
coefficient ring and root coordinates for the algebras
$\widetilde{\cR}$ and $\cR$.

\begin{theorem}[Residue-cover descent]\label{des:thm:cover}
Pullback of coefficient functions defines an injective homomorphism
\[
 \iota:\cR\longrightarrow\widetilde{\cR},\qquad
 f e_c\longmapsto\sum_{\pi(\tilde c)=c}f e_{\tilde c},\qquad
 \psi_i\longmapsto\widetilde\psi_i.
\]
Its image is $\widetilde{\cR}^{G}$.  For any fixed choices of reduced
words, PBW holds for $\cR$ if and only if it holds for
$\widetilde{\cR}$.

Let $I$ be the two-sided ideal of $\cR$ generated by any family
$(g_\nu)$ in $D_R$, and let $\widetilde I$ be the ideal generated by
their images in $\widetilde{\cR}$.  Then
\begin{equation}\label{des:eq:quotient}
 \widetilde I\cap\cR=I,\qquad
 \cR/I\xrightarrow{\sim}
             (\widetilde{\cR}/\widetilde I)^G.
\end{equation}
Thus root-power ideals descend with their stated generators.
\end{theorem}

\begin{proof}
The coefficient pullback extends to an embedding
$D_K\rtimes W\hookrightarrow D_K(\widetilde{\cC})\rtimes W$.
The residue values are constant on each fiber of $\pi$, so the
crossing formulas are carried to the corresponding sums upstairs.
This proves the injectivity and the formulas for $\iota$.
We identify $\cR$ with its image and write $E$ for the coefficient
algebra $D_R(\widetilde{\cC})$.

Idempotent transport gives
\begin{equation}\label{des:eq:normal-form}
 \widetilde{\cR}=\cR E=E\cR.
\end{equation}
Indeed, move each fine residue idempotent to the end of a word in
the common crossings and coefficients; moving it in the opposite
direction proves the other equality.  Here products of subspaces
mean finite sums of products.  The Reynolds map
$\operatorname{Av}(a)=|G|^{-1}\sum_{g\in G}g(a)$ is a
$\cR$-bimodule map, and $E^G=D_R$.  Averaging
\eqref{des:eq:normal-form} proves
$\widetilde{\cR}^G=\cR$.

If PBW holds upstairs, $G$ permutes the basis elements
$\widetilde\psi_w e_{\tilde c}$ and fixes their coefficients.
An invariant expansion therefore has constant coefficients on
each residue fiber.  Its orbit sums are precisely
$\iota(\psi_w e_c)$, proving PBW downstairs.  Conversely, PBW
downstairs and~\eqref{des:eq:normal-form} give spanning upstairs
by $\widetilde\psi_w e_{\tilde c}$.  Their independence is
Lemma~\ref{rc:lem:leading}.  This proves the equivalence.

Each $g_\nu$ commutes with $E$.  Using~\eqref{des:eq:normal-form},
we obtain
\[
 \widetilde I
 =\sum_\nu\cR E g_\nu\cR E
 =\sum_\nu\cR g_\nu\cR E=I E.
\]
Consequently $\operatorname{Av}(\widetilde I)=I$, and an element
of $\widetilde I\cap\cR$ belongs to $I$.  The reverse inclusion
is immediate.  Averaging a lift of a $G$-fixed quotient class proves
surjectivity in~\eqref{des:eq:quotient}; the intersection equality
proves injectivity.  This argument uses the coefficient generators,
not an unproved assertion about contraction of arbitrary ideals.
\end{proof}

\begin{corollary}[The skew-group form of descent]\label{des:cor:matrix}
Under Theorem~\ref{des:thm:cover}, put $r=|G|$.  There are
noncanonical isomorphisms
\[
 \widetilde{\cR}\rtimes G\cong\Mat_r(\cR),\qquad
 (\widetilde{\cR}/\widetilde I)\rtimes G
                         \cong\Mat_r(\cR/I).
\]
In particular, the quotient $\cR/I$ is Morita equivalent to the
displayed skew-group algebra.
\end{corollary}

\begin{proof}
We record the elementary ring argument.  Write $A$ for either
upstairs algebra and $B=A^G$.  Choose one fine residue idempotent
from each $G$-orbit and let $q$ be their sum.  The elements
$q_g=g(q)$ are orthogonal and sum to $1$.  For $a\in A$, set
\[
 b_g(a)=\sum_{h\in G}h(q_g a)\in B.
\]
Orthogonality gives $q_g b_g(a)=q_g a$, and hence
$a=\sum_g q_g b_g(a)$.  If $\sum_g q_g b_g=0$ with $b_g\in B$,
multiplication by $q_g$, followed by all $G$-translates and their
sum, gives $b_g=0$.  Thus $A_B$ is free with basis $(q_g)$.

The action map
$A\rtimes G\to\End_B(A)$ sends $a[h]$ to $v\mapsto a h(v)$.
It is onto: the map $q_h b\mapsto q_g c b$, zero on the other
basis summands, for $c\in B$, is the action of
\[
 \sum_{t\in G}q_g c\,t(q_h)[t].
\]
If $\sum_t a_t[t]$ acts as zero, apply it to $q_h$ and
multiply on the right by $q_{t_0h}$.  This gives
$a_{t_0}q_{t_0h}=0$; summing over $h$ proves $a_{t_0}=0$.
The action map is therefore an isomorphism, and the chosen right
$B$-basis identifies its target with $\Mat_r(B)$.
The same orthogonal-idempotent identities survive in the quotient;
Theorem~\ref{des:thm:cover} identifies its invariant algebra.
\end{proof}

\subsection{The zero-root cover and additive PBW in every finite type}

Fix $c_0\in\cC$ and write
\begin{equation}\label{des:eq:zero-root-data}
 \begin{gathered}
 H=\Stab_W(c_0),\qquad \rho_0=\rho_{c_0},\qquad
 \Psi=\{\alpha\in\Phi:\rho_0(\alpha)=0\},\\
 K_0=W_\Psi=\langle s_\alpha:\alpha\in\Psi\rangle.
 \end{gathered}
\end{equation}
The subscript in $K_0$ distinguishes this group from the fraction
field $K$ used in the operator construction.

\begin{proposition}[The canonical zero-root cover]
\label{des:prop:zero-root-cover}
The set $\Psi$ is a closed root subsystem, $K_0$ is normal in $H$,
and
\[
 \widetilde{\cC}=W/K_0\longrightarrow W/H=\cC
\]
is a residue cover with deck group $F=H/K_0$.  Its residues are
$\rho_{wK_0}=\rho_0\circ w^{-1}$.  On this cover,
\begin{equation}\label{des:eq:fixed-zero-cover}
 s_i(wK_0)=wK_0\quad\Longleftrightarrow\quad
 \rho_0(w^{-1}\alpha_i)=0.
\end{equation}
In particular, Theorem~\ref{des:thm:cover} and
Corollary~\ref{des:cor:matrix} apply to this cover.
The descent statements require no commutativity hypothesis on $F$.
\end{proposition}

\begin{proof}
The kernel condition is closed under negation, root addition when
the sum is a root, and reflections in its roots.  Thus $\Psi$ is
a closed root subsystem.  If $\beta\in\Psi$, choose $w$ and $i$
with $w\beta=\alpha_i$.  Admissibility at $wc_0$ gives
$s_iwc_0=wc_0$, hence $s_\beta\in H$.  This proves $K_0\subseteq H$.
Every $h\in H$ preserves $\rho_0$ and therefore normalizes both
$\Psi$ and $K_0$.  Also $\rho_0\circ s_\beta=\rho_0$ for
$\beta\in\Psi$, so the displayed residue map is well defined.

Right multiplication by $H/K_0$ on $W/K_0$ is free, commutes with
the left $W$-action, and preserves residues.  Taking inverses in
the right action gives a left deck action if desired.  The
reflections of the reflection subgroup $W_\Psi$ are precisely
$\{s_\beta:\beta\in\Psi\}$.  Consequently
$w^{-1}s_iw\in K_0$ if and only if $w^{-1}\alpha_i\in\Psi$,
which proves~\eqref{des:eq:fixed-zero-cover} and admissibility.
\end{proof}

\begin{theorem}[Additive PBW by Sauter's construction and descent]
\label{des:thm:additive-all}
Let $\Phi$ be any finite reduced crystallographic root system.
For every admissible residue datum and every uniform residue
quiver, the additive QDA over $P=\Sym_{\mathbb C}(V)$ has basis
\[
 \{\psi_w e_c:w\in W,\ c\in\cC\}
\]
over $P$, for arbitrary choices of reduced words.  It has the
presentation of Theorem~\ref{rc:thm:pbw}.
On the cover of Proposition~\ref{des:prop:zero-root-cover}, the
operator algebra is a Borel-case generalized quiver Hecke algebra
of Sauter.  Thus the PBW input upstairs is her basis theorem.
\end{theorem}

\begin{proof}
We specify the geometric data to make the comparison checkable.
Let $\mathbf G$ be the connected adjoint semisimple complex group
with root system $\Phi$, and choose
$\mathbf T\subseteq\mathbf B\subseteq\mathbf G$.
Then $X^*(\mathbf T)=Q$.  The group $J=\rho_0(Q)$ is finitely
generated.  The surjection $Q\to J$ defines a closed diagonalizable
subgroup
\[
 D=\Hom_{\mathbb Z}(J,\mathbb C^\times)\subseteq\mathbf T.
\]
Characters separate the points of $J$, including its torsion.
Hence the roots of $G_0=C_{\mathbf G}(D)^\circ$ are exactly
$\Psi$, and its Weyl group is $K_0$.
This allows disconnected $D$; only the centralizer's identity
component is used.

For each nonzero residue $r$ occurring on a root, put
\[
 \mathfrak g_r=
 \bigoplus_{\substack{\alpha\in\Phi\\\rho_0(\alpha)=r}}
                 (\operatorname{Lie}\mathbf G)_\alpha,
 \qquad
 \mathcal V=\bigoplus_{r\ne0}\mathfrak g_r^{\oplus d(-r)}.
\]
Each $\mathfrak g_r$ is a $G_0$-submodule, since it is a
$D$-weight space and $G_0$ centralizes $D$.
Use Sauter's Borel construction with these summands and one copy
of $\mathfrak u=\operatorname{Lie}(\operatorname{rad}_u\mathbf B)$
for each summand.  When $\mathcal V=0$ one may instead use the
single zero submodule and one copy of $\mathfrak u$.
These are permitted data in~\cite[\S0.1]{Sauter}.

Her indices $K_0\backslash W$ identify with $W/K_0$ by inversion.
For a target index $wK_0$, use $x=w^{-1}$ in her weight formula.
The multiplicity of $x\alpha_i$ in $\mathcal V$ is
\[
 h_{wK_0}(i)=
 \begin{cases}
 0,&\rho_0(w^{-1}\alpha_i)=0,\\
 d(-\rho_0(w^{-1}\alpha_i)),&\rho_0(w^{-1}\alpha_i)\ne0.
 \end{cases}
\]
This is independent of the representative $wK_0$.
By~\eqref{des:eq:fixed-zero-cover}, her fixed-index crossing
is $(s_i-1)/x_{\alpha_i}$, precisely our zero-residue crossing.
At a nonfixed target, her crossing from $s_iwK_0$ is
$x_{\alpha_i}^{h_{wK_0}(i)}s_i$.  The source residue is
$-\rho_0(w^{-1}\alpha_i)$, so our crossing differs from this
one by the scalar $(-1)^{h_{wK_0}(i)}$.
The coefficient multiplication operators coincide as well.
Thus~\cite[Corollary 0.1]{Sauter} identifies the two operator
algebras, with these stated sign changes.

Sauter's reduced-product basis
\cite[Lemma 19 and its proof]{Sauter} supplies PBW upstairs;
the scalar changes preserve that basis.  Descent gives PBW on
$\cC$.  To obtain the presentation, restrict to any rank-two
parabolic subgroup and any of its residue orbits.  The same
argument gives a polynomial reduced-product basis there.
Its triangular leading terms show that the difference of the
two longest braid words has polynomial coefficients on the
lower reduced products.  Uniqueness of triangular elimination
identifies these with~\eqref{rc:eq:ranktwo}.  Thus
Hypothesis~\ref{rc:hyp:integral} holds and
Theorem~\ref{rc:thm:pbw} applies.  Equivalently, the additional
$B_2$ and $G_2$ conditions in Sauter's presentation are automatic
here, since every fixed-index multiplicity is zero.
\end{proof}

The theorem gives a precise dependence on earlier geometry:
the additive cover is a Sauter algebra, and the general additive
QDA is its finite fixed algebra.  The type-$A$ argument below
remains useful because it classifies the cover by sequences and
also proves PBW for uncompleted multiplicative coefficients.

\subsection{All admissible type-\texorpdfstring{$A$}{A} residue data}

Use $\Phi=A_{n-1}$, $n\ge2$, with roots
$\epsilon_i-\epsilon_j$ and $W=\mathfrak S_n$.
For a sequence $\mathbf a=(a_1,\ldots,a_n)\in\mathbb I^n$,
write $\widetilde{\cC}=W\mathbf a$ and define its period group
\begin{equation}\label{des:eq:periods}
 T(\mathbf a)=\{t\in\mathbb I:
       \{a_1+t,\ldots,a_n+t\}=\{a_1,\ldots,a_n\}
                              \text{ as multisets}\}.
\end{equation}
This is a finite abelian group.  It acts freely on sequences by
simultaneous translation and commutes with $W$.

\begin{theorem}[Classification by sequence covers]
\label{des:thm:classification}
Every admissible type-$A_{n-1}$ residue datum is isomorphic to
\[
 \cC=(W\mathbf a)/T
 \quad\text{for a subgroup }T\le T(\mathbf a),
\]
with root residues $m_{\epsilon_i-\epsilon_j}([\mathbf a])=a_i-a_j$.
Conversely every such quotient is admissible.  If $c_0$ corresponds
to $[\mathbf a]$, and
$K=\Stab_W(\mathbf a)$, $H=\Stab_W(c_0)$, then
\begin{equation}\label{des:eq:stabilizers}
 K\triangleleft H,\qquad H/K\cong T,\qquad |T|\mid n.
\end{equation}
If the residue map $\rho$ is injective, then $T=T(\mathbf a)$.
\end{theorem}

\begin{proof}
Choose $c_0\in\cC$ and set
$a_n=0$, $a_i=\rho_{c_0}(\epsilon_i-\epsilon_n)$ for $i<n$.
Then $a_i-a_j=\rho_{c_0}(\epsilon_i-\epsilon_j)$.
The group $K$ is generated by transpositions of positions with
equal entries.  Each such transposition fixes $c_0$: conjugate it
to a simple reflection and apply admissibility to its zero root
residue.  Hence $K\subseteq H$.

Two sequences have the same differences precisely when one is a
common translate of the other.  Therefore, for each $h\in H$,
there is a unique $t(h)\in\mathbb I$ with
$h\mathbf a=\mathbf a+t(h)$.  Permutations commute with
translations, so $t(hh')=t(h)+t(h')$.  The kernel is $K$.
Its image $T$ lies in $T(\mathbf a)$, and the map
\[
 W\mathbf a\longrightarrow\cC,\qquad w\mathbf a\longmapsto wc_0
\]
is well defined, surjective, and has exactly the $T$-orbits as
fibers.  This proves the classification and the first two claims
of~\eqref{des:eq:stabilizers}.  Translation by $T$ acts freely on
the finite set of labels occurring in $\mathbf a$.  Their
multiplicities are constant on each $T$-orbit, so their total
$n$ is divisible by $|T|$.

If $\rho$ is injective, $H$ is the full stabilizer of its
difference datum; hence every multiset period is $t(h)$ for some
$h\in H$.  Finally, on any quotient $(W\mathbf a)/T$, a zero
simple residue means equal adjacent entries in every lift.  The
simple reflection fixes these lifts, proving admissibility.
\end{proof}

\begin{corollary}[Additive PBW for every type-$A$ datum]
\label{des:cor:additive-A}
For every admissible type-$A$ datum and every uniform quiver,
the additive QDA over $P$ has the PBW basis
$\{\psi_w e_c\}_{w,c}$ and the presentation of
Theorem~\ref{rc:thm:pbw}.  The same holds over $\widehat P$ and
after any permitted equivariant change by root units.
\end{corollary}

\begin{proof}
Apply Theorem~\ref{str:thm:sequence-pbw} to the sequence cover in
Theorem~\ref{des:thm:classification}, and descend its basis using
Theorem~\ref{des:thm:cover}.  In each rank-two braid difference,
triangularity forces the PBW expansion to have support in the
rank-two parabolic subgroup.  Its coefficients are therefore the
integral coefficients of~\eqref{rc:eq:ranktwo}.  This gives the
presentation.  Completion and unit changes are treated in exactly
the same way, using Proposition~\ref{str:prop:unit-change}.
\end{proof}

\subsection{Global multiplicative PBW and rank-two reduction}

The formal exponential alone does not prove PBW over an uncompleted
Laurent ring.  The next argument uses a faithfully flat Laurent
extension instead.

\begin{lemma}[Coefficient extension]\label{des:lem:flat}
Let $R\subseteq S$ be a flat inclusion of coefficient domains with
compatible $W$-actions.  Assume the same root coordinates in $R$
are permitted over both rings.  For fixed residue data there is an
isomorphism of left $S$-modules
\[
 S\otimes_R\cR_R\xrightarrow{\sim}\cR_S,
 \qquad f\otimes a\longmapsto fa.
\]
PBW ascends along this map; if $S$ is faithfully flat over $R$,
PBW also descends, using the same chosen reduced words.
\end{lemma}

\begin{proof}
Flatness embeds $S\otimes_R\cR_R$ in the coefficient extension
of $D_{\Frac(R)}\rtimes W$, which embeds in
$D_{\Frac(S)}\rtimes W$.  Its image is $S\cR_R$.
The mixed relation moves every $S$-coefficient past a crossing,
with coefficients still in $S$, so this image is the algebra
generated by $D_S$ and the crossings.  It is $\cR_S$.
Let $L_R$ be the $R$-span of the chosen reduced products.
Their leading terms give independence, and $S\otimes_R L_R=L_S$.
Flat base change of $\cR_R/L_R$ proves ascent; faithful flatness
proves the converse.
\end{proof}

\begin{theorem}[Global multiplicative PBW in type $A$]
\label{des:thm:multiplicative-A}
For every admissible type-$A_{n-1}$ datum and every uniform residue
quiver, the QDA with
\[
 R=\mathbb C[Q],\qquad y_\alpha=1-X_\alpha
\]
has the PBW basis and presentation of Theorem~\ref{rc:thm:pbw}.
No completion or localization at the identity is required.
\end{theorem}

\begin{proof}
First use an actual sequence orbit.  Introduce independent
variables $t_1,\ldots,t_n$ and put
$S=\mathbb C[t_1^{\pm1},\ldots,t_n^{\pm1}]$.
The ordinary polynomial quiver Hecke basis over
$\mathbb C[t_1,\ldots,t_n]$ remains a basis after this
$W$-stable localization, by Lemma~\ref{des:lem:flat}.
Its root coordinates are $\delta_{ij}=t_i-t_j$.
The coordinates
\[
 1-t_i/t_j=-t_j^{-1}\delta_{ij}
\]
are obtained from them by equivariant units in $S$.  Both
coordinate systems are permitted: reflection differences of
Laurent polynomials are divisible by $t_i-t_j$, and the new
$\kappa_i$ equals $1$.  Proposition~\ref{str:prop:unit-change}
therefore gives PBW for the multiplicative operators over $S$.

Embed $R=\mathbb C[Q]$ by $X_{\epsilon_i-\epsilon_j}=t_i/t_j$.
Then $S=R[t_n^{\pm1}]$, a faithfully flat extension.  The
multiplicative coordinates belong to $R$, and their divided
differences preserve both $R$ and $S$.  Lemma~\ref{des:lem:flat}
descends PBW to $R$.  Theorem~\ref{des:thm:classification} and
Theorem~\ref{des:thm:cover} remove the sequence hypothesis.
Finally, the triangular rank-two expansion identifies the
integral braid coefficients, as in
Corollary~\ref{des:cor:additive-A}, and gives the presentation.
\end{proof}

\begin{theorem}[Global multiplicative PBW in simply-laced type]
\label{des:thm:ADE}
Let $\Phi$ be finite and simply laced.  For every admissible
residue datum and every uniform residue quiver, PBW and the
presentation of Theorem~\ref{rc:thm:pbw} hold for
$R=\mathbb C[Q]$ and $y_\alpha=1-X_\alpha$.
They remain valid under permitted equivariant root-unit changes.
\end{theorem}

\begin{proof}
Fix two simple roots.  If their reflections commute, each fixes
the other root coordinate and root residue.  Expanding the two
crossings on a residue summand shows that they commute, including
the zero-residue cases.  Thus the braid correction is zero.

Otherwise the subsystem is $A_2$.  Restrict $\cC$ to its
rank-two Weyl group and decompose it into transitive orbits,
each admissible for the restricted residues.
Theorem~\ref{des:thm:multiplicative-A} gives integral braid
coefficients over $\mathbb C[\mathbb Z\alpha_i\oplus
\mathbb Z\alpha_j]$.  This ring embeds in $\mathbb C[Q]$.
The rank-two operator expansion uses only roots of this
subsystem, so uniqueness of triangular elimination identifies
these integral coefficients with~\eqref{rc:eq:ranktwo} in the
full system.  Hypothesis~\ref{rc:hyp:integral} follows on every
residue orbit, and Theorem~\ref{rc:thm:pbw} applies.
Root-unit changes are covered by
Proposition~\ref{str:prop:unit-change}.
\end{proof}

The additive theorem in every finite type was already proved in
Theorem~\ref{des:thm:additive-all}.  For global multiplicative
coefficients with mixed residues, this argument leaves the
rank-two cases $B_2$ and $G_2$ open.  The local comparison below
does not assume their global integrality.

\subsection{Central completion and the exponential comparison}

\begin{lemma}[Completion of the coefficient representation]
\label{des:lem:central-completion}\label{kl:rem:formal-comparison}
For either standard coefficient ring $R=P$ or $R=\mathbb C[Q]$,
put $S=R^W$.  Let $\mathfrak m$ be the maximal ideal of the
origin or identity, and $\mathfrak n=\mathfrak m\cap S$.
For every admissible residue datum and uniform quiver, the QDA
$A$ is finite over the central subring $S$.  Moreover,
\[
 \widehat R=\widehat S\otimes_S R,\qquad
 \widehat A=\widehat S\otimes_S A
\]
identify with the coefficient completion and the operator algebra
generated by $\widehat R e_c$ and the same crossings on
$\bigoplus_c\widehat R e_c$.  This assertion does not require PBW.

The formal exponential gives an isomorphism
\begin{equation}\label{des:eq:formal-exp}
 \widehat A_{\mathrm{mult}}\xrightarrow{\sim}
 \widehat A_{\mathrm{add}},\qquad
 X_\lambda\longmapsto\exp(x_\lambda),
\end{equation}
with the associated root-unit change of crossings.  It identifies
corresponding root-power ideals.  A quotient in which all simple
root coordinates are nilpotent is unchanged by this central
completion.
\end{lemma}

\begin{proof}
The invariant ring $S$ is Noetherian and $R$ is finite over $S$.
All crossings commute with $S$, and the coefficient representation
is faithful.  Hence
\[
 A\subseteq\End_S\left(\bigoplus_c R e_c\right).
\]
The endomorphism module is finite over $S$, so its submodule $A$
is finite as well.  The origin or identity is $W$-fixed and is
the unique point over its image in the finite invariant quotient.
Thus $\sqrt{\mathfrak nR}=\mathfrak m$, and the two coefficient
adic topologies agree.  This proves the assertion about
$\widehat R$.

Completion of a Noetherian ring is flat.  Since
$\bigoplus_c R e_c$ is finitely presented over $S$, flat base
change of its endomorphisms and of the displayed inclusion embeds
$\widehat S\otimes_S A$ faithfully in
$\End_{\widehat S}(\bigoplus_c\widehat R e_c)$.
Its image is generated by the completed coefficients and the
crossings: each tensor is a finite sum of completed central
coefficients times words in the original generators, and the
converse inclusion is immediate.  This proves the operator
description of $\widehat A$.

In simple-root coordinates the coefficient map and its inverse are
\[
 X_\lambda\longmapsto e^{x_\lambda},\qquad
 x_i\longmapsto\log(1-Y_i).
\]
They are $W$-equivariant formal power-series maps.  In particular,
\[
 Y_\alpha\longmapsto u_\alpha x_\alpha,
 \qquad u_\alpha=\frac{1-e^{x_\alpha}}{x_\alpha},
 \qquad u_\alpha(0)=-1.
\]
Proposition~\ref{str:prop:unit-change} identifies the generated
operator algebras.  Multiplication of a root power by the unit
$u_\alpha^m$ preserves its generated ideal, giving the ideal
assertion.

Finally suppose $y_i^{N_i}=0$ in a quotient $A/I$.  In its
commutative coefficient image, the ideal generated by the $y_i$
has zero $M$th power for $M=1+\sum_i(N_i-1)$.  Therefore
$\mathfrak n^M(A/I)=0$.  Since $A/I$ is finite over $S$,
$\widehat S\otimes_S(A/I)=A/I$.  Exactness of completion now
identifies the completed quotient with the original one.
All series act by finite truncations.  No completion at the
two-sided ideal generated by the $y_i$ is being used.
\end{proof}

\begin{corollary}[Completed PBW in every finite type]
\label{des:cor:completed-all}
In both standard coordinate systems, the completed QDA has a
reduced-product basis over its completed coefficient ring for
every finite Weyl type, admissible residue datum, and uniform
quiver.  Its integral rank-two presentation is obtained by
completion and the root-unit change.
\end{corollary}

\begin{proof}
Complete Theorem~\ref{des:thm:additive-all} and apply
Lemma~\ref{des:lem:central-completion} and
Proposition~\ref{str:prop:unit-change}.
\end{proof}

\subsection{Finite root-power quotients for arbitrary uniform quivers}

Fix a simple root $\alpha_1$ and finitely supported
$\Lambda:\mathbb I\to\mathbb Z_{\ge0}$.  For either standard
coordinate system, put
\[
 A_\star^\Lambda=\cR_\star/
 \left\langle y_1^{\Lambda_{m_1(c)}}e_c:c\in\cC\right\rangle,
 \qquad \star\in\{\mathrm{add},\mathrm{mult}\}.
\]
An exponent zero kills its residue idempotent.  The nilpotence
propagation used by Kong--Li for the standard residue quiver
\cite[Lemma 4.9]{KongLiD}\cite[Lemma 4.8]{KongLiND}
extends to arbitrary arrow multiplicities.  We record its bounds,
since these make finiteness independent of a Hecke realization.

\begin{proposition}[Propagation of root nilpotence]
\label{des:prop:root-nilpotence}
Suppose every vertex admits a path to $1$ such that
$s_r(\alpha_t)=\alpha_t+\alpha_r$ at each step from $r$ toward
$t$ and then toward $1$.  This holds for every vertex $1$ in a
connected simply-laced diagram; in the other irreducible finite
types it holds when $\alpha_1$ is short.  Choose a rooted spanning
tree of such steps.  Define
\[
 \begin{aligned}
 N_1&=\max\bigl(\{1\}\cup\{\Lambda_{m_1(c)}:c\in\cC\}\bigr),\\
 B_r&=\max\bigl(\{0\}\cup
       \{d(m_r(c))+d(-m_r(c)):c\in\cC,\ m_r(c)\ne0\}\bigr),
 \end{aligned}
\]
and, if $t$ is the parent of $r$, define
$N_r=N_t(N_t+B_r)$.  In $A_\star^\Lambda$ one has
$y_r^{N_r}=0$ for every simple root.  Consequently all root
coordinates are nilpotent.  If PBW holds, then
\begin{equation}\label{des:eq:dimension-bound}
 \dim_{\mathbb C} A_\star^\Lambda
          \le |W|\,|\cC|\prod_{r\in I}N_r.
\end{equation}
The nilpotence argument itself uses no characteristic-zero assumption.
\end{proposition}

\begin{proof}
Write $\kappa=0$ or $1$ for the additive or multiplicative case.
The coordinate laws give
\[
 y_{-\alpha}=-y_\alpha(1-\kappa y_\alpha)^{-1},\qquad
 y_{\alpha+\beta}=y_\alpha+y_\beta-\kappa y_\alpha y_\beta,
\]
where $1-\kappa y_\alpha$ is a coefficient unit.
The defining root powers give $y_1^{N_1}=0$.  Suppose
$y_t^N=0$, where $t$ is the parent of $r$, and put
$g=s_r(y_t)=y_t+y_r-\kappa y_ty_r$.

If $m_r(c)\ne0$, the mixed and quadratic relations give
\[
 0=\psi_r y_t^N\psi_r e_c
   =g^N\psi_r^2e_c=g^N y_r^b u e_c,
 \qquad b=d(m_r(c))+d(-m_r(c)),
\]
with $u$ a coefficient unit.  In the commutative coefficient
image on $e_c$, reduction modulo $(y_t e_c)$ yields
$y_r^{N+b}e_c\in(y_t e_c)$.  Taking the $N$th power proves
$y_r^{N(N+b)}e_c=0$.

If $m_r(c)=0$, admissibility gives $s_rc=c$.  The mixed relation,
multiplied on the right by $y_r$, gives
$\psi_r g^N y_r e_c=-g^N e_c$ because $s_r(g)^N=y_t^N=0$.
Applying $\psi_r$ on the left and using
$\psi_r^2e_c=-\kappa\psi_r e_c$ gives
$\psi_r g^N e_c=-\kappa g^N e_c$.
Substitution yields $(1-\kappa y_r)g^N e_c=0$, hence $g^N e_c=0$.
Since $g$ and $y_t$ commute,
\[
 y_r e_c=(g-y_t)(1-\kappa y_t)^{-1}e_c
 \quad\Longrightarrow\quad y_r^{2N-1}e_c=0.
\]
The chosen bound $N(N+B_r)$ dominates both cases.  Summing over
$c$ and proceeding outward along the tree proves the assertions
for simple roots.  The coefficient ideal generated by these
commuting nilpotents is nilpotent, so the coordinate laws give
nilpotence for every root.

In the additive case the coefficient image is spanned by monomials
with exponents less than $N_r$.  The same is true in the
multiplicative case, since
$(1-y_r)^{-1}=\sum_{j=0}^{N_r-1}y_r^j$ in the quotient.
Multiplying these monomials by the PBW spanning set proves
\eqref{des:eq:dimension-bound}.
\end{proof}

\begin{corollary}[Uncompleted comparison of root-power quotients]
\label{des:cor:cyclotomic-exp}
Assume the path hypothesis of
Proposition~\ref{des:prop:root-nilpotence}.  In every finite Weyl
type, for every admissible residue datum, every uniform quiver,
and the same $\Lambda$ on both sides, the quotients
$A_{\mathrm{add}}^\Lambda$ and
$A_{\mathrm{mult}}^\Lambda$ are finite dimensional and satisfy
\eqref{des:eq:dimension-bound}.  There is an isomorphism
\[
 A_{\mathrm{mult}}^\Lambda\xrightarrow{\sim}
 A_{\mathrm{add}}^\Lambda
\]
induced by $X_\lambda\mapsto\exp(x_\lambda)$ and the associated
root-unit change of crossings.  The exponential is a finite
polynomial on these nilpotent quotients.
If every exponent $\Lambda_{m_1(c)}$ is at least $|\Phi^+|+1$,
both quotients are nonzero and have dimension at least $|W|\,|\cC|$.
\end{corollary}

\begin{proof}
By Theorem~\ref{des:thm:additive-all} and
Proposition~\ref{des:prop:root-nilpotence}, the additive quotient
is finite dimensional with the asserted bound.  The proposition
also proves nilpotence in the multiplicative quotient without
using PBW there.  Lemma~\ref{des:lem:central-completion} identifies
the completed algebras and corresponding ideals, and shows that
completion changes neither quotient.  This proves the displayed
isomorphism and transfers the dimension bound to the
multiplicative quotient.  It does not require global
multiplicative PBW.

For the last assertion use the classical coinvariant algebra
$C_0=P/(P^W_+P)$, which has dimension $|W|$ and top polynomial
degree $|\Phi^+|$~\cite{BGG,KostantKumar}.  The operator action
descends to $\bigoplus_c C_0e_c$: the ideal $P^W_+P$ is
preserved by the reflections and divided differences, hence by
all weighted crossings.  The stated exponents kill this module.
Multiplication by $\bigoplus_c C_0e_c$ on itself is faithful,
so the coefficient operators already span a subspace of dimension
$|W|\,|\cC|$ in the image of $A_{\mathrm{add}}^\Lambda$.
The isomorphism transfers the same bound to the multiplicative
quotient.
\end{proof}

This compares root-power quotients for the same residue data.
An identification with a specified affine Hecke quotient still
requires the parameter and residue-quiver hypotheses stated below.

\subsection{Mixed residue examples and direct realization obstructions}

\begin{example}[A cyclic mixed-residue cover]\label{des:ex:A5}
Take $W=\mathfrak S_6$, $\mathbb I=\mathbb Z/3\mathbb Z$, and
$\mathbf a=(0,0,1,1,2,2)$.  Then
\[
 K=\mathfrak S_2^3,\qquad T(\mathbf a)=\mathbb Z/3\mathbb Z,
 \qquad |W\mathbf a|=90,
 \qquad |(W\mathbf a)/T(\mathbf a)|=30.
\]
The quotient is the actual orbit of its root-difference datum.
It has both zero and nonzero root residues.  For any arrow
multiplicities, its additive and multiplicative QDAs have PBW
rank $720\cdot30=21600$ over their respective coefficient rings.
Its stabilizer has order $24$ and is transitive on the six
positions.  It is not a Young subgroup: a transitive Young
subgroup would be all of $\mathfrak S_6$.  Stabilizers in the
weight torus of $SL_6$ are Young subgroups, so this orbit has no
weight-torus realization.  For the cyclic quiver, the strand
cover and its cyclic fixed algebra belong to the fixed-point
setting of~\cite{RostamFixed}; the root-coordinate statement
retains only the difference variables.
\end{example}

\begin{example}[A noncyclic translation group]\label{des:ex:noncyclic}
Let $\mathbb I=(\mathbb Z/2\mathbb Z)^2$ and take a sequence of
length eight containing each of its four elements twice.  Its
period group is all of $\mathbb I$.  The sequence orbit has
$8!/(2!)^4=2520$ elements and its difference orbit has $630$.
Theorem~\ref{des:thm:classification} and the two PBW theorems
apply with a noncyclic deck group.  For example, choose
$d(r)=1$ for all three nonzero residues.  This gives mixed
divided differences and weighted reflections.  Quotients by
any prescribed root powers satisfy~\eqref{des:eq:quotient},
and the associated skew-group algebra is a $4$ by $4$ matrix
algebra over the descended quotient.
\end{example}

\begin{proposition}[A mixed type-$D$ family]\label{des:prop:D-family}
Let $\Phi=D_n$, $n\ge4$, in $\mathbb R^n$, so
\[
 Q=\{(v_1,\ldots,v_n)\in\mathbb Z^n:\textstyle\sum_i v_i
       \text{ is even}\}.
\]
Choose $J\subseteq\{1,\ldots,n\}$ of size $k$, with
$2\le k\le n-2$, and let
$\rho_J(v)=\sum_{j\in J}v_j\pmod2$.
For its orbit $\cC=W\rho_J$ one has
\[
 |\cC|=
 \begin{cases}
 \binom nk,&2k\ne n,\\
 \tfrac12\binom nk,&2k=n.
 \end{cases}
\]
For every $d(\overline1)\ge0$, the associated QDAs have mixed
root residues and unconditional additive and global multiplicative
PBW.  The orbit $\cC$ has no weight-torus realization in the sense
of Definition~\ref{liu:def:realization}.
\end{proposition}

\begin{proof}
The Weyl group consists of signed permutations with an even
number of sign changes.  Sign changes do not alter coordinate
parities.  Two forms $\rho_J$ and $\rho_{J'}$ agree on $Q$
precisely when $J'=J$ or $J'=J^c$: the annihilator of the
even-sum subspace of $(\mathbb Z/2\mathbb Z)^n$ is generated
by the all-ones vector.  Permutations now give the orbit formula.
The roots $\pm\epsilon_i\pm\epsilon_j$ have residue zero
when their two indices are both in $J$ or both in $J^c$,
and residue one otherwise.  Both cases occur.  Admissibility
is automatic for an orbit of homomorphisms, and
Theorems~\ref{des:thm:additive-all} and~\ref{des:thm:ADE} give PBW.

Let $H=\Stab_W(\rho_J)$.  It contains every even sign change.
If a weight-torus point $\ell$ were fixed by $H$, apply the
sign change in positions $i,j$ to the weight
$\lambda=\tfrac12(\epsilon_1+\cdots+\epsilon_n)$.
Since $h\lambda-\lambda=-\epsilon_i-\epsilon_j$, invariance
implies $\ell(\epsilon_i+\epsilon_j)=1$ for every $i\ne j$.
Using a third index also gives
$\ell(\epsilon_i-\epsilon_j)=1$.  Thus $\ell$ is trivial
on every root.  Every Weyl reflection fixes such a character
of the weight lattice, so $\Stab_W(\ell)=W$.  But $H$ is
proper, as the displayed orbit has more than one element.
No equivariant bijection with a weight-torus orbit can exist.
\end{proof}

For $D_4$ and $J=\{1,2\}$, the simple residues in the usual
basis $\epsilon_1-\epsilon_2,\epsilon_2-\epsilon_3,
\epsilon_3-\epsilon_4,\epsilon_3+\epsilon_4$ are $(0,1,0,0)$.
The orbit has three elements, its stabilizer has order $64$,
and the PBW rank is $192\cdot3=576$.  This is a mixed-residue
example, not a consequence of the root-regular theorem.  The
failure of a direct weight-torus realization does not exclude a
finite cover.  Proposition~\ref{liu:prop:D4-cover} constructs a
twelve-point weight-torus cover explicitly.  The root-power
quotients of the arrow-free specialization are computed next.
For general permitted coordinate systems, the rank-two
integrality hypothesis remains the criterion in
Theorem~\ref{rc:thm:pbw}.

\section{Mixed residues and finite quotients without arrows}
\label{app:sec:mixed-quotients}

We now compute a class of finite quotients, including their simple
modules and, in type $D_4$, their Cartan matrices.  In this section
the residue quiver is \emph{arrow-free}: $d(r)=0$ for every
nonzero residue.  Zero-residue crossings are still divided
differences; the other crossings are Weyl operators.  The residue
datum can therefore remain genuinely mixed.
We first use additive coefficients over $\mathbb C$.
Corollary~\ref{des:cor:cyclotomic-exp} transfers the finite-quotient
conclusions to multiplicative coordinates under its path hypothesis.

\subsection{The zero-root subsystem and its complement}

Use $H$, $\Psi$, and $K_0$ from~\eqref{des:eq:zero-root-data}.
Choose a positive system $\Psi^+$ and put
\[
 \Gamma=\{h\in H:h\Psi^+=\Psi^+\},\qquad S_\Psi=P^{K_0}.
\]
The group $H$ permutes the chambers of $\Psi$, and $K_0$ acts
simply transitively on them.  Thus every $h\in H$ has a unique
factorization $h=k\gamma$, with $k\in K_0$ and $\gamma\in\Gamma$:
\begin{equation}\label{mix:eq:stabilizer-splitting}
 H=K_0\rtimes\Gamma,\qquad \Gamma\cong H/K_0.
\end{equation}
This remains valid when $\Psi$ does not span $V$; the action on
the complementary fixed subspace is retained.

\begin{theorem}[Matrix description of the arrow-free algebra]
\label{mix:thm:matrix}
There are noncanonical algebra isomorphisms
\begin{equation}\label{mix:eq:matrix}
 A_{\mathrm{add}}
 \cong\Mat_{|\cC|}(NH_\Psi\rtimes\Gamma)
 \cong\Mat_N(S_\Psi\rtimes\Gamma),
 \qquad N=|\cC|\,|K_0|=\frac{|W|}{|\Gamma|},
\end{equation}
where $NH_\Psi=\End_{S_\Psi}(P)$ is the classical nil-Hecke
algebra for $\Psi$, with its fixed polynomial coordinates included.
The action of $\Gamma$ on $NH_\Psi$ is conjugation by its action
on $P$; its action on $S_\Psi$ is the induced coefficient action.
\end{theorem}

\begin{proof}
Every Weyl operator belongs to $A_{\mathrm{add}}$.  Indeed,
on a nonzero residue $\psi_i e_c=[s_i]e_c$, whereas on a zero
residue
\[
 [s_i]e_c=(1+x_{\alpha_i}\psi_i)e_c.
\]
Choose $w_c\in W$ with $w_cc_0=c$ and $w_{c_0}=1$.
The invertible transports $[w_c]e_{c_0}$ and their inverses
identify the algebra with the full matrix algebra over
$e_{c_0}A_{\mathrm{add}}e_{c_0}$.

We identify this corner as an operator algebra on $P$.
It contains $P$, the group $H$, and
$(s_\beta-1)/x_\beta$ for every $\beta\in\Psi$:
conjugate a zero simple-root crossing on a suitable residue by
the available Weyl operators.  Conversely, transport a generating
crossing from $c$ to $s_ic$ back to $c_0$.
If its residue is nonzero, its transported operator is
$[w_{s_ic}^{-1}s_iw_c]\in H$.
If its residue is zero, admissibility gives $s_ic=c$, and
the transported operator is the divided difference for
$w_c^{-1}\alpha_i\in\Psi$.
Transported coefficients are in $P$.  These facts identify the
corner with the algebra generated by $NH_\Psi$ and $H$.
Here we use the classical equality
$NH_\Psi=\End_{S_\Psi}(P)$ and the freeness of $P$ over
$S_\Psi$ of rank $|K_0|$~\cite{BGG,KostantKumar}.

The splitting~\eqref{mix:eq:stabilizer-splitting} now gives a
surjective map $NH_\Psi\rtimes\Gamma$ onto the corner.
It is injective: after passing to rational coefficients, the
terms indexed by different $\gamma$ are supported on disjoint
cosets $K_0\gamma$ in the skew group algebra
$\Frac(P)\rtimes W$.  Independence of its Weyl terms proves
the required independence before localization.

For the second matrix description, choose a graded
$\Gamma$-stable lift $\mathcal H$ of
$C_\Psi=P/(S_{\Psi,+}P)$ inside $P$.  Such a lift exists by
averaging a graded linear splitting over the finite group
$\Gamma$.  The usual free-module decomposition is then
$P=S_\Psi\otimes_{\mathbb C}\mathcal H$.
Choose a homogeneous basis of $\mathcal H$, and let $M_\gamma$
be the constant matrix of $\gamma$ on it.  On
$\Mat_{|K_0|}(S_\Psi)$ the action is
$B\mapsto M_\gamma\gamma(B)M_\gamma^{-1}$, where
$\gamma(B)$ acts entrywise.  The assignment
\[
 B[\gamma]\longmapsto B M_\gamma[\gamma]
\]
untwists this action and identifies its skew group algebra with
$\Mat_{|K_0|}(S_\Psi\rtimes\Gamma)$.
Combining the two matrix identifications proves the result.
\end{proof}

\begin{example}[The arrow-free hypothesis is essential in type $A_2$]
\label{mix:ex:A2-non-arrow-free}
The assumption $d=0$ in Theorem~\ref{mix:thm:matrix} is structural.
Already in type $A_2$, the matrix description can fail for a uniform
quiver with arrows.  Let $W=\mathfrak S_3$ act on the permutation
orbit
\[
 \cC=W(0,1,2)\subseteq\mathbb Z^3,
\]
and take the oriented uniform residue quiver on $\mathbb Z$ given by
\[
 d(-1)=1,\qquad d(r)=0\quad(r\ne-1).
\]
Thus there is one arrow $a\to a+1$ at every vertex.  At
$c_0=(0,1,2)$ one has
\[
 m_1(c_0)=m_2(c_0)=-1.
\]
Since the entries of $c_0$ are distinct, its stabilizer is trivial;
moreover no root has zero residue at $c_0$.  Hence
\[
 H=K_0=\Gamma=1,\qquad \Psi=\varnothing.
\]
In additive coordinates the two simple crossings at $c_0$ are weighted
reflections,
\[
 \psi_i e_{c_0}
   =x_{-\alpha_i}[s_i]e_{c_0}
   =-x_i[s_i]e_{c_0},
 \qquad i=1,2.
\]
Let $\mathfrak m=(x_1,x_2)\subseteq P$ and set
\[
 \overline A=A_{\mathrm{add}}/\langle\mathfrak m\rangle.
\]
Then
\[
 \overline\psi_1e_{c_0}=\overline\psi_2e_{c_0}=0.
\]
Every positive-length crossing word with right idempotent $e_{c_0}$
therefore vanishes in $\overline A$, because its rightmost crossing is
one of these two elements.  Consequently
\[
 \overline A e_{c_0}=\mathbb C e_{c_0},
 \qquad
 e_c\overline A e_{c_0}=0\quad(c\ne c_0).
\]
In particular, $\overline A$ is not a full matrix algebra over
$\mathbb C$.

If the conclusion of Theorem~\ref{mix:thm:matrix} extended to this
non-arrow-free datum, then $\Psi=\varnothing$ and $\Gamma=1$ would give
\[
 A_{\mathrm{add}}\cong\Mat_6(P),
\]
and reduction modulo $\mathfrak m$ would give
$\overline A\cong\Mat_6(\mathbb C)$, a contradiction.  Thus the
arrow-free hypothesis cannot in general be removed.  The obstruction is
precisely the root-coordinate factor in a nonzero-residue crossing:
when $d\ne0$ such a crossing need not be an invertible Weyl transport.
\end{example}

\subsection{Root-power ideals and an exact nonvanishing criterion}

Fix a distinguished simple root $\alpha_1$ and exponents $\Lambda$.
For each $\beta\in W\alpha_1$, let $M_\beta$ be the multiplication
matrix of
$x_\beta^{\Lambda_{\rho_0(\beta)}}$ on the free $S_\Psi$-module
$P$.  Define the coefficient ideal
\begin{equation}\label{mix:eq:entry-ideal}
 J_\Lambda=
 \big\langle\text{all entries of }M_\beta:
                         \beta\in W\alpha_1\big\rangle
                         \subseteq S_\Psi.
\end{equation}
It is independent of the chosen $S_\Psi$-basis: conjugation by an
invertible matrix preserves the ideal generated by all entries.
It is homogeneous, and it is $\Gamma$-stable because $\Gamma$
preserves $\rho_0$ and permutes $W\alpha_1$.

\begin{theorem}[Finite-quotient reduction and sharp threshold]
\label{mix:thm:quotient}
For arrow-free data there is an isomorphism
\begin{equation}\label{mix:eq:quotient}
 A_{\mathrm{add}}^\Lambda\cong
 \Mat_N\bigl((S_\Psi/J_\Lambda)\rtimes\Gamma\bigr),
 \qquad N=|W|/|\Gamma|.
\end{equation}
Define, for $\beta\in W\alpha_1$,
\begin{equation}\label{mix:eq:threshold}
 \nu_\Psi(\beta)
 =1+\#\{\gamma\in\Psi^+:
                        \langle\beta,\gamma^\vee\rangle\ne0\}.
\end{equation}
Then
\begin{equation}\label{mix:eq:nonzero}
 A_{\mathrm{add}}^\Lambda\ne0
 \quad\Longleftrightarrow\quad
 \Lambda_{\rho_0(\beta)}\ge\nu_\Psi(\beta)
               \quad\hbox{for every }\beta\in W\alpha_1.
\end{equation}
If the path hypothesis of
Proposition~\ref{des:prop:root-nilpotence} holds and the quotient
is nonzero, its simple modules are indexed by
$\operatorname{Irr}(\Gamma)$.  The simple module indexed by
$\chi$ has dimension $N\dim\chi$.
\end{theorem}

\begin{proof}
Put $\beta_c=w_c^{-1}\alpha_1$.  Transporting the root-power
generator on $e_c$ to the $c_0$-corner gives multiplication by
$x_{\beta_c}^{\Lambda_{\rho_0(\beta_c)}}$.
Conjugating further by $H$ gives exactly the family with
$\beta\in W\alpha_1$, since $W$ is the union of the cosets
$w_cH$.  In a full matrix algebra, the ideal generated by a
family of matrices is generated entrywise by their entries.
Because~\eqref{mix:eq:entry-ideal} is $\Gamma$-stable, the
corresponding ideal in the skew group ring is
$J_\Lambda\rtimes\Gamma$.
Theorem~\ref{mix:thm:matrix} therefore proves
\eqref{mix:eq:quotient}.

The connected graded polynomial ring $S_\Psi$ has positive ideal
$S_{\Psi,+}$.  The homogeneous ideal $J_\Lambda$ is proper if
and only if it is contained in $S_{\Psi,+}$.  By reduction of
multiplication matrices, this is equivalent to every specified
$x_\beta^{\Lambda_{\rho_0(\beta)}}$ acting as zero on $C_\Psi$.
Such a multiplication operator is zero if and only if its value
at $1$ is zero.  It remains to compute the nilpotence index of
$x_\beta$ in $C_\Psi$.

The component of $\beta$ fixed by $K_0$ vanishes in $C_\Psi$.
Its component on the span of $\Psi$ is an integral weight for
$\Psi$, since its pairings with the coroots of $\Psi$ are
integers.  Applying $K_0$, make this weight dominant, say $\lambda$.
Use the characteristic presentation
$C_\Psi\cong H^*(G_\Psi/B_\Psi;\mathbb C)$, with polynomial
degrees doubled, for a simply connected semisimple group with
root system $\Psi$~\cite{BGG,KumarFlag}.
The weight $\lambda$ defines an ample line bundle on
$G_\Psi/P_\lambda$, where the Levi simple roots are those
orthogonal to $\lambda$.  Its pullback has first Chern class
$x_\lambda$ on the full flag variety, up to the immaterial
characteristic-map sign.

The pullback in cohomology is injective, by the flag fibration
and its even-dimensional Schubert cells.  If
$m=\dim(G_\Psi/P_\lambda)$, the $m$th power of the ample
class is nonzero and its $(m+1)$st power is zero.  Moreover,
\[
 m=\#\{\gamma\in\Psi^+:
                         \langle\lambda,\gamma^\vee\rangle\ne0\}.
\]
The number counts pairs of nonorthogonal roots, so it is
$K_0$-invariant and equals the count for $\beta$.
Thus the nilpotence index is exactly~\eqref{mix:eq:threshold}.
If $\lambda=0$, or $\Psi$ is empty, the same statement says
that the zero element has nilpotence index one.  This proves
\eqref{mix:eq:nonzero}.

Under the path hypothesis the quotient is finite dimensional by
Corollary~\ref{des:cor:cyclotomic-exp}.  Hence the nonzero algebra
$R_\Lambda=S_\Psi/J_\Lambda$ is finite dimensional and connected
graded.  Its positive ideal is nilpotent and its residue field
is $\mathbb C$.  The radical of $R_\Lambda\rtimes\Gamma$ is
$R_{\Lambda,+}\rtimes\Gamma$: this ideal is nilpotent, and
the quotient $\mathbb C\Gamma$ is semisimple in characteristic
zero.  Its simple modules are therefore the irreducible
$\Gamma$-modules.  The matrix factor multiplies their dimensions
by $N$, as asserted.
\end{proof}

The threshold calculation uses the classical geometry of
coinvariants.  Its role here is to turn the entry-ideal reduction
into a necessary and sufficient condition for a mixed-residue
quotient, instead of a sufficient large-exponent bound.

\subsection{The parity family in type \texorpdfstring{$D_n$}{Dn}}

\begin{corollary}[Sharp thresholds and simple modules in type $D_n$]
\label{mix:cor:Dn}
Use the residue orbit of Proposition~\ref{des:prop:D-family}, with
$|J|=k$, $2\le k\le n-2$, and no arrows.  Put
$\Lambda_{\overline0}=a$ and $\Lambda_{\overline1}=b$.
For either standard coordinate system the root-power quotient is
nonzero if and only if
\begin{equation}\label{mix:eq:Dn-threshold}
 a\ge\max\{4k-6,4(n-k)-6\},\qquad b\ge2n-3.
\end{equation}
For every nonzero quotient there are exactly $|\Gamma|$ simple
modules, all of dimension $|W(D_n)|/|\Gamma|$, where
\[
 \Gamma\cong
 \begin{cases}
 \mathbb Z/2\mathbb Z,&2k\ne n,\\
 (\mathbb Z/2\mathbb Z)^2,&2k=n.
 \end{cases}
\]
Here $D_2$ in the zero-root subsystem means $A_1\sqcup A_1$.
\end{corollary}

\begin{proof}
The zero-root subsystem is $D_k\sqcup D_{n-k}$, with Weyl
group $K_0=W(D_k)\times W(D_{n-k})$.
The stabilizer $H$ consists of even signed permutations
preserving the two blocks, with block interchange also allowed
when they have equal size.  Modulo $K_0$, simultaneous odd sign
parity in both blocks gives an involution.  In the equal-size
case block interchange gives a second commuting involution.
This proves the displayed description of $H/K_0\cong\Gamma$.

For a root of $D_r$, the positive roots nonorthogonal to it
number $1+4(r-2)=4r-7$: count the root's own pair and the
roots involving exactly one of its two indices.  Thus its
nilpotence index is $4r-6$.  A root joining the two blocks
is nonorthogonal to $2(k-1)$ positive roots in the first block
and $2(n-k-1)$ in the second.  Its index is $2n-3$.
The Weyl group of $D_n$ is transitive on all roots, so
Theorem~\ref{mix:thm:quotient} gives exactly
\eqref{mix:eq:Dn-threshold}.  The diagram is connected and
simply laced, so finiteness and the multiplicative comparison
hold.  Both displayed groups are abelian; all their irreducible
representations have dimension one, giving the simple-module
statement.
\end{proof}

\subsection{An explicit mixed \texorpdfstring{$D_4$}{D4} quotient}

Take $J=\{1,2\}$ in $D_4$.  The residue orbit has three elements,
$K_0\cong(\mathbb Z/2\mathbb Z)^4$, and
$\Gamma\cong(\mathbb Z/2\mathbb Z)^2$.  Introduce the orthogonal
zero-root coordinates
\[
 u_1=x_{\epsilon_1-\epsilon_2},\quad
 u_2=x_{\epsilon_1+\epsilon_2},\quad
 u_3=x_{\epsilon_3-\epsilon_4},\quad
 u_4=x_{\epsilon_3+\epsilon_4},\qquad z_i=u_i^2.
\]
Then $S_\Psi=\mathbb C[z_1,z_2,z_3,z_4]$, and $P$ is free
over it with basis $u^\varepsilon$ for
$\varepsilon\in\{0,1\}^4$.  One may choose the complement
$\Gamma$ to act on the four $u_i$, and hence the four $z_i$,
by the double transpositions $(12)(34)$ and $(13)(24)$.
These are induced respectively by changing the signs of
$\epsilon_2,\epsilon_4$ and by interchanging the two coordinate
blocks.

For $b\ge0$ define the explicit polynomials
\begin{equation}\label{mix:eq:F-polynomial}
 F_{\varepsilon,b}(z)=
 \sum_{\substack{r\in\mathbb Z_{\ge0}^4\\2|r|+|\varepsilon|=b}}
       \frac{b!}{\prod_{i=1}^4(2r_i+\varepsilon_i)!}\,z^r,
 \qquad \varepsilon\in\{0,1\}^4.
\end{equation}
An empty sum is zero.  Thus
$(u_1+u_2+u_3+u_4)^b=\sum_\varepsilon u^\varepsilon
F_{\varepsilon,b}(z)$.

\begin{theorem}[Presentation, dimensions, and Cartan matrices]
\label{mix:thm:D4}
Let $A_{a,b}$ be the arrow-free mixed $D_4$ root-power quotient
with exponents $a,b\ge0$, in either standard coordinate system,
and put $p=\lfloor a/2\rfloor$.  There is an isomorphism
\begin{equation}\label{mix:eq:D4-presentation}
 A_{a,b}\cong
 \Mat_{48}\left(
 \left(
 \frac{\mathbb C[z_1,z_2,z_3,z_4]}
 {(z_1^p,z_2^p,z_3^p,z_4^p,F_{\varepsilon,b}:
                                  \varepsilon\in\{0,1\}^4)}
 \right)\rtimes\Gamma\right).
\end{equation}
It is nonzero exactly when $a\ge2$ and $b\ge5$; every nonzero
$A_{a,b}$ has four simple modules, each of dimension $48$.

For $p\ge1$, the $F_{\varepsilon,b}$ relations in
\eqref{mix:eq:D4-presentation} are redundant exactly when
$b\ge8p-3$.  In this range, put
$R_p=\mathbb C[z_1,z_2,z_3,z_4]/(z_1^p,z_2^p,z_3^p,z_4^p)$.
Then
\begin{equation}\label{mix:eq:D4-numerics}
 \begin{gathered}
 A_{a,b}\cong\Mat_{48}(R_p\rtimes\Gamma),\qquad
 \dim A_{a,b}=9216p^4,\\
 \dim\rad A_{a,b}=9216(p^4-1),\qquad
 \text{Loewy length of }A_{a,b}=4p-3.
 \end{gathered}
\end{equation}
Its Cartan matrix, indexed by the four characters of $\Gamma$,
has entries
\begin{equation}\label{mix:eq:D4-cartan}
 c_{\chi,\psi}=
 \begin{cases}
 (p^4+3p^2)/4,&\chi=\psi,\\
 (p^4-p^2)/4,&\chi\ne\psi.
 \end{cases}
\end{equation}
\end{theorem}

\begin{proof}
Here $N=|\cC|\,|K_0|=3\cdot16=48$.
The zero roots are $\pm u_i$.  If $a=2p$, multiplication by
$u_i^a$ is the scalar matrix $z_i^p$.  If $a=2p+1$, its
entries generate the same ideal $(z_i^p)$: multiplication by
$u_i$ on the two parity possibilities has entries $1$ and
$z_i$.  This includes $p=0$, when the ideal is the unit ideal.

The nonzero-residue roots are
$(\pm u_1\pm u_2\pm u_3\pm u_4)/2$ with all sign choices.
The scalar $2^{-b}$ is harmless.  In the parity expansion,
changing a sign of $u_i$ only multiplies each
$F_{\varepsilon,b}$ by the corresponding sign to the power
$\varepsilon_i$.  For multiplication by
$L^b=(u_1+u_2+u_3+u_4)^b$, the matrix entry in row $\eta$ and
column $\delta$ is
\[
 F_{\eta\mathbin{\oplus}\delta,b}(z)
       \prod_{i=1}^4z_i^{(\eta\mathbin{\oplus}\delta)_i\delta_i},
 \qquad \eta,\delta\in\{0,1\}^4,
\]
where $\oplus$ is coordinatewise addition modulo two.
The column $\delta=0$ contains all the $F_{\varepsilon,b}$,
and every other entry is a monomial multiple of one of them.
The entry ideal is therefore precisely the ideal in
\eqref{mix:eq:D4-presentation}.  Apply
Theorem~\ref{mix:thm:quotient} and the exponential comparison.
The nonvanishing and simple-module assertions also follow
from Corollary~\ref{mix:cor:Dn}.  Directly, for $b\le4$ one
can choose $|\varepsilon|=b$ and obtain
$F_{\varepsilon,b}=b!$, whereas for $b\ge5$ every nonzero
$F_{\varepsilon,b}$ has positive degree.

For $p\ge1$ reduce the parity expansion modulo the $z_i^p$.
All its coefficients vanish exactly when $L^b$ vanishes in
\[
 P/(z_1^p,z_2^p,z_3^p,z_4^p)P
       =\mathbb C[u_1,u_2,u_3,u_4]/(u_1^{2p},u_2^{2p},
                                             u_3^{2p},u_4^{2p}).
\]
This ring has top degree $8p-4$.  For each degree from zero
through $8p-4$ there is a monomial with all exponents at most
$2p-1$; its coefficient in $L^b$ is a positive integer.
Thus $L^b$ first vanishes at $b=8p-3$.
This proves the exact redundancy threshold.

The dimension of $R_p$ is $p^4$.  The radical of
$R_p\rtimes\Gamma$ is $R_{p,+}\rtimes\Gamma$, its quotient
is $\mathbb C\Gamma$, and its radical powers are
$R_{p,+}^j\rtimes\Gamma$.  The largest degree in $R_p$ is
$4(p-1)$, proving the dimension and Loewy formulas.

For the Cartan matrix, the character idempotents of
$\mathbb C\Gamma$ lift to primitive idempotents of
$R_p\rtimes\Gamma$.  The space between the $\chi$ and $\psi$
idempotents has dimension equal to the multiplicity of
$\chi\psi^{-1}$ in the permutation representation of $\Gamma$
on the monomial basis of $R_p$.
The identity fixes $p^4$ monomials, and each nonidentity double
transposition fixes $p^2$.  Character averaging gives
$(p^4+3p^2)/4$ for the trivial character and $(p^4-p^2)/4$ for
each of the three nontrivial characters.  Morita equivalence
preserves the Cartan matrix, proving~\eqref{mix:eq:D4-cartan}.
\end{proof}

For $p=1$ and $b\ge5$ the algebra is the direct product of four
copies of $\Mat_{48}(\mathbb C)$.  For $a\in\{4,5\}$ and
$b\ge13$, its dimension is $147456$, its Loewy length is five,
and its Cartan matrix has diagonal entries seven and all other
entries three.  These formulas concern mixed residues, despite
the absence of arrows.  Positive arrow multiplicities still have
the PBW and finiteness theorems above, but the Weyl transports
used in~\eqref{mix:eq:matrix} need not then be invertible.

\section{Comparison with quiver double Hecke algebras}
\label{liu:sec:comparison}

Work over $\mathbb C$ with $\Phi$ irreducible and reduced.
Let $L_{\mathrm{wt}}$ be its weight lattice and put
\[
 P=\operatorname{Sym}_{\mathbb C}(V),\qquad
 T_{\mathrm{wt}}=\operatorname{Hom}_{\mathbb Z}
                         (L_{\mathrm{wt}},\mathbb C^\times).
\]
For $\lambda\in L_{\mathrm{wt}}$ and $\ell\in T_{\mathrm{wt}}$,
write $\ell(\lambda)$ for character evaluation. This torus
uses the weight lattice; it is distinct from the root torus
$\operatorname{Spec}\mathbb C[Q]$ in the coefficient construction.

\subsection{A comparison with specified orbit data}

\begin{definition}\label{liu:def:realization}
A weight-torus realization of an admissible residue datum
$\cC$ is a $W$-equivariant bijection
\[
 \ell:\cC\xrightarrow{\sim}W\ell_0\subseteq T_{\mathrm{wt}}
\]
such that $m_\alpha(c)=0$ implies $\ell(c)(\alpha)=1$
for every $\alpha\in\Phi$ and $c\in\cC$.
\end{definition}

Both the bijection and the lattice in this definition matter.
A map to the root torus, or a map identifying distinct residue
labels, is not a weight-torus realization.

For such a realization define, for $\alpha\in\Phi^+$,
\begin{equation}\label{liu:eq:order}
 \Omega_{\ell(c)}(\alpha)=
 \begin{cases}
 -1,&m_\alpha(c)=0,\\
 d(m_\alpha(c)),&m_\alpha(c)\ne0.
 \end{cases}
\end{equation}
Liu's algebra $B^\Omega$ is generated on
$\bigoplus_{\ell\in W\ell_0}P e_\ell$ by multiplication and
\begin{equation}\label{liu:eq:crossings}
 \tau_i e_\ell=
 \begin{cases}
 x_{\alpha_i}^{-1}([s_i]-1)e_\ell,
                       &\Omega_\ell(\alpha_i)=-1,\\
 x_{\alpha_i}^{\Omega_\ell(\alpha_i)}[s_i]e_\ell,
                       &\Omega_\ell(\alpha_i)\geq0.
 \end{cases}
\end{equation}
The allowed order functions and this operator definition are
from~\cite[\S9.1]{LiuKZ}.

\begin{theorem}[The additive comparison]\label{liu:thm:comparison}
Suppose $\cC$ has a realization as in
Definition~\ref{liu:def:realization}. For $R=P$ and
$y_\alpha=x_\alpha$, there is an algebra isomorphism
\[
 \cR_{\mathrm{add}}\xrightarrow{\sim}B^\Omega,
 \qquad f e_c\longmapsto f e_{\ell(c)},
\]
whose crossing map is
\[
 \psi_i e_c\longmapsto
 \begin{cases}
 \tau_i e_{\ell(c)},&m_i(c)=0,\\
 (-1)^{d(m_i(c))}\tau_i e_{\ell(c)},&m_i(c)\ne0.
 \end{cases}
\]
\end{theorem}

\begin{proof}
If $\Omega_{\ell(c)}(\alpha)=-1$, the realization gives
$\ell(c)(\alpha)=1$. If $\alpha,w\alpha\in\Phi^+$, then
\[
 m_{w\alpha}(wc)=m_\alpha(c),\qquad
 \Omega_{w\ell(c)}(w\alpha)=\Omega_{\ell(c)}(\alpha).
\]
These are exactly the vanishing and covariance requirements
for the order functions in the reduced-root case. Thus
$B^\Omega$ is defined for~\eqref{liu:eq:order}.
Relabeling the direct summands by $c\mapsto\ell(c)$
intertwines multiplication and the $W$-action. A zero crossing
then agrees with~\eqref{liu:eq:crossings}. At a nonzero
residue, $y_{-\alpha_i}=-x_{\alpha_i}$ gives the displayed
sign. Each sign is invertible in the coefficient algebra.
Therefore the two sets of generators generate the same
operator algebra, proving the isomorphism.
\end{proof}

\begin{corollary}[Consequences of the comparison]
\label{liu:cor:pbw}
Under Theorem~\ref{liu:thm:comparison}, the elements
$\psi_w e_c$ form a $P$-basis, without an additional
rank-two hypothesis. The algebra is Frobenius over its center.
For compatible affine order data it is an idempotent corner
$e_\gamma\mathcal A^\omega e_\gamma$ of a quiver double
Hecke algebra.
\end{corollary}

\begin{proof}
Apply the comparison and the basis theorem, Frobenius property,
and corner theorem of~\cite[Theorem 74, Lemma 76, and
Theorem 82]{LiuKZ}. Along a reduced word, the signs in the
comparison multiply to a nonzero scalar on each source summand;
thus they preserve the basis property.
\end{proof}

These inherited results are attributed to Liu. In particular,
the additive operator construction and its PBW theorem in this
range are part of that earlier theory. The corner gives the
exact restriction $M\mapsto e_\gamma M$; an equivalence of
the full module categories would additionally require fullness
of the idempotent. The general affine order data need not be
the parameters of an ordinary trigonometric Cherednik algebra.

\subsection{Realizations and a lattice obstruction}

\begin{example}[Free orbits and strand orbits]
\label{liu:ex:realizations}
If $\cC$ is a $W$-torsor, admissibility forces every root
residue to be nonzero: conjugate a root to a simple root and
use~\eqref{rc:eq:admissible}. Choose a point $\ell_0$ of
$T_{\mathrm{wt}}$ with trivial stabilizer. Such points form
the complement of the fixed loci of the finitely many
nonidentity elements of $W$. The assignment
$wc_0\mapsto w\ell_0$ is a realization, since its vanishing
condition is vacuous.

For a genuine type-$A_{n-1}$ sequence orbit, choose distinct
numbers $t_a\in\mathbb C^\times$ for the finitely many
residue labels occurring in the sequences. For a sequence
$\mathbf i$, use the diagonal point
\[
 (\eta t_{i_1},\ldots,\eta t_{i_n}),\qquad
 \eta^n\prod_{j=1}^n t_{i_j}=1.
\]
The product is constant on the orbit, so one choice of
$\eta$ works for every sequence. This identifies the
sequence orbit with an orbit in the diagonal torus of
$SL_n$. Equality of two residues is exactly the vanishing
condition for the corresponding root. With the coefficient
ring $\mathbb C[z_1,\ldots,z_n]$, the strand algebra also
has the independent central coordinate
$u=(z_1+\cdots+z_n)/n$: it is the tensor product of the
reflection-space algebra in the theorem with $\mathbb C[u]$.
The auxiliary numbers $t_a$ are scalar torus coordinates;
the $z_j$ are polynomial strand variables.
\end{example}

\begin{example}[An $A_2$ orbit with cyclic stabilizer]
\label{liu:ex:cyclic-stabilizer}
Let $\Phi=A_2$, $\mathbb I=\mathbb Z/3\mathbb Z$, and let
$\rho:Q\to\mathbb I$ satisfy
$\rho(\alpha_1)=\rho(\alpha_2)=\overline1$. Then
\[
 s_1\rho=s_2\rho=-\rho,\qquad
 \cC=\{\rho,-\rho\},\qquad
 H=\operatorname{Stab}_W(\rho)=\langle s_1s_2\rangle.
\]
Indeed, $s_1$ sends the pair of simple residues $(1,1)$
to $(-1,2)=(2,2)$, and $s_2$ does the same. The values
on all six roots are nonzero. Thus the datum is admissible
and root-regular, and $H$ is cyclic of order three.

It has no weight-torus realization. For $A_2$ the torus
$T_{\mathrm{wt}}$ is the diagonal torus of $SL_3$, with
$S_3$ permuting its three coordinates. A point fixed by a
three-cycle has three equal coordinates and is therefore
fixed by all of $S_3$. It cannot have stabilizer $H$.
There is a root-torus realization, however:
$\lambda\mapsto\zeta^{\rho(\lambda)}$, where $\zeta$ is
a primitive cube root of unity. This distinguishes the
lattice issue from a failure of the residue equations.

The corresponding sequence cover is the orbit of $(0,1,2)$,
modulo simultaneous translation by $\mathbb Z/3\mathbb Z$.
This illustrates Theorem~\ref{des:thm:classification}, but its
root-regular PBW is already covered by
Theorem~\ref{str:thm:regular-pbw}.  The mixed type-$D$ family in
Proposition~\ref{des:prop:D-family} is outside that simpler case.
\end{example}

\begin{proposition}[A weight-torus cover of the mixed $D_4$ orbit]
\label{liu:prop:D4-cover}
The three-point orbit of Proposition~\ref{des:prop:D-family}
for $D_4$, $J=\{1,2\}$, is the quotient of a twelve-point
weight-torus orbit by a free action of
$(\mathbb Z/2\mathbb Z)^2$.  For every arrow multiplicity,
its additive QDA and its coefficient-generated quotients are
therefore obtained by finite descent from a Liu algebra
$B^\Omega$.
\end{proposition}

\begin{proof}
The weight lattice of $D_4$ is generated by
$\epsilon_1,\ldots,\epsilon_4$ and
$\lambda=(\epsilon_1+\epsilon_2+\epsilon_3+\epsilon_4)/2$,
with $2\lambda=\sum_i\epsilon_i$.  Define a character $\ell$
by
\[
 (\ell(\epsilon_1),\ldots,\ell(\epsilon_4))=(-1,-1,1,1),
 \qquad \ell(\lambda)=1.
\]
Its restriction to $Q$ is $(-1)^{\rho_J}$.
Permutations give all six choices of two negative coordinates.
An even sign change involving one negative and one positive
coordinate fixes these four signs and changes $\ell(\lambda)$
by $-1$.  Thus $W\ell$ consists of the twelve characters with
two negative coordinates and $\lambda$-value $1$ or $-1$.

The kernel of restriction of characters to $Q$ is
$Z=\Hom(L_{\mathrm{wt}}/Q,\mathbb C^\times)
\cong(\mathbb Z/2\mathbb Z)^2$.
Its elements have a common value $u\in\{1,-1\}$ on the four
$\epsilon_i$, and an independent value $v\in\{1,-1\}$ on
$\lambda$.  Multiplication by these characters sends the four
signs to their common $u$-multiple and the $\lambda$-value to
its $v$-multiple.  It preserves $W\ell$, acts freely, and
commutes with $W$, since $W$ acts trivially on
$L_{\mathrm{wt}}/Q$.
Restriction identifies the signs for $J$ and $J^c$, giving
exactly the three-point residue orbit.

Give each lift its restricted root residue.  At zero root
residue, the torus character takes that root to $1$, so its
reflection fixes the lift.  The cover is admissible and has
the tautological weight-torus realization.
Theorem~\ref{liu:thm:comparison} identifies its additive algebra
with $B^\Omega$, and Theorem~\ref{des:thm:cover} descends the
algebra and the stated ideals.
\end{proof}

Consequently the failure of a direct realization in
Proposition~\ref{des:prop:D-family} does not put the descended
additive algebra outside all known geometric or quiver double
Hecke constructions.  The explicit quotient and representation
calculations of Section~\ref{app:sec:mixed-quotients} are separate
from that realization issue.

\subsection{Formal multiplicative coordinates and Cherednik theory}

Complete the additive coefficients at the origin and the
multiplicative coefficients at the identity. On each residue
summand use the formal exponential
$X_\lambda\mapsto e^{x_\lambda}$. Writing
\[
 u_\alpha=\frac{1-e^{x_\alpha}}{x_\alpha}
             \in\widehat P^\times,
\]
the completed multiplicative crossing becomes
$u_{\alpha_i}^{-1}\psi_i^{\mathrm{add}}e_c$ at a zero
residue and
$u_{-\alpha_i}^{d(m_i(c))}\psi_i^{\mathrm{add}}e_c$
otherwise. These are changes by coefficient units.
Hence the theorem also identifies the resulting completed
operator algebras with the coefficient completion of
$B^\Omega$. This assertion does not identify the uncompleted
multiplicative algebra with $B^\Omega$.

At the singleton zero residue the additive algebra is
$NH_W=\operatorname{End}_{P^W}(P)$.  This setting is also related
to the polynomial and Dunkl realizations of rational Cherednik
algebras; see~\cite{EtingofMa}.  These connections are not used
in the PBW arguments here.

\section{Completion and regular restriction at zero residue}
\label{sec:zeroorbit}\label{sec:completed-blocks}

The singleton zero residue is independent of the arrow function by
Proposition~\ref{qv:prop:zero-independence}. We give its local
endomorphism description and determine the restrictions of its
nilpotent modules to the finite $0$-Hecke algebra. This section
works over $\mathbb C$ and uses no root-power center calculation.

\subsection{The multiplicative operator algebra}

Put $B_0=\mathbb C[Q]$ and $Y_\lambda=1-X_\lambda$. Define
\[
 C_i(f)=\frac{f-s_i(f)}{1-X_{\alpha_i}}\qquad(f\in B_0),
 \qquad
 \mathcal A_0=\langle B_0,C_i:i\in I\rangle
                  \subseteq\operatorname{End}_{\mathbb C}(B_0).
\]
For a Laurent monomial the numerator is divisible by
$1-X_{\alpha_i}$, since $\lambda-s_i\lambda$ is an integral
multiple of $\alpha_i$. Hence $C_i$ preserves $B_0$.
In the notation of Section~\ref{rc:sec:operators}, $C_i=-\psi_i$.

The classical multiplicative Demazure relations give
\[
 C_i^2=C_i,\qquad C_if=s_i(f)C_i+C_i(f),
\]
and the ordinary Coxeter braid relations~\cite{Demazure,KumarFlag}.
Here $C_i(f)$ on the right means the multiplication operator
by that polynomial. The usual reduced products $C_w$ form a
left and right $B_0$-basis of $\mathcal A_0$: the relations
give spanning and the leading Weyl terms give independence,
as in Lemma~\ref{rc:lem:leading}. In particular, the algebra
generated by the $C_i$ is the finite $0$-Hecke algebra
$H_W(0)$, with standard idempotent generators $\pi_i\mapsto C_i$.
The same skew-group centralizer argument gives $Z(\mathcal A_0)=B_0^W$.

\begin{remark}[The formal ideal-monoid model]
For a generalized preprojective algebra $\Pi(C,D)$, the vertex
ideals $I_i=\Pi(1-e_i)\Pi$ form the Demazure monoid:
$I_uI_v=I_{u*v}$, with its elements indexed by $W$
\cite{BIRS,GLS,FuGeng}. Thus the algebra with formal basis
$[I_w]$ and multiplication $[I_u][I_v]=[I_{u*v}]$ is
$H_W(0)$, and $[I_i]$ acts as $C_i$. If an actual ideal $I_w$
is zero, its formal basis symbol is still nonzero. This is a
monoid-algebra realization of the finite operator subalgebra;
it does not assert an action of the full QDA by preprojective
bimodule functors.
\end{remark}

\subsection{The completed endomorphism algebra}

Let $Z_0=B_0^W$, let $\varepsilon(X_\lambda)=1$, and put
\[
 \mathfrak n_0=\ker(\varepsilon|_{Z_0}),\qquad
 \mathscr S=\widehat{(Z_0)}_{\mathfrak n_0},\qquad
 \widehat B=\mathscr S\otimes_{Z_0}B_0,\qquad
 \widehat{\mathcal A}_0=\mathscr S\otimes_{Z_0}\mathcal A_0.
\]
By Lemma~\ref{des:lem:central-completion}, this is
central completion. The identity is the only point above its
invariant orbit, so $\widehat B$ is also the completion at the
identity of the root torus.

\begin{lemma}[The local nil-Hecke description]
\label{lem:completed-nilhecke}\label{thm:completed-matrix}
Put $P=\operatorname{Sym}_{\mathbb C}(V)$, $S=P^W$, and
$n=|W|$. The formal exponential $X_\lambda\mapsto e^{x_\lambda}$
identifies $\widehat B$ with $\widehat P$ and $\mathscr S$
with $\widehat S$. It gives
\[
 \widehat{\mathcal A}_0
 \cong\operatorname{End}_{\widehat S}(\widehat P)
 \cong\operatorname{Mat}_n(\widehat S).
\]
\end{lemma}

\begin{proof}
The formal exponential is $W$-equivariant. In characteristic zero,
averaging shows that invariants commute with completion.
The polynomial ring $P$ is free of rank $n$ over $S$, and
the classical nil-Hecke theorem identifies its $S$-linear
endomorphisms with the algebra generated by multiplication and
$\partial_i=(1-s_i)/x_{\alpha_i}$~\cite{KostantKumar}.
See also~\cite[Proposition 4.10(iv)]{Wojciechowski}.
This description commutes with completion because $P/S$ is
finite free. In logarithmic coordinates,
\[
 C_i=\frac{x_{\alpha_i}}{1-e^{x_{\alpha_i}}}\partial_i.
\]
The coefficient is a power-series unit with constant term $-1$.
Consequently the completed crossings generate precisely the
completed nil-Hecke endomorphisms. Flatness of central completion
preserves the faithful operator realization, so this equality
of images proves the first isomorphism. A basis of $\widehat P$
over $\widehat S$ gives the matrix identification.
\end{proof}

\subsection{Nilpotent modules and the finite Hecke restriction}

Let $\mathscr U_W$ be the category of finite-dimensional
$\mathcal A_0$-modules on which every $Y_{\alpha_i}$ is
nilpotent. Commutativity and the multiplicative root relations
then make every $Y_\alpha$ nilpotent. The category is closed
under submodules, quotients, and extensions. Write
\[
 \mathcal R_W=\widehat B/\mathfrak m_{\mathscr S}\widehat B.
\]
The exponential identifies its underlying algebra with $P/S_+P$.

\begin{lemma}[Regularity of the closed fiber]
\label{lem:invariant-linear-regular}
The module $\mathcal R_W$ has dimension $n=|W|$, and its
restriction to $H_W(0)$ is the left regular module.
\end{lemma}

\begin{proof}
Let $P_W$ be the weight lattice and use the multiplicative
Demazure operators
\[
 \varpi_i(f)=\frac{f-e^{-\alpha_i}s_i(f)}{1-e^{-\alpha_i}}
 \quad\text{on }\mathbb C[P_W].
\]
Put $A=\mathbb C[P_W]$ and $Z=A^W$. We use the invariant-linear
decomposition in the proof of~\cite[Theorem 5.3]{HuangCoinvariants},
before specializing the invariant parameters. If $D(w)$ is the
right descent set, $\lambda_J$ is the sum of the fundamental
weights indexed by $J$, and $\bar\varpi_i=\varpi_i-1$, that proof
gives a $Z$-basis $
 \{\bar\varpi_w e^{\lambda_{D(w)}}:w\in W\}$.
For each $J\subseteq I$, the complex span $U_J$ of the basis
elements with $D(w)=J$ is an $H_W(0)$-submodule isomorphic
to the corresponding projective indecomposable module.
These submodules give
\[
 A=\bigoplus_{J\subseteq I}Z U_J
   \cong Z\otimes_{\mathbb C}\Bigl(\bigoplus_{J\subseteq I}U_J\Bigr)
   \cong Z\otimes_{\mathbb C}H_W(0)
\]
as invariant-linear $H_W(0)$-modules. Consequently specialization
at the augmentation character of $Z$ gives the regular module.
This step uses the decomposition over $Z$; the invariant
parameters set to zero in the stated theorem are orbit sums,
whose augmentation values need not be zero.
Replacing the standard generator $\pi_i$ by $1-\pi_i$ is an
automorphism of $H_W(0)$, and
\[
 (1-\varpi_i)(f)=\frac{f-s_i(f)}{1-e^{\alpha_i}}=C_i(f).
\]
Twisting the regular module by an automorphism leaves its
isomorphism class unchanged. Finally, $Q$ and $P_W$ have the
same complexification, so the formal exponential identifies
their identity completions and the completed invariant rings.
Their identity coinvariant fibers, with these operators, are
therefore the same. This proves the assertion for $\mathcal R_W$.
\end{proof}

\begin{theorem}[Regular restriction and extension criterion]
\label{thm:regular-restriction}
The category $\mathscr U_W$ is equivalent to the category of
finite-length $\mathscr S$-modules. It has the unique simple
object $\mathcal R_W$, up to isomorphism. For $M\in\mathscr U_W$,
\[
 M|_{H_W(0)}\cong H_W(0)^{\oplus c_M},\qquad
 c_M=\operatorname{length}_{\mathscr U_W}(M)
      =\frac{\dim_{\mathbb C}M}{|W|}.
\]
A finite-dimensional $H_W(0)$-module admits a nilpotent
zero-residue QDA extension if and only if it is free.
Every nonzero such restriction is faithful.
\end{theorem}

\begin{proof}
If $Y_i^{N_i}M=0$, commutativity shows that
$(Y_i:i\in I)^N M=0$ for $N=1+\sum_i(N_i-1)$.
In particular, $\mathfrak n_0^NM=0$, and the action extends
uniquely to $\widehat{\mathcal A}_0$ through the identical
central quotient modulo $\mathfrak n_0^N$. Morphisms extend
as well. The matrix description of
Lemma~\ref{lem:completed-nilhecke} gives inverse Morita functors
\[
 M\longmapsto E_{11}M,\qquad
 L\longmapsto\widehat B\otimes_{\mathscr S}L.
\]
If $L$ has finite length, a power of $\mathfrak m_{\mathscr S}$
kills it. The maximal ideal of the local algebra
$\widehat B/\mathfrak m_{\mathscr S}\widehat B$ is nilpotent,
so some power of $(Y_i)\widehat B$ lies in
$\mathfrak m_{\mathscr S}\widehat B$. A further power
therefore kills $\widehat B\otimes_{\mathscr S}L$.
This proves that the Morita equivalence restricts to the stated
categories. The unique simple $\mathscr S$-module is $\mathbb C$;
its image is $\mathcal R_W$.

Each factor of a composition series of $M$ is $\mathcal R_W$,
which restricts regularly by the preceding lemma. Since the left
regular $H_W(0)$-module is projective, every resulting restricted
extension splits. Thus $M|_{H_W(0)}$ is a direct sum of $c_M$
regular modules, and the dimension formula follows.
Conversely, transport the action on $\mathcal R_W^{\oplus c}$
through an $H_W(0)$-module isomorphism from any free module of
rank $c$. This gives a nilpotent extension. A nonzero free
module is faithful, proving the final assertion.
\end{proof}

\begin{corollary}[Weak-order ranks]
For $M\in\mathscr U_W$ and $u,v\in W$,
\[
 \dim C_uM=c_M\#\{x\in W:x\geq_Ru\},\qquad
 C_uM\cap C_vM=C_{u\vee_Rv}M.
\]
Here $\geq_R$ is right weak order and $\vee_R$ is its join.
\end{corollary}

\begin{proof}
In the regular module, $C_uH_W(0)$ is spanned by precisely
the standard basis elements $C_x$ with $x\geq_Ru$.
Their intersections are described by the join in the finite
weak-order lattice. Apply Theorem~\ref{thm:regular-restriction}
and take direct sums.
\end{proof}

The restriction theorem does not assert semisimplicity over
$\mathcal A_0$ or uniqueness of an extension of a fixed free
$H_W(0)$-module. The corresponding finite-length $\mathscr S$-module
may have nontrivial extensions. On a cyclotomic zero-orbit
summand, when the comparison of Section~\ref{qv:subsec:zero-orbit}
applies, all root coordinates are nilpotent and the same
restriction theorem applies.

\appendix
\section{Specializations and cyclotomic affine Hecke comparisons}
\label{rc:sec:specializations}

We compare the additive and multiplicative QDA with cyclotomic
affine Hecke algebras, using the definitions
of~\cite{CherednikUnification} and the coordinate conventions and
comparison theorems of~\cite{KongLiD,KongLiND}. We first recover
these constructions, then give their cyclotomic maps and study
specified changes of residue quiver.

The two types must be distinguished throughout. The Kong--Li
root-coordinate comparisons use the type-$A$ residue quiver
$\Gamma_A$ and allow arbitrary finite Weyl type under their
hypotheses. The usual strand comparison also requires
$W=\mathfrak S_n$. When we vary the residue quiver below, the
coefficient data and finite Weyl group remain fixed.

\subsection{Additive coordinates and ordinary strand variables}
\label{kl:subsec:strand-operators}
In characteristic zero, take $R=\Sym_\kk(Q\otimes_{\mathbb Z}\kk)$
and $y_\alpha=\alpha$. Then $y_{-\alpha}=-y_\alpha$,
$\kappa_i=0$, and the zero-residue quadratic relation is
$\psi_i^2e_c=0$. Reflection differences are divisible by the
simple root, so \eqref{rc:eq:root-data} holds. On the singleton zero
residue this is the nil-Hecke algebra, with $\psi_i$ the negative
of the usual divided difference.  For the oriented residue quiver
and the full orbit data in~\eqref{kl:eq:orientation}, the
rank-two integrality condition is verified by
Theorem~\ref{rc:thm:kong-li-recovery}; the resulting algebra is the
degenerate Kong--Li algebra.  For a genuine type-$A$ sequence orbit, arbitrary uniform
residue weights are covered by Theorem~\ref{str:thm:sequence-pbw}.
For arbitrary admissible residue data over $\mathbb C$,
Theorem~\ref{des:thm:additive-all} now verifies the additive
PBW condition in every finite Weyl type.

For the strand realization take $W=\mathfrak S_n$,
$R=\kk[z_1,\ldots,z_n]$, and
$y_{\epsilon_a-\epsilon_b}=z_a-z_b$. Let $\cC$ be an orbit of
sequences $\mathbf i=(i_1,\ldots,i_n)\in \mathbb I^n$, with
$m_{\epsilon_a-\epsilon_b}(\mathbf i)=i_a-i_b$.
Formula~\eqref{rc:eq:psi} becomes
\[
 \psi_r e_{\mathbf i}=
 \begin{cases}
 \dfrac{[s_r]-1}{z_r-z_{r+1}}e_{\mathbf i},&i_r=i_{r+1},\\[2mm]
 (z_{r+1}-z_r)^{d(i_r-i_{r+1})}[s_r]e_{\mathbf i},&i_r\ne i_{r+1}.
 \end{cases}
\]
This is a standard polynomial realization of the quiver Hecke
algebra, with our sign and orientation conventions. In particular,
its quadratic polynomial is
\begin{equation}\label{rc:eq:ordinary-Q}
 Q_{ab}(u,v)=
 \begin{cases}
 0,&a=b,\\
 (u-v)^{d(a-b)}(v-u)^{d(b-a)},&a\ne b.
 \end{cases}
\end{equation}
The order of the two exponents in this formula follows directly
from \eqref{rc:eq:quadratic}. The familiar quiver Hecke braid
correction is polynomial; for example, with the usual strand
notation,
\[
 \begin{aligned}
 &(\psi_{r+1}\psi_r\psi_{r+1}-\psi_r\psi_{r+1}\psi_r)e_{\mathbf i}\\
 &\qquad=\delta_{i_r,i_{r+2}}
 \frac{Q_{i_r i_{r+1}}(z_{r+2},z_{r+1})
       -Q_{i_r i_{r+1}}(z_r,z_{r+1})}{z_{r+2}-z_r}
 e_{\mathbf i}.
 \end{aligned}
\]
This also verifies Hypothesis~\ref{rc:hyp:integral} in this example;
see the polynomial representations in~\cite{KL,Rouquier}.
The symmetric group on strands and the Cartan datum of the residue
quiver play different roles.

\subsection{The multiplicative coefficient algebra}
Put $B=\kk[Q]$ and write its basis elements as $X_\lambda$, so
$X_\lambda X_\mu=X_{\lambda+\mu}$. Set
\begin{equation*}\label{rc:eq:multiplicative}
 Y_\lambda=1-X_\lambda,
 \qquad
 Y_{\lambda+\mu}=Y_\lambda+Y_\mu-Y_\lambda Y_\mu.
\end{equation*}
The $W$-action is $w(X_\lambda)=X_{w\lambda}$. We have
\[
 Y_{-\alpha}=\frac{Y_\alpha}{Y_\alpha-1},
 \qquad \frac1{Y_\alpha}+\frac1{Y_{-\alpha}}=1.
\]
For a Laurent monomial, $f-s_i(f)$ is divisible by
$1-X_{\alpha_i}$, by the finite geometric-series identity.
Thus \eqref{rc:eq:root-data} holds with $R=B$ and $y_\alpha=Y_\alpha$.

For $\cC=\{0\}$ and $\rho=0$, define
\begin{equation}\label{rc:eq:demazure}
 C_i=-\psi_i=\frac{1-[s_i]}{1-X_{\alpha_i}},
 \qquad
 D_i(f)=\frac{f-s_i(f)}{1-X_{\alpha_i}}.
\end{equation}
The operators $C_i$ are idempotent and satisfy the ordinary Coxeter
braid relations. These are the multiplicative Demazure relations;
one may extend $B$ to the weight-lattice group algebra and use
$C_i=1-\varpi_i$, where
$\varpi_i(f)=(f-X_{-\alpha_i}s_i(f))/(1-X_{-\alpha_i})$.
Their restriction preserves $B$; see~\cite{HuangCoinvariants,HMSZ}.
Consequently Theorem~\ref{rc:thm:pbw} applies with zero braid
corrections.

\subsection{Flag varieties and the two coordinate laws}
\label{geo:subsec:operators}

Over $\mathbb C$, let $\mathbf G$ be the adjoint semisimple group
of root system $\Phi$ and choose $\mathbf T\subseteq\mathbf B$.
Then $X^*(\mathbf T)=Q$.  Put $S=P^W$.  The characteristic
presentations, with polynomial degrees doubled, are
\begin{equation}\label{geo:eq:characteristic}
 H_{\mathbf G}^*(\mathbf G/\mathbf B;\mathbb C)=P,
 \qquad H_{\mathbf G}^*(\mathrm{pt};\mathbb C)=S,
 \qquad H^*(\mathbf G/\mathbf B;\mathbb C)=P/S_+P;
\end{equation}
see~\cite{BGG,Demazure,KumarFlag}.  Choose the characteristic-map
sign for which a dominant weight gives the first Chern class of
its globally generated line bundle.  Minimal-parabolic push--pull
then acts by $\partial_i=(1-s_i)/x_i$.  In equivariant $K$-theory
the coefficient ring is $\mathbb C[Q]$, and push--pull acts by
$\varpi_i$ of~\eqref{rc:eq:demazure}.  Its complementary operator
$C_i=1-\varpi_i$ satisfies, under $X_\lambda\mapsto\exp(x_\lambda)$,
\[
 C_i=\frac{x_i}{1-\exp(x_i)}\partial_i.
\]
The displayed coefficient is a unit of constant term $-1$.
Formal Demazure algebras give the broader context~\cite{HMSZ,CZZ}.

Theorem~\ref{des:thm:additive-all} matches the additive cover
with Sauter's convolution construction by specifying its
centralizer, index set, and root weights.  Downstairs, a fixed
residue can have nonzero root residue, as in
Example~\ref{kl:ex:two-block-orbit}; the finite-cover step is
what resolves this difference in the operator rules.

\subsection{Recovery of the Kong--Li families}
\label{rc:subsec:kong-li}

We now identify the two KLR-like constructions of Kong--Li
with the additive and multiplicative specializations of
Definition~\ref{rc:def:root-algebra}. These are the polynomial
presentations used in their comparisons with modified forms of
degenerate and non-degenerate affine Hecke algebras after
localization. We specify the coefficient rings and parameter
spaces separately; the cyclotomic comparisons are given in
Section~\ref{kl:subsec:cyclotomic-maps}.

Let $\mathbb I=\mathbb Z/e\mathbb Z$, where $e=0$ means
$\mathbb I=\mathbb Z$ and, otherwise, $e\geq2$.  Give $\mathbb I$
the oriented residue quiver with one arrow $a\to a+1$. This is
$\Gamma_A$, of type $A_\infty$ when $e=0$ and cyclic affine
type $A$ otherwise. With the orientation used in
Definition~\ref{rc:def:residues}, this is
\begin{equation}\label{kl:eq:orientation}
d(a)=d_A(a):=\delta_{a,-1}\qquad(a\neq0).
\end{equation}
For the additive comparison, $e$ is the characteristic of $\kk$
with the convention $e=0$ in characteristic zero.  For the
multiplicative comparison, fix $q\in\kk^\times$ with $q\neq1$
and take $e$ to be the order of $q$ (or $e=0$ when $q$ has
infinite order).  These two choices of $e$ belong to their
respective coefficient fields; no identification of the two
parameter spaces is being made here.
Let $\cC$ be a finite $W$-orbit in
$\operatorname{Hom}_{\mathbb Z}(Q,\mathbb I)$, with its inclusion
as the residue map.  The admissibility condition
$m_i(c)=0\Rightarrow s_ic=c$ is automatic for this orbit.

We write $\cR_{\mathrm{add}}(\cC)$ and
$\cR_{\mathrm{mult}}(\cC)$ for the algebras of
Definition~\ref{rc:def:root-algebra} formed with the additive and
multiplicative coefficient data, respectively.

For the additive choice, put
\[
 P_{\mathrm{add}}=\operatorname{Sym}_{\kk}
       (Q\otimes_{\mathbb Z}\kk),\qquad
 x_\alpha=1\otimes\alpha,\qquad y_\alpha=x_\alpha .
\]
Then $x_{-\alpha}=-x_\alpha$ and $\kappa_i=0$.  The relations
for the root coordinates are the additive relations
$x_{\alpha+\beta}=x_\alpha+x_\beta$.  The polynomial ring
generated by the $x_\alpha$ subject to these relations is the
coefficient ring denoted by $\kk[y]$ in~\cite{KongLiD}.

For the multiplicative choice, put $P_{\mathrm{mult}}=\kk[Q]$ and
$Y_\alpha=1-X_\alpha$.  The relations are
\[
 Y_{\alpha+\beta}=Y_\alpha+Y_\beta-Y_\alpha Y_\beta,\qquad
 Y_{-\alpha}=\frac{Y_\alpha}{Y_\alpha-1},
\]
and $\kappa_i=1$.  The resulting root-coordinate algebra is the
ring $\kk[Y;G]$ of~\cite{KongLiND}.  Indeed, the characters
$X_{\alpha_i}=1-Y_{\alpha_i}$ generate $\kk[Q]$, and the displayed
relations express every root coordinate in terms of these
characters and their inverses.

For later reference, define
\[
 \vartheta_i=\psi_i+
 \sum_{c\in\cC}\delta_{m_i(c),0}y_i^{-1}e_c
 \quad\text{in }D_K\rtimes W.
\]
Thus $\vartheta_i e_c$ is a scalar multiple of the reflection
$[s_i]e_c$ for every $c$.  Directly from the skew-group
multiplication one obtains
\begin{equation}\begin{aligned}\label{kl:eq:theta-transport}
 \vartheta_i e_c&=e_{s_ic}\vartheta_i,
 &\vartheta_i f&=s_i(f)\vartheta_i\ (f\in K),\\
 \vartheta_i^2e_c&=
 \begin{cases}
 (y_i y_{-\alpha_i})^{-1}e_c,&m_i(c)=0,\\
 y_i^{d(m_i(c))}y_{-\alpha_i}^{d(-m_i(c))}e_c,
       &m_i(c)\neq0.
 \end{cases}
\end{aligned}
\end{equation}
By Lemma~\ref{rc:lem:leading}, the coefficient of $[w]e_c$ in a
reduced product $\vartheta_w e_c$ is nonzero and is independent of
the reduced word.  Hence the $\vartheta_i$ satisfy the ordinary Coxeter braid
relations and
\begin{equation}\label{kl:eq:theta-pbw}
D_K\rtimes W=
\bigoplus_{w\in W,\ c\in\cC}K\vartheta_w e_c .
\end{equation}
This is the operator form of the modified intertwiners in
\cite[Lemma 2.27]{KongLiD} and
\cite[Proposition 2.26]{KongLiND}.  Notice that~\eqref{kl:eq:theta-pbw}
is a fraction-field statement; it does not by itself assert
integrality of the lower braid corrections.

\begin{theorem}[Recovery of the Kong--Li KLR-like algebras]
\label{rc:thm:kong-li-recovery}
Assume that the standing hypotheses of Section~\ref{rc:sec:operators}
hold.  For the additive and multiplicative coefficient data above,
there are canonical isomorphisms
\begin{equation}\label{kl:eq:recovery}
\begin{aligned}
\cR_{\mathrm{add}}(\cC)&\cong
 L_{\mathrm{deg}}(\cC),\\
\cR_{\mathrm{mult}}(\cC)&\cong
 L_q(\cC).
\end{aligned}
\end{equation}
Here $L_{\mathrm{deg}}(\cC)$ is the polynomial KLR-like algebra
$L(E)$ in~\cite{KongLiD}, and $L_q(\cC)$ is the algebra
$\mathcal L(\cC)$ in~\cite{KongLiND}.  The isomorphisms identify
the coefficient generators, residue idempotents, and crossing
generators.  In particular, both algebras have the left and right
PBW bases indexed by $W\times\cC$.

More explicitly, in either specialization the quadratic relation
is
\begin{equation}\label{kl:eq:quadratic-table}
\psi_i^2e_c=
\begin{cases}
 -\kappa_i\psi_i e_c,&m_i(c)=0,\\
 y_{-\alpha_i}e_c,&m_i(c)=1,\ e\neq2,\\
 y_i e_c,&m_i(c)=-1,\ e\neq2,\\
 y_i y_{-\alpha_i}e_c,&m_i(c)=1,\ e=2,\\
 e_c,&m_i(c)\notin\{0,1,-1\}.
\end{cases}
\end{equation}
Thus the additive specialization gives exactly the degenerate
quadratic table, while the multiplicative specialization gives
the non-degenerate factors of~\cite[Theorem 3.2(5)]{KongLiND}.
\end{theorem}

\begin{proof}
The idempotent transport and the mixed relation agree by
construction.  With~\eqref{kl:eq:orientation}, the second line
of~\eqref{rc:eq:quadratic} gives exactly~\eqref{kl:eq:quadratic-table}.
For the additive specialization $\kappa_i=0$ and
$y_{-\alpha_i}=-x_i$; for the multiplicative specialization
$\kappa_i=1$ and $y_{-\alpha_i}=Y_{-\alpha_i}$.
These are precisely the quadratic relations in
\cite[Theorem 3.2]{KongLiD} and
\cite[Theorem 3.2(5)]{KongLiND}.

Over the fraction field, the modified crossings in the two
Kong--Li papers are the operators $\vartheta_i$ above, after
identifying their root coordinates with $x_\alpha$ or
$Y_\alpha$.  Their ordinary Coxeter braid relations and the
decompositions by reduced products are therefore exactly
\eqref{kl:eq:theta-transport}--\eqref{kl:eq:theta-pbw}.
Subtracting the zero-residue term $y_i^{-1}e_c$ recovers the
Kong--Li crossing $\psi_i e_c$.  Consequently the two maps in
\eqref{kl:eq:recovery} are the identity on the displayed
generators inside the same localized operator algebra.

It remains to check that the maps restrict to the polynomial
subalgebras.  Proposition 4.1 of each cited paper gives a left
and right PBW basis over its coefficient ring and gives a complete
set of rank-two braid corrections.  Under the coordinate
identifications, those corrections are the unique coefficients
obtained by expanding the two alternating products in the basis
of Lemma~\ref{rc:lem:leading}.  They therefore belong to the
corresponding additive or multiplicative coefficient ring.  This
proves Hypothesis~\ref{rc:hyp:integral} for both Kong--Li
families, and the identity on generators gives the asserted
isomorphisms and PBW bases.
\end{proof}

\begin{remark}[Attribution and scope]\label{kl:rem:attribution}
The complete presentations and PBW theorems in
Theorem~\ref{rc:thm:kong-li-recovery} are due to Kong--Li.
The theorem records their operator identification in the notation
used for our root-power quotients. The earlier geometric and
formal-group constructions are discussed in
Section~\ref{geo:subsec:operators}. The additive polynomial
ring and the multiplicative Laurent ring remain distinct.
The degenerate algebra is graded as in
\cite[Remark 4.3]{KongLiD}; the non-degenerate presentation is
generally not $\mathbb Z$-graded. Our use of the term
KLR-like makes no additional quantum-group categorification
claim.
\end{remark}

\subsection{Explicit cyclotomic affine Hecke comparisons}
\label{kl:subsec:cyclotomic-maps}

Retain the coefficient fields, residue groups, full residue orbits,
and orientation~\eqref{kl:eq:orientation} from
Section~\ref{rc:subsec:kong-li}. The two choices of the residue
parameter $e$ are still made separately. Fix a distinguished
simple root $\alpha_1$ and finitely supported data
$\Lambda:\mathbb I\to\mathbb Z_{\geq0}$. Put
\begin{equation}\label{kl:eq:root-cyclotomic-definition}
 \cR_\star(\cC;\Lambda)=\cR_\star(\cC)/\mathfrak a_\star,
 \qquad
 \mathfrak a_\star=
 \left\langle y_1^{\Lambda_{m_1(c)}}e_c:c\in\cC\right\rangle,
 \quad \star\in\{\mathrm{add},\mathrm{mult}\}.
\end{equation}
Here $y_\alpha=x_\alpha$ in the additive case and
$y_\alpha=Y_\alpha$ in the multiplicative case. If
$\Lambda_{m_1(c)}=0$, the displayed relation sets $e_c=0$.

We state the directed-path hypothesis in our root--coroot
convention. Every vertex $r$ can be joined to $1$ by a path
$r=r_0,r_1,\ldots,r_\ell=1$ such that
\begin{equation}\label{kl:eq:path-condition}
 \langle\alpha_{r_{j+1}},\alpha_{r_j}^{\vee}\rangle=-1
 \qquad(0\leq j<\ell).
\end{equation}
Equivalently,
$s_{r_j}(\alpha_{r_{j+1}})=\alpha_{r_{j+1}}+\alpha_{r_j}$
along that path. This is the path condition used for the
cyclotomic nilpotence arguments in~\cite{KongLiD,KongLiND}.
It holds for every distinguished vertex in a connected
simply-laced Dynkin diagram. In the non-simply-laced finite
types it holds when the distinguished simple root is short.
These are conditions on the Dynkin diagram of $W$; they do
not change the residue quiver $\Gamma_A$.

Let $P_\star$ denote the coefficient ring of
Section~\ref{rc:subsec:kong-li}, with augmentation $f\mapsto f(0)$
at $x_\alpha=0$ or $Y_\alpha=0$. Define
\begin{equation}\label{kl:eq:local-coefficients}
 \mathscr S_\star=
 \{f\in P_\star^W:f(0)\ne0\},\qquad
 \cR_\star^\circ=\mathscr S_\star^{-1}\cR_\star(\cC).
\end{equation}
This is a central localization. It inverts every coefficient
$g$ with $g(0)\ne0$: the product $\prod_{w\in W}w(g)$ lies
in $\mathscr S_\star$. Thus its coefficient ring is the local
ring $P_\star^\circ=P_\star[g^{-1}:g(0)\ne0]$.
No root coordinate is inverted in this construction.

\begin{lemma}[Nilpotence and removal of localization]
\label{kl:lem:cyclotomic-localization}
Assume~\eqref{kl:eq:path-condition}. All root coordinates are
nilpotent in $\cR_\star(\cC;\Lambda)$. The natural map is an
isomorphism
\begin{equation}\label{kl:eq:remove-localization}
 \cR_\star(\cC;\Lambda)
 \xrightarrow{\sim}
 \cR_\star^\circ/\mathfrak a_\star\cR_\star^\circ.
\end{equation}
\end{lemma}

\begin{proof}
The nilpotence assertion is due to
\cite[Lemma 4.9]{KongLiD} and \cite[Lemma 4.8]{KongLiND}.
It is also the standard-quiver case of the argument in
Proposition~\ref{des:prop:root-nilpotence}: along the directed
paths, the identities there use only the mixed and quadratic
relations and apply over the present coefficient fields.

The commuting simple root coordinates generate a nilpotent
ideal in the coefficient image.  Every coefficient $f$ with
$f(0)\ne0$ therefore has a finite geometric-series inverse
in the quotient.  In particular all elements of
$\mathscr S_\star$ are already invertible.  The universal
property of central localization gives an inverse to
\eqref{kl:eq:remove-localization}.
\end{proof}

We recall the affine Hecke algebras whose cyclotomic quotients
occur in the comparison. We use the equal-parameter forms of
\cite[Definition~1.1(b),(c), p.~414]{CherednikUnification}, with
$q_i=q$ in the non-degenerate case and $h_i=1$ in the degenerate
case. Cherednik gives these definitions over $\mathbb C$; here
we use their defining relations over the coefficient fields
specified above, as in~\cite{KongLiD,KongLiND}.
The \emph{degenerate
affine Hecke algebra} $H_{\mathrm{deg}}$ is generated by a copy of
$\Sym_\kk(Q\otimes_{\mathbb Z}\kk)$, with additive coordinates
$\xi_\lambda$, and generators $t_i$ indexed by the simple
reflections of the finite Weyl group $W$. The \emph{non-degenerate
affine Hecke algebra} $H_q$ is generated by a copy of $\kk[Q]$,
with multiplicative coordinates $X_\lambda^{\mathrm H}$, and
generators $T_i$ indexed by the same simple reflections.
Thus $\xi_{\lambda+\mu}=\xi_\lambda+\xi_\mu$ and
$X_{\lambda+\mu}^{\mathrm H}=X_\lambda^{\mathrm H}X_\mu^{\mathrm H}$,
with $\xi_0=0$ and $X_0^{\mathrm H}=1$.
Cherednik's $y_\alpha$ and $Y_\alpha$ correspond to our
$\xi_\alpha$ and $X_\alpha^{\mathrm H}$, respectively; our
centered coordinate $Y_\alpha=1-X_\alpha$ has a different meaning.
In addition to the relations in the respective coefficient
algebras, their defining relations are the Coxeter braid
relations for the $t_i$ or $T_i$ and
\begin{equation}\label{kl:eq:hecke-coordinate-conventions}
 \begin{aligned}
 t_i^2&=1,&
 t_i f-s_i(f)t_i&=\frac{s_i(f)-f}{\xi_i},\\
 (T_i+1)(T_i-q)&=0,&
 T_i f-s_i(f)T_i&=(q-1)\frac{f-s_i(f)}{1-X_i^{\mathrm H}},
 \end{aligned}
\end{equation}
for every $f$ in the relevant coefficient algebra. Here
$\xi_i=\xi_{\alpha_i}$ and $X_i^{\mathrm H}=X_{\alpha_i}^{\mathrm H}$.
The quotients on the right are elements of the coefficient
algebra, by divisibility of reflection differences. The multiplicative mixed relation is the Bernstein relation.
The $t_i$ generate the finite Weyl-group algebra and the $T_i$
its finite Hecke deformation; the affine algebras additionally
contain their stated coefficient rings.

Define
\begin{equation*}\label{kl:eq:hecke-cyclotomic-polynomials}
 p_{\mathrm{add}}(t)=\prod_{a\in\mathbb I}(t-\bar a)^{\Lambda_a},
 \qquad
 p_{\mathrm{mult}}(t)=\prod_{a\in\mathbb I}(t-q^a)^{\Lambda_a}.
\end{equation*}
The corresponding root-coordinate cyclotomic quotients of
these affine Hecke algebras are
\[
 \begin{aligned}
 H_{\mathrm{deg}}(\Lambda)=H_{\mathrm{deg}}/\langle p_{\mathrm{add}}(\xi_1), \quad H_q(\Lambda)
   =H_q/\langle p_{\mathrm{mult}}(X_1^{\mathrm H})\rangle.
 \end{aligned}
\]

Use the generalized-weight idempotents $e(c)$ of the respective
Kong--Li cyclotomic constructions. Their defining property is
nilpotence of $(\xi_\alpha-\overline{m_\alpha(c)})e(c)$ or
$(X_\alpha^{\mathrm H}-q^{m_\alpha(c)})e(c)$ for every root.
The transport relations imply that
$e(\cC)=\sum_{c\in\cC}e(c)$
is central; some of its summands, or the whole sum, may vanish.
We write
\begin{equation}\label{kl:eq:orbit-summands}
 \begin{aligned}
 H_{\mathrm{deg}}(\cC;\Lambda)=H_{\mathrm{deg}}(\Lambda)e(\cC), \quad
 H_q(\cC;\Lambda)=H_q(\Lambda)e(\cC).
 \end{aligned}
\end{equation}
When $\Lambda$ is fixed, we may abbreviate the second algebra
to $H_q(\cC)$. Its identity is $e(\cC)$.

To specify the crossing maps without singular expressions at
zero residue, introduce the following rational functions,
regular and nonzero at $z=0$:
\begin{equation*}\label{kl:eq:additive-normalizing-units}
 U_a(z)=
 \begin{cases}
 1-z,&a=0,\\
 1,&a=1,\ e\ne2,\\
 (z-2)(1-z)^{-2},&a=-1,\ e\ne2,\\
 (1-z)^{-1},&a=1,\ e=2,\\
 1-(z+\bar a)^{-1},&a\notin\{0,1,-1\},
 \end{cases}
\end{equation*}
and
\begin{equation*}\label{kl:eq:multiplicative-normalizing-units}
 V_a(z)=
 \begin{cases}
 1-q-z,&a=0,\\
 1,&a=1,\ e\ne2,\\
 \bigl(q-q^{-1}(1-z)\bigr)
       \bigl(1-q^{-1}(1-z)\bigr)^{-2},&a=-1,\ e\ne2,\\
 \bigl(q(1-z)-1\bigr)^{-1},&a=1,\ e=2,\\
 \bigl(q^a(1-z)-q\bigr)
       \bigl(q^a(1-z)-1\bigr)^{-1},&a\notin\{0,1,-1\}.
 \end{cases}
\end{equation*}
These are the Kong--Li normalizing factors in our root
coordinates; see \cite[Section 2]{KongLiD,KongLiND} and
\cite[Section 3.4]{KongLiBK}. Their inverses are defined on
nilpotent centered coordinates by
Lemma~\ref{kl:lem:cyclotomic-localization}.

\begin{theorem}[Explicit cyclotomic affine Hecke comparison]
\label{rc:cor:kong-li-cyclotomic}
\label{kl:thm:explicit-cyclotomic}
Retain the Kong--Li coefficient and parameter assumptions of
Section~\ref{rc:subsec:kong-li}, and assume
\eqref{kl:eq:path-condition}. There are algebra isomorphisms
\begin{equation}\label{kl:eq:cyclotomic-recovery}
 \begin{aligned}
 \Phi_{\mathrm{add}}\colon \cR_{\mathrm{add}}(\cC;\Lambda)
    &\xrightarrow{\sim}H_{\mathrm{deg}}(\cC;\Lambda),\\
 \Phi_{\mathrm{mult}}\colon \cR_{\mathrm{mult}}(\cC;\Lambda)
    &\xrightarrow{\sim}H_q(\cC;\Lambda).
 \end{aligned}
\end{equation}
They send $e_c$ to $e(c)$ and satisfy
\begin{equation}\label{kl:eq:centered-coordinates}
 \begin{aligned}
 \Phi_{\mathrm{add}}(x_\alpha e_c)
   &=\bigl(\xi_\alpha-\overline{m_\alpha(c)}\bigr)e(c),\\
 \Phi_{\mathrm{mult}}(Y_\alpha e_c)
   &=\bigl(1-q^{-m_\alpha(c)}X_\alpha^{\mathrm H}\bigr)e(c).
 \end{aligned}
\end{equation}
Put $z_i=\sum_c(\xi_i-\overline{m_i(c)})e(c)$ in the
additive target and $Z_i=\sum_c(1-q^{-m_i(c)}X_i^{\mathrm H})e(c)$
in the multiplicative target. For $a=m_i(c)$, the crossing maps are
\begin{equation}\label{kl:eq:explicit-additive-crossing}
 \Phi_{\mathrm{add}}(\psi_i e_c)=
 \begin{cases}
 (t_i+1)U_0(z_i)^{-1}e(c),&a=0,\\
 (t_i+\xi_i^{-1})U_a(z_i)^{-1}e(c),&a\ne0,
 \end{cases}
\end{equation}
and
\begin{equation}\label{kl:eq:explicit-multiplicative-crossing}
 \Phi_{\mathrm{mult}}(\psi_i e_c)=
 \begin{cases}
 (T_i+1)V_0(Z_i)^{-1}e(c),&a=0,\\
 \left(T_i+\dfrac{q-1}{X_i^{\mathrm H}-1}\right)
 V_a(Z_i)^{-1}e(c),&a\ne0.
 \end{cases}
\end{equation}
All inverses in these formulas are taken in the commutative
coefficient corner with identity $e(c)$; the order of the
factors is part of the formulas. The maps identify central
orbit summands, which need not be single blocks.
\end{theorem}

\begin{proof}
The isomorphisms are the Kong--Li cyclotomic comparisons
\cite[Corollary 4.14]{KongLiD,KongLiND}. Their normalized
intertwiners give the displayed crossing maps. On each residue
corner the Hecke cyclotomic polynomial is the centered root
power times a unit. Lemma~\ref{kl:lem:cyclotomic-localization}
removes the localization. The full inverse maps, the unit
factorizations, and the orbit-summand identification are given
in Appendix~\ref{app:proof-root-comparison}.
\end{proof}

\begin{remark}[Scope of the comparison]
\label{kl:rem:comparison-scope}
The theorem recovers the two Kong--Li isomorphisms, with their
type-$A$ residue quiver, parameter data, and path condition.
The Weyl group need not have type $A$. It is not an isomorphism
claim for arbitrary crossing multiplicities $d$ or arbitrary
admissible residue data of Definition~\ref{rc:def:residues}.
Corollaries~\ref{qv:cor:hecke-maps} and~\ref{qv:cor:zero-hecke}
give comparisons for specified changes of quiver; an unrestricted
comparison to the same Hecke target is excluded by
Proposition~\ref{qv:prop:no-hecke-map}.
Nor does a fraction-field isomorphism suffice
for cyclotomic descent: the denominators used above remain
units after imposing the cyclotomic relations.

The distinction between original and centered coordinates
persists in the cyclotomic quotient. In particular,
$X_\alpha=1-Y_\alpha$ on the quiver Demazure side maps to
$q^{-m_\alpha(c)}X_\alpha^{\mathrm H}$ on the $c$-summand.
On the zero orbit this factor is $1$, so the notation
$X_\alpha$ can be used for both characters in the later
zero-orbit sections.
\end{remark}

\subsection{The type-\texorpdfstring{$A$}{A} strand cyclotomic comparison}
\label{kl:subsec:strand-cyclotomic}

For the usual cyclotomic quotients of the degenerate and
non-degenerate affine Hecke algebras with finite Weyl group
$\mathfrak S_n$, use the strand realization of Section~\ref{kl:subsec:strand-operators}:
\[
 W=\mathfrak S_n,\qquad P_{\mathrm{str}}=\kk[z_1,\ldots,z_n],
 \qquad y_{\epsilon_r-\epsilon_s}=z_r-z_s.
\]
Let $\cC\subseteq\mathbb I^n$ be an orbit of residue
sequences. Keep $d(a)=\delta_{a,-1}$ and denote the resulting
operator algebra by $\cR_{\mathrm{str}}(\cC)$. Define
\begin{equation}\label{kl:eq:strand-cyclotomic-definition}
 \cR_{\mathrm{str}}(\cC;\Lambda)=
 \cR_{\mathrm{str}}(\cC)/
 \left\langle z_1^{\Lambda_{i_1}}e_{\mathbf i}:
                    \mathbf i\in\cC\right\rangle.
\end{equation}
The cyclotomic datum here uses the first residue $i_1$,
rather than the root residue $i_1-i_2$.

\begin{corollary}[Type-$A$ strand recovery]
\label{kl:cor:strand-recovery}
With the residue parameter chosen for each case as in
Section~\ref{rc:subsec:kong-li}, there are algebra isomorphisms
\begin{equation}\label{kl:eq:strand-recovery}
 \begin{aligned}
 \cR_{\mathrm{str}}(\cC;\Lambda)
   &\xrightarrow{\sim}H_{n,\mathrm{deg}}^\Lambda e(\cC),\\
 \cR_{\mathrm{str}}(\cC;\Lambda)
   &\xrightarrow{\sim}H_{n,q}^\Lambda e(\cC).
 \end{aligned}
\end{equation}
Here $H_{n,\mathrm{deg}}^\Lambda$ and $H_{n,q}^\Lambda$ are
the usual strand-coordinate cyclotomic affine Hecke algebras, with
cyclotomic polynomial $p_{\mathrm{add}}(\xi_1)$ or
$p_{\mathrm{mult}}(X_1^{\mathrm H})$ in the first strand
coordinate. Both maps send $e_{\mathbf i}$ to $e(\mathbf i)$,
and their coefficient maps are
\begin{equation}\label{kl:eq:strand-coefficient-maps}
 z_r e_{\mathbf i}\longmapsto
 \begin{cases}
 (\xi_r-\bar i_r)e(\mathbf i),&\text{degenerate case},\\
 (1-q^{-i_r}X_r^{\mathrm H})e(\mathbf i),
       &\text{non-degenerate case}.
 \end{cases}
\end{equation}
Here $W=\mathfrak S_n$ has type $A_{n-1}$ and the residue
quiver is also of type $A$. In the standard strand block
parametrization, every nonzero
target orbit summand is a block. These are the
Brundan--Kleshchev--Rouquier isomorphisms, expressed using
Kong--Li's presentations.
\end{corollary}

\begin{proof}
The strand polynomial representation is the ordinary quiver
Hecke representation. The centered coefficient maps and an
explicit crossing sign change identify it with the presentations
of~\cite{KongLiTypeA}. Nilpotence makes all localization factors
units, and the first-strand cyclotomic polynomial differs from
the centered power by a unit. Appendix~\ref{app:proof-strand-comparison}
records the crossing maps, the sign conversion, and the complete
descent argument. The resulting isomorphisms are those of
\cite[Theorem 5.6]{KongLiTypeA}; their block interpretation
uses~\cite{BK}.
\end{proof}

\begin{remark}[Root powers and strand powers]
\label{kl:rem:root-versus-strand}
Theorem~\ref{kl:thm:explicit-cyclotomic} and
Corollary~\ref{kl:cor:strand-recovery} use different
coefficient data and different cyclotomic relations.
In the additive strand realization,
$x_{\alpha_1}=z_1-z_2$, whereas the usual strand cyclotomic
relation involves $z_1$. Moreover, restriction of a sequence
$\mathbf i$ to the root lattice remembers differences
$i_r-i_s$ and can forget a common shift of all residues.
Thus the root-coordinate quotient cannot be substituted for
the strand quotient merely by taking $\Phi$ of type $A$.

The cyclotomic quotients of both degenerate and non-degenerate
strand affine Hecke algebras are compared with the additive
strand operator quotient in
\eqref{kl:eq:strand-recovery}. Its parameters are chosen
separately in the two cases; this does not identify unrelated
residue groups. The two root-coordinate specializations in
Theorem~\ref{kl:thm:explicit-cyclotomic} instead retain their
respective additive and multiplicative coordinate laws.
The recovered Hecke isomorphism theorems remain attributed to
Brundan--Kleshchev--Rouquier and Kong--Li.
\end{remark}

\subsection{Changing the residue quiver}
\label{qv:subsec:changing-quiver}

The type of the residue quiver must be distinguished from the type
of the finite Weyl group. In particular, varying the former here is
not the same operation as passing from affine Hecke type $A$ to
types $B$ and $D$ in~\cite{VVTypeB,SVVTypeD,PWTypeB,PWTypeD}.
The involution and framing in the orientifold constructions
of~\cite{PrzezdzieckiOrientifold,AppelPrzezdziecki} are additional
structure. The assertions below concern the stated QDA data and
their comparison maps.

Return to the root-coordinate construction of
Definition~\ref{rc:def:root-algebra}. Fix $W$, the coefficient
domain $R$, the coordinates $y_\alpha$, the residue group
$\mathbb I$, and the admissible orbit $\cC$. We now vary only
the arrow multiplicity $d$. Write $\cR_d$ and
$\psi_i^{(d)}$ for the resulting algebra and its crossings.
Fix a distinguished simple root $\alpha_1$ and finitely
supported $\Lambda:\mathbb I\to\mathbb Z_{\geq0}$, and set
\begin{equation*}\label{qv:eq:cyclotomic-ideal}
 \begin{split}
 \mathfrak a_d^\Lambda
   &=\left\langle y_1^{\Lambda_{m_1(c)}}e_c:c\in\cC
     \right\rangle_{\cR_d}\ \ \text{and} \ \ 
 \cR_d(\cC;\Lambda)=\cR_d/\mathfrak a_d^\Lambda.
 \end{split}
\end{equation*}
The subscript on the ideal records its ambient algebra.
No PBW or nilpotence hypothesis is imposed in the next proposition.

\begin{proposition}[Comparison under arrow containment]
\label{qv:prop:arrow-containment}
If $d(a)\ge d'(a)$ for every $a\ne0$, the identity in
$D_K\rtimes W$ restricts to an inclusion $
 \iota_{d,d'}\colon \cR_d\hookrightarrow\cR_{d'}$.
It induces a unital algebra homomorphism
\begin{equation}\label{qv:eq:cyclotomic-map}
 \overline\iota_{d,d'}\colon 
 \cR_d(\cC;\Lambda)\longrightarrow\cR_{d'}(\cC;\Lambda)
\end{equation}
with kernel
\begin{equation}\label{qv:eq:comparison-kernel}
 \ker\overline\iota_{d,d'}
 =\bigl(\cR_d\cap\mathfrak a_{d'}^\Lambda\bigr)
      /\mathfrak a_d^\Lambda.
\end{equation}
These maps compose in the expected order when
$d\ge d'\ge d''$.
\end{proposition}

\begin{proof}
All the algebras have the same coefficient subalgebra $D_R$.
For $a=m_i(c)\ne0$, formula~\eqref{rc:eq:psi} gives
\begin{equation}\label{qv:eq:weighted-inclusion}
 \begin{split}
 \psi_i^{(d)}e_c
 &=y_{-\alpha_i}^{d(a)-d'(a)}\psi_i^{(d')}e_c\\
 &=\psi_i^{(d')}y_i^{d(a)-d'(a)}e_c.
 \end{split}
\end{equation}
The second equality uses $[s_i]y_i=y_{-\alpha_i}[s_i]$.
If $a=0$, the two crossings are the same divided difference.
Thus every generator of $\cR_d$ belongs to $\cR_{d'}$,
which proves the inclusion. It fixes each generator of the
cyclotomic ideal, so
$\mathfrak a_d^\Lambda\subseteq\mathfrak a_{d'}^\Lambda$.
This gives~\eqref{qv:eq:cyclotomic-map} and its stated kernel.
All maps come from the same ambient identity, which also
proves the composition assertion.
\end{proof}

\begin{corollary}[Hecke homomorphisms for larger residue quivers]
\label{qv:cor:hecke-maps}
Assume the coefficient, residue-orbit, and directed-path
hypotheses of Theorem~\ref{kl:thm:explicit-cyclotomic}.
If $d\ge d_A$, there are unital algebra homomorphisms
\begin{equation*}\label{qv:eq:hecke-maps}
 \begin{aligned}
 \Phi_{\mathrm{add},d}\colon \cR_{\mathrm{add},d}(\cC;\Lambda)
   &\longrightarrow H_{\mathrm{deg}}(\cC;\Lambda),\\
 \Phi_{\mathrm{mult},d}\colon \cR_{\mathrm{mult},d}(\cC;\Lambda)
   &\longrightarrow H_q(\cC;\Lambda).
 \end{aligned}
\end{equation*}
They are nonzero whenever their target is nonzero. They are
the isomorphisms of Theorem~\ref{kl:thm:explicit-cyclotomic}
when $d=d_A$.

For either specialization, let
$\Psi_{i,\star}=\Phi_\star(\psi_i^{(d_A)})$ and
$z_{\alpha,\star}=\Phi_\star(y_\alpha)$ in the Hecke target.
Set $r(0)=0$ and $r(a)=d(a)-d_A(a)$ for $a\ne0$.
Then the maps are given by
\begin{equation*}\label{qv:eq:explicit-hecke-map}
 \begin{aligned}
 \Phi_{\star,d}(e_c)&=e(c),\\
 \Phi_{\star,d}(y_\alpha e_c)&=z_{\alpha,\star}e(c),\\
 \Phi_{\star,d}(\psi_i^{(d)}e_c)
   &=\Psi_{i,\star}
       z_{\alpha_i,\star}^{\,r(m_i(c))}e(c).
 \end{aligned}
\end{equation*}
The coefficient maps are those of~\eqref{kl:eq:centered-coordinates},
and the power in the last line is on the right of the crossing.
\end{corollary}

\begin{proof}
Compose~\eqref{qv:eq:cyclotomic-map}, with $d'=d_A$, with
the relevant isomorphism in~\eqref{kl:eq:cyclotomic-recovery}.
Equation~\eqref{qv:eq:weighted-inclusion} gives the crossing
formula; all coefficient generators and idempotents have the
stated images. The identity maps to $e(\cC)$, so a nonzero
target gives a nonzero map.
\end{proof}

\begin{example}[A residue quiver with additional arrows]
\label{qv:ex:enlarged-quiver}
Let $\mathbb I=\mathbb Z$ and
$d(a)=\delta_{a,-1}+\delta_{a,-2}$. The arrows of $\Gamma_d$
are $a\to a+1$ and $a\to a+2$. Its underlying graph is
connected and every vertex has degree four, so it is not a
type-$A$ line or cycle. Corollary~\ref{qv:cor:hecke-maps}
applies over $\mathbb C$ in the additive case and, with $q$
of infinite order, in the multiplicative case. The new crossing
image is $\Psi_{i,\star}z_{\alpha_i,\star}e(c)$ when
$m_i(c)=-2$; all other crossing images are unchanged.

The induced map can fail to be both injective and surjective.
To see this, take $W$ of type $A_1$, use additive coordinates,
put $\cC=\{2,-2\}$, and set
$\Lambda_2=\Lambda_{-2}=1$, with all other entries zero.
The source has idempotents $e_+,e_-$ and crossings
$u=\psi e_+$, $v=\psi e_-$ with $u\in e_-\cR_d e_+$ and $v\in e_+\cR_d e_-$.
After the cyclotomic quotient, $y=0$ and $uv=vu=0$.
The four elements $e_+,e_-,u,v$ form a basis: before the
quotient the rank-one PBW basis is free over $\mathbb C[y]$,
and $y\cR_d$ is the two-sided ideal generated by $y$.
For $d_A$, both crossings on this orbit are unweighted,
so $\cR_{d_A}(\cC;\Lambda)\cong\Mat_2(\mathbb C)$.
The comparison sends $v$ to zero, and sends $u$ to a nonzero
off-diagonal matrix unit after a choice of matrix units.
Its kernel has dimension one and its image is a triangular
matrix algebra of dimension three.
\end{example}

\begin{remark}[The strand quotient and the choice of target]
\label{qv:rem:target-scope}
The ideal argument in Proposition~\ref{qv:prop:arrow-containment}
uses only that its specified generators belong to $D_R$.
It therefore also applies to the strand powers
$z_1^{\Lambda_{i_1}}e_{\mathbf i}$ of
\eqref{kl:eq:strand-cyclotomic-definition}. Composing with
Corollary~\ref{kl:cor:strand-recovery} gives analogous maps
to the usual strand-coordinate Hecke summands when $d\ge d_A$.

In each case the target is the already specified Hecke algebra,
with its original spectral and cyclotomic data. The construction
does not define a standard Hecke algebra for an arbitrary
uniform quiver. Formula~\eqref{qv:eq:comparison-kernel} also
shows why an inclusion before the quotient need not remain
injective after it.
\end{remark}

\subsection{The zero orbit for arbitrary residue quivers}
\label{qv:subsec:zero-orbit}

\begin{proposition}[Quiver independence at zero residue]
\label{qv:prop:zero-independence}
Fix the coefficient data and take $\cC=\{0\}$ with residue
map identically zero. The algebras $\cR_d(\{0\})$ are the
same operator subalgebra of $K\rtimes W$ for every permitted
arrow function $d$. The identification fixes the coefficients
and the crossings and passes to any quotient by specified
root-coordinate powers, as well as to the corresponding
completions.
\end{proposition}

\begin{proof}
For every $i$, the only crossing formula is
$\psi_i=y_i^{-1}([s_i]-1)$. Neither the generators nor the
ambient coefficient algebra involve $d$. Thus the subalgebras
are equal, and the stated quotients and completions agree.
At this orbit the construction also forgets the residue group.
\end{proof}

For $\star\in\{\mathrm{add},\mathrm{mult}\}$ and $m\ge1$, put
\[
 \cR_{\star,d}^{(m)}(0)
 =\cR_{\star,d}(\{0\})/\langle y_1^m\rangle.
\]
Define the standard cyclotomic affine Hecke quotients
\[
 H_{\mathrm{deg}}^{(m)}
   =H_{\mathrm{deg}}/\langle\xi_1^m\rangle,
 \qquad
 H_q^{(m)}=H_q/\langle(X_1^{\mathrm H}-1)^m\rangle.
\]
In each target, $e(0)$ denotes the all-zero residue idempotent;
the original multiplicative root characters have generalized
eigenvalue $1$ there.

\begin{corollary}[Zero-orbit Hecke comparison for every uniform quiver]
\label{qv:cor:zero-hecke}
Use the additive or multiplicative coefficient data of
Section~\ref{rc:subsec:kong-li}, with its field and parameter
hypotheses, and assume~\eqref{kl:eq:path-condition}.
For every permitted uniform residue quiver and every $m\ge1$,
there are algebra isomorphisms
\begin{equation}\label{qv:eq:zero-hecke}
 \begin{aligned}
 \cR_{\mathrm{add},d}^{(m)}(0)
   &\cong H_{\mathrm{deg}}^{(m)}e(0),\\
 \cR_{\mathrm{mult},d}^{(m)}(0)
   &\cong H_q^{(m)}e(0).
 \end{aligned}
\end{equation}
The source residue group may be arbitrary. The cyclic residue
group for each Hecke target is chosen as in
Section~\ref{rc:subsec:kong-li}.
\end{corollary}

\begin{proof}
By Proposition~\ref{qv:prop:zero-independence}, the source
agrees with the zero-orbit algebra for $\Gamma_A$ and the
same coordinate law. Apply
Theorem~\ref{kl:thm:explicit-cyclotomic} with
$\Lambda=m\delta_0$. Explicitly, the additive map sends
\[
 x_\alpha\longmapsto\xi_\alpha e(0),\qquad
 \psi_i\longmapsto(t_i+1)(1-\xi_i)^{-1}e(0),
\]
and the multiplicative map sends
\[
 Y_\alpha\longmapsto(1-X_\alpha^{\mathrm H})e(0),\qquad
 \psi_i\longmapsto(T_i+1)(X_i^{\mathrm H}-q)^{-1}e(0).
\]
The indicated inverses are on the right. They exist because
the centered coordinates are nilpotent, while the constant
terms of the denominators are $1$ and $1-q$, respectively.
The inverse maps are the zero-residue cases of
\eqref{kl:eq:inverse-finite-hecke}. Thus the maps are independent
of the arrows of $\Gamma_d$.
\end{proof}

The zero-orbit operator realization and its completion in
Section~\ref{sec:zeroorbit} therefore apply to every permitted
residue quiver. This intrinsic quiver independence does not require
the Hecke parameter or path hypotheses; those hypotheses are used
for the Hecke identifications in~\eqref{qv:eq:zero-hecke}.

\subsection{A rank-one obstruction with the standard spectral data}
\label{qv:subsec:obstruction}

Arrow containment is a sufficient condition for the comparison
maps above. The following example rules out a statement for
all uniform quivers with the standard Hecke target and unchanged
spectral data. It makes no claim about different target algebras.

\begin{proposition}[Failure of an unrestricted Hecke comparison]
\label{qv:prop:no-hecke-map}
Over $\mathbb C$, take $W=\langle s\rangle$ of type $A_1$,
$\mathbb I=\mathbb Z$, and
\[
 d(a)=\delta_{a,-2}+\delta_{a,-3}.
\]
Thus $\Gamma_d$ has arrows $a\to a+2$ and $a\to a+3$.
Take the root-residue orbit $\cC=\{1,-1\}$ and cyclotomic
data $\Lambda_1=\Lambda_{-1}=1$, zero elsewhere.
In the additive case, and in the multiplicative case with
$q=2$, one has
\begin{equation}\label{qv:eq:obstruction-source}
 \cR_{\star,d}(\cC;\Lambda)\cong\Mat_2(\mathbb C).
\end{equation}
There is no unital algebra homomorphism from this source to
the corresponding standard Hecke quotient
$H_{\mathrm{deg}}(\cC;\Lambda)$ or $H_2(\cC;\Lambda)$.
\end{proposition}

\begin{proof}
The underlying graph of $\Gamma_d$ is connected, since moves
by $2$ and $3$ generate $\mathbb Z$, and every vertex has
degree four. It is therefore not a type-$A$ line or cycle.
The values entering the two rank-one crossings are
$d(1)=d(-1)=0$. Hence $\psi=[s]$, with
\[
 [s]^2=1,\qquad [s]e_+=e_-[s],\qquad [s]e_-=e_+[s].
\]
The cyclotomic relations set the centered root coordinate
$y$ equal to zero on both idempotents. Its coefficient ideal
is $s$-stable: $s(y)=-y$ in the additive case and
$s(y)=-y/(1-y)$ in the multiplicative case. The quotient is
therefore
\[
 (\mathbb Ce_+\oplus\mathbb Ce_-)\rtimes\langle s\rangle
 \cong\Mat_2(\mathbb C),
\]
where $e_+,e_-$ correspond to the diagonal matrix units and
$[s]$ to the matrix exchanging them. This proves
\eqref{qv:eq:obstruction-source} in both cases.

In the additive target, the original root coordinate $\xi$
and Hecke generator $t$ satisfy
\begin{equation*}\label{qv:eq:degenerate-obstruction-target}
 H_{\mathrm{deg}}(\Lambda)
 =\mathbb C\langle\xi,t\rangle/
    (\xi^2-1,\ t^2-1,\ t\xi+\xi t+2).
\end{equation*}
The mixed relation follows from
\eqref{kl:eq:hecke-coordinate-conventions}. Since the only
possible eigenvalues of $\xi$ are $1,-1$, the orbit idempotent
$e(\cC)$ is the identity. This algebra has the character
\[
 \xi\longmapsto1,\qquad t\longmapsto-1.
\]

For the multiplicative target write $X=X_\alpha^{\mathrm H}$.
The defining rank-one relations and cyclotomic relation are
\begin{equation*}\label{qv:eq:multiplicative-obstruction-target}
 \begin{gathered}
 (T+1)(T-q)=0,\qquad
 TX-X^{-1}T=-(q-1)(1+X^{-1}),\\
 (X-q)(X-q^{-1})=0,
 \qquad q=2.
 \end{gathered}
\end{equation*}
Here $e(\cC)=1$ as well. The assignment
$X\mapsto q$, $T\mapsto-1$ is a character: the two sides
of the mixed relation both evaluate to $-q+q^{-1}$.

In either case, composing a unital map from
$\Mat_2(\mathbb C)$ with the displayed character would give
a unital map $\Mat_2(\mathbb C)\to\mathbb C$. Such a map
would be injective because $\Mat_2(\mathbb C)$ is simple,
which is impossible by dimension. This proves both assertions.
\end{proof}

\subsection{Orbit summands and block refinement}
\label{kl:subsec:orbit-summands}

\begin{remark}[Block refinement]\label{kl:prop:block-refinement}
Under either recovered isomorphism
$\phi:A\xrightarrow{\sim}H(\Lambda)e(\cC)$, a decomposition
$e(\cC)=\sum_a b_a$ into primitive central idempotents pulls back to
the block decomposition
\begin{equation}\label{kl:eq:block-refinement}
 A=\bigoplus_a A\phi^{-1}(b_a),\qquad
 A\phi^{-1}(b_a)\cong H(\Lambda)b_a.
\end{equation}
For finite-dimensional nonzero algebras, the orbit summand is one
block precisely when $Z(A)$ is local.  This follows from the
decomposition of a commutative Artinian algebra into local factors.
The module categories decompose accordingly, and all $\Ext^n$
between distinct factors vanish.  These are standard consequences
of an algebra isomorphism, not a primitivity assertion about
$e(\cC)$.
\end{remark}

\begin{example}[A singleton orbit containing two blocks]
\label{kl:ex:two-block-orbit}
In type $A_1$ over $\mathbb C$, let $q$ have order four and take
$\mathbb I=\mathbb Z/4\mathbb Z$, $\cC=\{2\}$,
$\Lambda_2=1$, and $\Lambda_j=0$ for $j\ne2$.
The reflection fixes this nonzero residue.  For
$P_{\mathrm{mult}}=\mathbb C[z,z^{-1}]$, $s(z)=z^{-1}$,
and $Y=1-z$, the standard quiver has $d(2)=0$, so
\[
 \cR_{\mathrm{mult}}(\cC)=\mathbb C[z,z^{-1}]\rtimes\langle s\rangle.
\]
The coefficient ideal $(Y)$ is stable under $s$, and its quotient is
\begin{equation}\label{kl:eq:two-block-example}
 \cR_{\mathrm{mult}}(\cC;\Lambda)
 \cong\mathbb C[\psi]/(\psi^2-1)
 \cong\mathbb C\times\mathbb C.
\end{equation}
The primitive central idempotents are $(1\pm\psi)/2$.
On the Hecke side the cyclotomic relation is
$X=X_\alpha^{\mathrm H}=q^2=-1$.  Since $s(X)=X^{-1}$,
the Bernstein correction is zero at this point, and
\[
 H_q(\cC;\Lambda)=\mathbb C[T]/((T+1)(T-q)).
\]
Its primitive central idempotents are
$(T+1)/(q+1)$ and $(q-T)/(q+1)$.
Before the quotient the center is
$\mathbb C[z+z^{-1}]$; its image is only $\mathbb C\cdot1$.
Thus a singleton orbit summand can split, and the center can acquire
new idempotents after a cyclotomic quotient.
\end{example}

The formal exponential and its descent to nilpotent quotients
are proved in Lemma~\ref{des:lem:central-completion} and
Corollary~\ref{des:cor:cyclotomic-exp}.
In the Hecke comparison the original character on the $c$-summand
is $X_\alpha^{\mathrm H}=q^{m_\alpha(c)}X_\alpha$.
Thus the centered identity $X_\alpha=1$ represents the original
weight $q^{m_\alpha(c)}$, which need not be the identity weight.

\subsection{Generalized weights and normalized intertwiners}
\label{geo:subsec:intertwiners}

The centered coordinates in~\eqref{kl:eq:centered-coordinates}
give a representation-theoretic meaning to the recovery maps.
In the additive case a residue describes a point of the Cartan
space, and $x_\alpha$ measures displacement from its root value.
In the multiplicative case it describes the torus character
$t_c(\alpha)=q^{m_\alpha(c)}$, and $Y_\alpha$ measures
displacement from that character. The idempotent $e(c)$ projects
onto the corresponding generalized weight space.

For clarity, keep the original multiplicative Hecke coordinate
$X_i^{\mathrm H}$ separate from the centered coordinate. The
Bernstein relation and the associated intertwiner are
\[
 T_i f-s_i(f)T_i
   =(q-1)\frac{f-s_i(f)}{1-X_i^{\mathrm H}},
 \qquad
 \Phi_i=T_i+\frac{q-1}{X_i^{\mathrm H}-1}.
\]
In the localized algebra they satisfy
\begin{equation*}\label{geo:eq:intertwiner}
 \Phi_i f=s_i(f)\Phi_i,\qquad
 \Phi_i^2=
 \frac{(1-qX_i^{\mathrm H})(q-X_i^{\mathrm H})}
      {(1-X_i^{\mathrm H})^2};
\end{equation*}
see~\cite[Proposition 2.9]{KongLiND}. The denominator of the
localized intertwiner singles out the root hypertorus
$X_i^{\mathrm H}=1$. Since $q\ne1$, the numerator of
$\Phi_i^2$ equals $-(q-1)^2\ne0$ there; hence this is a
genuine double polar divisor of $\Phi_i^2$. For example, in type $A_1$ over
$\mathbb C$ with $Q=\mathbb Z\alpha_i$, the reflection takes
$X_i^{\mathrm H}$ to $(X_i^{\mathrm H})^{-1}$ and fixes both
$X_i^{\mathrm H}=1$ and $X_i^{\mathrm H}=-1$; at the latter
point the intertwiner denominator is nonzero, as in
Example~\ref{kl:ex:two-block-orbit}. The Bernstein correction
term itself remains regular at $X_i^{\mathrm H}=1$, because
$f-s_i(f)$ is divisible by $1-X_i^{\mathrm H}$.
The numerator of $\Phi_i^2$ vanishes at
$X_i^{\mathrm H}=q$ or $q^{-1}$, with multiplicity when
these values coincide. On the weight indexed by $c$, the
substitution
$X_i^{\mathrm H}=q^{m_i(c)}(1-Y_i)$ expresses these factors
in the centered coordinate. A factor with nonzero constant
term is invertible on a nilpotent generalized weight space.

The Kong--Li normalization uses these units and the remaining
root factors to produce the crossing table
\eqref{kl:eq:quadratic-table}. At zero root residue the
crossing is a divided difference; at nonzero root residue
it is a weighted reflection. The polynomial PBW theorems
ensure that these localized formulas preserve the required
coefficient subalgebras, and the cyclotomic comparison then
imposes the nilpotence relations. This explains the role of
generalized weights and Weyl intertwiners in the isomorphisms.
The operator description alone does not replace those
integrality and cyclotomic arguments.

\section{Details of the recovered cyclotomic comparisons}
\label{app:hecke-comparison-proofs}

We collect the inverse maps and localization arguments for
the comparisons stated in Section~\ref{rc:sec:specializations}.
These recover the cited Hecke isomorphisms and keep track of
the centered coordinates, crossing signs, and cyclotomic
unit factors. The zero-residue local results in
Section~\ref{sec:zeroorbit} use the operator algebra directly.

\subsection{Root-coordinate comparison}\label{app:proof-root-comparison}

\begin{proof}[Proof of Theorem~\ref{kl:thm:explicit-cyclotomic}]
The isomorphism statements are the cyclotomic comparisons of
\cite[Corollary 4.14]{KongLiD,KongLiND}. We explain the maps
and their descent in a common proof, using the localized
presentations from those papers.

First work in $\cR_\star^\circ$. By
Theorem~\ref{rc:thm:kong-li-recovery} and the localized affine
Hecke presentations in \cite[Section 4]{KongLiD,KongLiND}, the
centered coefficient maps identify this algebra with the
modified affine Hecke algebra whose coefficients are localized at
the weights in $\cC$. The crossing maps are
\eqref{kl:eq:explicit-additive-crossing} and
\eqref{kl:eq:explicit-multiplicative-crossing}. At nonzero
residue they are the normalized intertwiners. At zero residue
the apparent pole cancels before passing to a quotient. For
example, in the additive localization of the fraction algebra,
\[
 (t_i+z_i^{-1})(1-z_i)^{-1}-z_i^{-1}
   =(t_i+1)(1-z_i)^{-1}.
\]
In the multiplicative case the same calculation uses
$V_0(Z_i)=1-q-Z_i$ and the intertwiner
$T_i+(1-q)Z_i^{-1}$. Thus the displayed zero-residue maps
only invert units at the centered origin.

For the generators corresponding to the simple reflections
of $W$, the inverse formulas in $\cR_\star^\circ$ are
\begin{equation}\label{kl:eq:inverse-finite-hecke}
 \begin{aligned}
 t_i e_c&\longmapsto
 \begin{cases}
 (\psi_iU_0(x_i)-1)e_c,&a=0,\\
 \bigl(\psi_iU_a(x_i)-(x_i+\bar a)^{-1}\bigr)e_c,&a\ne0,
 \end{cases}\\[1mm]
 T_i e_c&\longmapsto
 \begin{cases}
 (\psi_iV_0(Y_i)-1)e_c,&a=0,\\
 \left(\psi_iV_a(Y_i)
       -\dfrac{q-1}{q^a(1-Y_i)-1}\right)e_c,&a\ne0.
 \end{cases}
 \end{aligned}
\end{equation}
The coefficient inverses are
$\xi_\alpha e_c\mapsto
(x_\alpha+\overline{m_\alpha(c)})e_c$ and
$X_\alpha^{\mathrm H}e_c\mapsto
q^{m_\alpha(c)}(1-Y_\alpha)e_c$.
The localized presentations ensure that these formulas respect
all relations; they are inverse on the displayed generators.

Next compare the cyclotomic ideals. On the summand indexed
by $c$, put $a=m_1(c)$. In the localized presentation,
\begin{equation}\label{kl:eq:cyclotomic-unit-factorization}
 \begin{aligned}
 p_{\mathrm{add}}(\xi_1)e_c
  &=x_1^{\Lambda_a}
    \prod_{b\ne a}(x_1+\bar a-\bar b)^{\Lambda_b}e_c,\\
 p_{\mathrm{mult}}(X_1^{\mathrm H})e_c
  &=(-q^a)^{\Lambda_a}Y_1^{\Lambda_a}
    \prod_{b\ne a}\bigl(q^a(1-Y_1)-q^b\bigr)^{\Lambda_b}e_c.
 \end{aligned}
\end{equation}
All factors other than the indicated power of the centered
coordinate are units at zero. This includes the case
$\Lambda_a=0$, when the cyclotomic ideal contains $e_c$.
The ideals generated by the Hecke cyclotomic polynomial and
by the relations in~\eqref{kl:eq:root-cyclotomic-definition}
therefore agree in the localized algebra.

We also recall why this quotient is the required orbit summand
of the cyclotomic affine Hecke algebra. The original affine
Hecke generators give a homomorphism
from $H_{\mathrm{deg}}(\Lambda)$ or $H_q(\Lambda)$ to the
localized quotient. Its image of a generalized-weight
idempotent $e(c')$ is $e_{c'}$ if $c'\in\cC$, and is zero
otherwise. Indeed, for $c\ne c'$ one can choose a root
whose residues differ. A sufficiently high power of the
corresponding shifted Hecke coordinate annihilates $e(c')$,
whereas that shifted coordinate is a unit on the $c$-summand
of the localized quotient. Thus $e_c$ times the image of
$e(c')$ is zero. Summing the orthogonal idempotents proves
the assertion. Conversely, evaluate the localized coefficients
on the generalized-weight spaces of $H(\Lambda)e(\cC)$.
Every denominator is invertible there, and the Hecke transport
relations give a homomorphism back. The two maps are inverse
on the Hecke generators, the coefficients, and the residue
idempotents. This is the inverse-map argument of
\cite[Section 4]{KongLiD,KongLiND}.

Finally, Lemma~\ref{kl:lem:cyclotomic-localization} removes
the localization on the quiver Demazure side. The result is
\eqref{kl:eq:cyclotomic-recovery}, with precisely the maps
stated above.
\end{proof}

\subsection{Strand comparison and orientation}\label{app:proof-strand-comparison}

\begin{proof}[Proof of Corollary~\ref{kl:cor:strand-recovery}]
The polynomial representation in Section~\ref{kl:subsec:strand-operators}
identifies $\cR_{\mathrm{str}}(\cC)$ with the ordinary quiver
Hecke algebra having the quadratic polynomials
\eqref{rc:eq:ordinary-Q}. We record the sign conversion to
\cite{KongLiTypeA} so that the crossing maps are specified
as well as the coefficient maps. For $a\in\mathbb I$ put
\[
 \varepsilon(a)=
 \begin{cases}
 -1,&a=-1,\ e\ne2,\\
 1,&\text{otherwise}.
 \end{cases}
\]
Let $\psi_r^{\mathrm{KL}}e(\mathbf i)$ denote the normalized
Hecke crossing of \cite[Sections 4.2--4.3]{KongLiTypeA} in
the relevant degenerate or non-degenerate algebra. The
crossing map for our orientation is
\begin{equation}\label{kl:eq:strand-crossing-sign}
 \psi_r e_{\mathbf i}\longmapsto
 \varepsilon(i_r-i_{r+1})
 \psi_r^{\mathrm{KL}}e(\mathbf i).
\end{equation}
Indeed, for $a=\pm1$ and $e\ne2$, the product
$\varepsilon(a)\varepsilon(-a)=-1$ changes the quadratic
polynomial in \cite[(4.20)]{KongLiTypeA} to
\eqref{rc:eq:ordinary-Q}. The equal-residue, double-edge,
and nonadjacent cases are unchanged. Substitution in the
mixed and braid relations gives the same presentation; in
particular, the equal-residue divided difference is unchanged.
Thus the coefficient identification and this sign change give
an isomorphism of the polynomial quiver Hecke presentations.

For completeness, the normalized crossings used in
\eqref{kl:eq:strand-crossing-sign} have the form
$(h_r+P_r(\mathbf i))Q_r(\mathbf i)^{-1}e(\mathbf i)$,
where $h_r=t_r$ or $T_r$. The units $Q_r(\mathbf i)$ are
the factors in \cite[Section 4.2]{KongLiTypeA}, evaluated
at the centered strand coordinates in
\eqref{kl:eq:strand-coefficient-maps}. For equal neighboring
residues $P_r(\mathbf i)=1$ in both cases. For distinct
neighboring residues,
\[
 P_r(\mathbf i)=
 \begin{cases}
 (\xi_r-\xi_{r+1})^{-1},&\text{degenerate case},\\
 (q-1)X_{r+1}^{\mathrm H}
       (X_r^{\mathrm H}-X_{r+1}^{\mathrm H})^{-1},
       &\text{non-degenerate case}.
 \end{cases}
\]
These denominators and all $Q_r(\mathbf i)$ are units on
the specified generalized-weight summand.

All strand coordinates $z_r$ are nilpotent in the cyclotomic
quiver Hecke quotient \cite[Lemma 2.1]{BK}. Thus localization
at functions nonzero at the centered origin adds no elements,
by the geometric-series argument in
Lemma~\ref{kl:lem:cyclotomic-localization}. On the
$\mathbf i$-summand the Hecke cyclotomic polynomial factors
as $z_1^{\Lambda_{i_1}}$ times a unit, exactly as in
\eqref{kl:eq:cyclotomic-unit-factorization}, now with
$a=i_1$. The localized presentation and inverse-map argument
of \cite[Theorem 4.24 and Lemma 5.5]{KongLiTypeA} therefore
descend to~\eqref{kl:eq:strand-recovery}. This is
\cite[Theorem 5.6]{KongLiTypeA}, with the orientation
conversion made explicit. The block interpretation is the
standard type-$A$ parametrization in~\cite{BK}.
\end{proof}

\vskip5pt
\noindent\textbf{Declaration of generative AI and AI-assisted
technologies in the manuscript preparation process.}
ChatGPT-5.6 Sol was used during manuscript preparation for assistance
with literature searches and
editorial revision of the exposition. ChatGPT-6 Astra was used during manuscript preparation for checking of proofs.
Responsibility for verifying
the mathematical statements, references, and final content of the
manuscript rests with the author.

\end{document}